\pdfoutput=1
\documentclass[english, reqno]{amsart}
\usepackage{geometry}            
\usepackage{amssymb,amsmath,amsthm,amsfonts,color}
\usepackage{mathrsfs,dsfont, comment,mathscinet}
\usepackage{graphicx}
\usepackage{epstopdf}
\usepackage{mathtools}
\usepackage{babel}
\usepackage{enumerate,esint}

\usepackage{fixltx2e,amsmath}
\MakeRobust{\eqref}

\usepackage{indentfirst}
\usepackage{bm}
\usepackage{picinpar}
\usepackage{lipsum}

\usepackage[toc,page]{appendix}

\usepackage{etoolbox}
\usepackage{authblk}
\usepackage{amsaddr}

\allowdisplaybreaks 

\mathtoolsset{showonlyrefs} 

\makeatletter

\usepackage{fancyhdr}
 
\makeatletter
\patchcmd{\@maketitle}
  {\ifx\@empty\@dedicatory}
  {\ifx\@empty\@date \else {\vskip3ex \centering\footnotesize\@date\par\vskip1ex}\fi
   \ifx\@empty\@dedicatory}
  {}{}
\patchcmd{\@maketitle}
  {\ifx\@empty\@date\else \@footnotetext{\@setdate}\fi}
  {}{}{}
\makeatother

\begingroup
\newtheorem{theorem}{Theorem}[section]
\newtheorem{lemma}[theorem]{Lemma}
\newtheorem{proposition}[theorem]{Proposition}

\newtheorem{define}[theorem]{Definition}
\endgroup

\theoremstyle{definition}
\begingroup

\newtheorem{notation}[theorem]{Notation}

\endgroup

\newcommand\ba[1]{\begin{align}\label{#1}}
\newcommand\ea{\end{align}}
\newcommand\bas{\begin{align*}}
\newcommand\eas{\end{align*}}

\newcommand\ee{\end{equation}}
\newcommand\be{\begin{equation}}
\newcommand\ees{\end{equation*}}
\newcommand\bes{\begin{equation*}}

\newcommand{\e}{\varepsilon}

\newcommand{\Om}{\Omega}
\newcommand{\om}{\omega}
 
\newcommand{\R}{{\mathbb R}}

\newcommand{\rn}{{{\R}^N}}

\newcommand{\wto}{\rightharpoonup}

\newcommand{\wtoh}{\mathrel{\mathop{\to}\limits^{\mathcal H}}}

\newcommand{\into}{{\int_{\Omega}}}

\newcommand{\WW}{{\mathcal W}}

\newcommand{\HH}{{\mathcal H}}

\newcommand\norm[1]{\left\|#1\right\|}

\newcommand{\abs}[1]{\left\lvert#1\right\rvert} 
\newcommand{\fsp}[1]{\left(#1\right)} 
\newcommand{\fmp}[1]{\left[#1\right]}
\newcommand{\flp}[1]{\left\{#1\right\}}  
\newcommand{\fjp}[1]{\left<#1\right>}

\newcommand{\vp}{\varphi}

\newcommand{\limn}{\lim_{n\rightarrow\infty}}

\newcommand{\lime}{\lim_{\e\rightarrow0}}

\newcommand{\sumn}{{\sum_{n=0}^\infty}}

\newcommand{\seqn}[1]{\left\{#1\right\}_{n=1}^\infty}  
\newcommand{\seqk}[1]{\left\{#1\right\}_{k=1}^\infty}  
\newcommand{\seqe}[1]{\left\{#1\right\}_{\e>0}}

\definecolor{CMUred}{RGB}{153,0,0}
\definecolor{CMUgreen}{RGB}{0,135,81}
\definecolor{CMUblue}{RGB}{0,51,127}

\newcommand{\argmin}{{\operatorname{arg\,min}}}

\newcommand{\hnmo}{{\mathcal H^{N-1}}}

\newcommand{\egamma}{{\e}}

\newcommand{\dist}{{\operatorname{dist}}}

\newcommand{\epsn}{{\e(n(m))}}

\def\essinf{\mathop{\rm ess\, inf}}
\def\esssup{\mathop{\rm ess\, sup}}
\def\argmin{\mathop{\rm arg\, min}}

\DeclareRobustCommand{\om}{\omega}
\DeclareRobustCommand{\Om}{\Omega}

\numberwithin{equation}{section}

\newcommand{\normmm}[1]{{\left\vert\kern-0.25ex\left\vert\kern-0.25ex\left\vert #1 
    \right\vert\kern-0.25ex\right\vert\kern-0.25ex\right\vert}}

\title{The Weighted Ambrosio - Tortorelli Approximation Scheme}
\author[1]{Irene Fonseca and PAN Liu}
\date{\today}        
 \address[Pan Liu]{Department of Mathematics, Carnegie Mellon University\\5000 Forbes Avenue, Pittsburgh, PA, 15213}

\begin{document}

\begin{abstract}
The Ambrosio-Tortorelli approximation scheme with weighted underlying metric is investigated. It is shown that it $\Gamma$-converges to a Mumford-Shah image segmentation functional depending on the weight $\om\, dx$, where $\om\in SBV(\Om)$, and on its value $\om^-$.
\end{abstract}
\maketitle
\thispagestyle{empty}
\tableofcontents

\section{Introduction and Main Results}\label{intro_sec}

A central problem in image processing is image denoising. Given an image $u_0$, we decompose it as
\bes
u_0=u_g+n
\ees
where $u_g$ represents a noisy-free ground truth picture, while $n$ encodes noise or textures. Examples of models for such noise distributions are Gaussian noise in Magnetic Resonance Tomography, and Poisson noise in radar measurements \cite{bovik2010handbook}. Variational PDE methods have proven to be efficient to remove the noise $n$ from $u_0$. Several successful variational PDEs have been proposed in the literature (see, for example \cite{lions1992denoising, rudin1993segmentation, rudin1994total, rudin1992nonlinear}) and, among these, the Mumford-Shah image segmentation functional
\begin{align}\label{mumfordshahori}
\begin{split}
&G(u,K):=\alpha\int_{\Om\setminus K}{\abs{\nabla u}^2}dx+\alpha\hnmo(K)+\into (u-u_0)^2dx\\
&\text{where } u\in W^{1,2}(\Om\setminus K),\,K\subset\Om\text{ closed in }\Om,
\end{split}
\end{align}
introduced in \cite{mumford1985boundary}, is one of the most successful approaches. By minimizing the functional \eqref{mumfordshahori} one tries to find a ``piecewise smooth'' approximation of $u_0$. The existence of such minimizers can be proved by using compactness and lower semicontinuity theorems in $SBV(\Om)$ (see \cite{ambrosiocompactness,  ambrosio1989variational, ambrosio1990existence,ambrosio2000functions}). Furthermore, regularity results in \cite{dal1992variational, de1989existence} give that minimizers $u$ satisfy 
\bes
u\in C^1(\Om\setminus \overline S_u)\text{ and }\hnmo(\overline S_u\cap\Om\setminus S_u)=0.
\ees
Here, as in what follows, $S_u$ stands for the jump set of $u$.\\\\
The parameter $\alpha>0$ in \eqref{mumfordshahori}, determined by the user, plays an important role. For example, choosing $\alpha>0$ too large will result in over-smoothing and the edges that should have been preserved will be lost, and choosing $\alpha>0$ too small may keep the noise un-removed. The choice of the ``best" parameter $\alpha$ then becomes an interesting task. In \cite{reyes2015structure} the authors proposed a training scheme by using bilevel learning optimization defined in machine learning, which is a semi-supervised learning scheme that optimally adapts itself to the given ``perfect data" (see \cite{chen2013revisiting, chen2014insights, domke2012generic, domke2013learning, tappen2007utilizing,tappen2007learning}). This learning scheme searches $\alpha>0$ such that the recovered image $u_\alpha$, obtained from\eqref{mumfordshahori}, best fits the given clean image $u_g$, measured in terms of the $L^2$-distance. A simplified bilevel learning scheme 
$(\mathcal B)$ from \cite{reyes2015structure} is the following: 
\begin{enumerate}[Level 1.]
\item
\be\label{result_level_bi}
\bar \alpha:=\argmin_{\alpha>0}\into \abs{u_{\alpha}-u_g}^2dx,
\ee
\item
\be
u_{\alpha}:=\argmin_{u\in SBV(\Om)}\flp{\int_{\Om}\alpha\abs{\nabla u}^2dx+\alpha\hnmo(S_u)+\into\abs{u-u_0}^2dx},
\ee
\end{enumerate}
In \cite{reyes2015structure} the authors proved that the above bilevel learning scheme has at least one solution $\bar \alpha\in(0,+\infty]$, and a small modification rules out the possibility of $\bar\alpha=+\infty$.\\\\
The model proposed in \cite{liu2016optimal} is aimed at improving the above learning scheme. It is a bilevel learning scheme which utilizes the scheme $(\mathcal B)$ in each subdomain of $\Om$, and searches for the best combination of different subdomains from which a recovered image $\bar u$, which best fits $u_g$, might be obtained.\\\\
To present the model, we first fix some notation. For $K\in\mathbb N$, $Q_K\subset \rn$ denotes a cube with its faces normal to the orthonormal basis of $\rn$, and with side-length greater than or equal to $1/K$. Define $\mathcal P_K$ to be a collection of finitely many $Q_K$ such that 
\bes
\mathcal P_K:=\flp{Q_K\subset \rn:\,\,Q_K\text{ are mutually disjoint},\,\,\Om\subset\bigcup Q_K},
\ees
and $\mathcal V_K$ denotes the collection of all possible $\mathcal P_K$. For $K=0$ we set $Q_0:=\Om$, hence $\mathcal P_0=\flp{\Om}$.\\\\
A simplified bilevel learning scheme $(\mathcal P)$ in \cite{liu2016optimal} is as follows:
\begin{enumerate}[Level 1.]
\item
\begin{align}
&\bar u:= \argmin_{K\geq 0,\,\mathcal P_K\in \mathcal V_K}\flp{\into\abs{u_g-u_{\mathcal P_K}}^2dx}\label{resu_lev_tri}\\
&\text{where }u_{\mathcal P_K}:=\argmin_{u\in SBV(\Om)}\flp{\into\alpha_{\mathcal P_K}(x)\abs{\nabla u}^2dx+\int_{S_u}\alpha_{\mathcal P_k}(x)d\hnmo+\int_\Om\abs{u-u_0}^2dx}\label{key_weight_jump}
\end{align}
\item
\begin{align}\label{alp_jump_intro}
\begin{split}
&\alpha_{\mathcal P_K}(x):=\alpha_{Q_K}\text{ for }x\in Q_K\in\mathcal P_K,\text{ where } \alpha_{Q_K}:=\argmin_{\alpha>0}\int_{Q_K} \abs{u_{\alpha}-u_g}^2dx,\\
&u_{\alpha}:=\argmin_{u\in SBV(Q_K\cap \Om)}\flp{\int_{Q_K\cap \Om}\alpha\abs{\nabla u}^2dx+\alpha\hnmo(S_u)+\int_{Q_K\cap \Om}\abs{u-u_0}^2dx}
\end{split}
\end{align}
\end{enumerate}
Scheme $(\mathcal P)$ allows us to perform the denoising procedure ``point-wisely", and it is an improvement of \eqref{result_level_bi}. Note that at step $K=0$, \eqref{resu_lev_tri} reduces to \eqref{result_level_bi}. It is well known that the Mumford-Shah model, as well as the ROF model in \cite{rudin1992nonlinear}, leads to undesirable phenomena like the staircasing effect (see \cite{aubert2006mathematical, chan2006total}). However, such staircasing effect is significantly mitigated in \eqref{resu_lev_tri}, according to numerical simulations in \cite{liu2016optimal} (a theoretical validation of such improvement is needed). We remark that the most important step is \eqref{key_weight_jump} for the following reasons:
\begin{enumerate}[1.]
\item
\eqref{key_weight_jump} is the bridge connecting level 1 and level 2;
\item
since $\alpha_{\mathcal P_K}$ is defined by locally optimizing the parameter $\alpha_{Q_K}$, we expect $u_{\mathcal P_K}$ be ``close" to $u_g$ locally in $Q_K$;
\item
the last integrand in \eqref{key_weight_jump} keeps $u_{\mathcal P_K}$ close to $u_0$ globally, hence we may expect $u_{\mathcal P_K}$ to have a good balance between local optimization and global optimization. 
\end{enumerate}
We may view \eqref{key_weight_jump} as a weighted version of \eqref{mumfordshahori} by changing the underlying metric from $dx$ to $\alpha_{\mathcal P_K}dx$. By the construction of $\alpha_{\mathcal P_k}$ in \eqref{alp_jump_intro}, we know it is a piecewise constant function and, since $K>0$ is finite, the discontinuity set of $\alpha_{\mathcal P_K}$ has finite $\hnmo$ measure. However, $\alpha_{\mathcal P_K}$ is only defined $\mathcal L^N$-a.e., and hence the term 
\bes
\int_{S_u}\alpha_{\mathcal P_k}(x)d\hnmo
\ees
might be ill-defined.\\\\
In this paper, we deal with the well-definess of \eqref{key_weight_jump} by modifying $\alpha_{\mathcal P_K}$ accordingly, and by building a sequence of functionals which $\Gamma$-converges to \eqref{key_weight_jump}. To be precise, we adopt the approximation scheme of Ambrosio and Tortorelli in \cite{ambrosio1990approximation} and change the underlying metric properly. In \eqref{mumfordshahori} Ambrosio and Tortorelli considered a sequence of functionals reminiscent of the Cahn-Hilliard approximation, and introduced a family of elliptic functionals
\be\label{AT_intro}
G_{\e}(u,v):=\int_\Omega \alpha\abs{\nabla u}^2v^2dx +\int_{\Omega}\alpha\fmp{\e\abs{\nabla v}^2+\frac{1}{4\e}{(v-1)^2}}dx+\into\fsp{u-u_0}^2dx,
\ee
where $u\in W^{1,2}(\Om)$, $(v-1)\in W_0^{1,2}(\Om)$, and $u_0\in L^2(\Om)$. The additional field $v$ plays the role of controlling variable on the gradient of $u$. In \cite{ambrosio1990approximation} a rigorous argument has been made to show that $G_\e\to G$ in the sense of $\Gamma$-convergence (\cite{attouch1984variational}), where $G$ is defined in \eqref{mumfordshahori}.\\\\
In view of \eqref{alp_jump_intro}, we fix a weight function $\om\in SBV(\Om)$ such that $\om$ is positive and $\hnmo(S_\om)<+\infty$. Our new (weighted version) Mumford-Shah image segmentation functional is defined as
\be\label{ATS_intro_wei}
E_\om(u):=\int_\Omega \abs{\nabla u}^2\om\,dx+ \int_{S_u} {\om^-}\,d\mathcal H^{N-1},
\ee
and the (weighted version) of Ambrosio - Tortorelli functionals are defined as
\be\label{AT_intro_wei}
E_{\om,\e}(u,v):=\int_\Omega \abs{\nabla u}^2v^2\om\,dx+\int_{\Omega}\fmp{\e\abs{\nabla v}^2+\frac{1}{4\e}{(v-1)^2}}\om\,dx.
\ee
It is natural to take $u\in GSBV_\om(\Om)$ in \eqref{ATS_intro_wei} (see Definition \ref{LpSobBVSBVdef}. For basic definitions and theorems of weighted spaces we refer to \cite{ambrosio2003some, ambrosio2004special,baldi2001weighted,bellettini1999bv,bouchitte1996energies,capogna1994geometric, franchi1997approximation, franchi2003regular}). Moreover, since $K\geq 0$ is finite and $\alpha_{Q_K}>0$ in \eqref{alp_jump_intro} , it is not restricted to assume that 
\be\label{lower_intro}
\text{essinf}\flp{\om(x),x\in\Om}\geq l,\text{ where }l>0\text{ is a constant}.
\ee
This condition implies that all weighted spaces considered in this paper are embedded in the corresponding non-weighted spaces, and hence we may apply some results that hold in the context of non-weighted spaces. For example, $BV_\om\subset BV$ and $W_\om^{1,2}\subset W^{1,2}$ (see Definition \ref{LpSobBVSBVdef}), and most theorems in \cite{ambrosio1990approximation} can be applied to $u\in SBV_\om(\Om)$ (for example, Theorem 2.3 in \cite{ambrosio1990approximation}). \\\\
Before we state our main result, we recall that similar problems have been studied for different types of weight functions $\om$ (see, for example \cite{baldi2003gamma, baldi2005mumford, 2016ArRMA.219.1383L}). In particular, \cite{baldi2003gamma,baldi2005mumford} treat a uniformly continuous and strong $A_\infty$ (defined in \cite{david1989strong}) weight function on Modica-Mortola and Mumford-Shah-type functionals, respectively, and in \cite{2016ArRMA.219.1383L} the authors considered a $C^{1,\beta}$-continuous weight function, with some other minor assumptions, in the one-dimensional Cahn-Hilliard model. \\\\
Our main result is the following:
\begin{theorem}\label{AT_n_intro}
Let $\Om\subset \rn$ be open bounded, let $\om\in SBV(\Om)\cap L^\infty(\Om)$, and let $\mathcal{E}_{\om,\e}$: $L^1_\om(\Om)\times L^1(\Om)\rightarrow[0,+\infty]$ be defined by
\begin{align}\label{e_w_ju_intro}
\mathcal{E}_{\om,\e}(u,v):=
\begin{cases}
E_{\om,\e}(u,v)&\text{if}\,\,(u,v)\in W_\om^{1,2}(\Om)\times W^{1,2}(\Om), \,0\leq v\leq 1,\\
+\infty &\text{otherwise}.
\end{cases}
\end{align}
Then the functionals $\mathcal{E}_{\om,\e}$ $\Gamma$-converge, with respect to the $L^1_\om\times L^1$ topology, to the functional
\begin{align}\label{etd_w_ju_intro}
\mathcal{E}_\om(u,v):=
\begin{cases}
E_\om(u)&\text{if}\,\,u\in GSBV_\om(\Om)\text{ and }v=1\,\,a.e.,\\
+\infty &\text{otherwise}.
\end{cases}
\end{align}
\end{theorem}
The proof of $\Gamma$-convergence consists of two steps. The first step is to prove the ``$\liminf$ inequality" 
\bes
\liminf_{\e\to 0}E_{\om,\e}(u_\e,v_\e)\geq E_\om(u)
\ees
for every sequence $u_\e\to u$, $v_\e\to v$. This is obtained in Section \ref{1_d_jump_intro} in the case $N=1$ by using most of the arguments proposed in \cite{ambrosio1990approximation}, and the properties of $SBV$ functions in one dimension (see Lemma \ref{jump_jiuehuishi}). The case $N>1$ is studied in Section \ref{ATA_nd}, and it uses a special slicing argument (see Lemma \ref{slicing_single}). \\\\
The second step is the construction of a recovery sequence $(u_\e\to u,v_\e\to 1)$ such that the term 
\be\label{jump_rec_intro}
\int_{\Omega}\fmp{\e\abs{\nabla v}^2+\frac{1}{4\e}{(v-1)^2}}\om\,dx
\ee
only captures the information of $\om^-$. We note that for small $\e>0$, \eqref{jump_rec_intro} only penalizes a $\e$-neighborhood around the jump point of $u$. By using fine properties of $SBV$ functions (see Theorem \ref{fine_properties_BV}), we are able to incorporate $u$ and $v$ in our model such that \eqref{jump_rec_intro} will only penalize along the direction $-\nu_{S_\om}$. This will be carried out in Lemma \ref{1d_jump_reco}.\\\\ 
We remark that the techniques we developed in this paper can be adapted to other functional models. For example, 
\begin{enumerate}[1.]
\item
the weighted Cahn-Hilliard model defined as
\bes
CH_{\om,\e}(u):=\int_I\fmp{\e\abs{\nabla u(x)}^2+\frac{1}{\e} W(u)}\om\,dx,
\ees
for $u\in W_\om^{1,2}(\Om)$ and with a double well potential function $W$: $\mathbb R\to [0,+\infty)$ such that $\flp{W=0}=\flp{0,1}$ with the $\Gamma$-limit
\bes
CH_{\om}(u):= c_W P_{\om}(u)
\ees
defined for $u=\chi_E \in BV_\om(\Om)$, where
\bes
c_W:=2\int_0^1\sqrt{W(s)}\,ds\text{ and }P_\om(u):=\int_{S_u}\om^-d\hnmo;
\ees
\item
higher order singular perturbation models defined by the $\Gamma$-limit
\bes
H_\om(u):=\int_\Omega \abs{\nabla u}^2\om\,dx+ \int_{S_u} {\om^-(x)}\,d\mathcal H^{N-1},
\ees
and approximation energies 
\bes
H_{\om,\e}(u,v):=\int_\Omega \abs{\nabla u}^2v^2\om\,dx+\frac{1}{C(k)}\int_{\Omega}\fmp{\e^{2k-1}\abs{\nabla^{\fsp{k}} v}^2+\frac{1}{\e}{(v-1)^2}}\om\,dx,
\ees
where 
\bes
C(k):=\min\flp{ \int_{\R^+} \abs{v^{(k)}}^2+(v-1)^2dx,\,\,v(0)=v'(0)=\cdots=v^{(k-1)}(0)=0,\,\lim_{t\to\infty}v(t)=1}.
\ees
\end{enumerate}
The analysis of items 1 and 2 above is forthcoming (see \cite{liu2016weightedhigher}).\\\\
This article is organized as follows: In Section \ref{Def_Preres} we introduce some definitions and we recall preliminary results. In Section \ref{O_D_Toy_Sec} we prove the one dimensional version of Theorem \ref{AT_n_intro}. Section \ref{T_M_D_C} is devoted to the proof of our main result.
\section{Definitions and Preliminary Results}\label{Def_Preres}
Throughout this paper, $\Omega\subset \rn$ is an open bounded set with Lipschitz boundary, and $I:=(-1,1)$.

\begin{define}
We say that $u\in BV(\Om)$ is a {special function of bounded variation}, and we write $u\in SBV(\Om)$, if the Cantor part of its derivative, $D^cu$, is zero, so that (see \cite{ambrosio2000functions}, $(3.89)$)
\begin{align}\label{SBV_decom}
Du=D^a u+D^j u=\nabla u \mathcal L^N\lfloor \Om +(u^+-u^-)\nu_u \mathcal H^{N-1}\lfloor S_u.
\end{align}
Moreover, we say that $u\in GSBV(\Om)$ if $K\wedge u\vee -K\in SBV(\Om)$ for all $K\in\mathbb N$.
\end{define}
Here we always identify $u\in SBV(\Om)$ with its approximation representative $\bar u$, where 
\bes
\bar u(x):=\frac12\fmp{u^+(x)+u^-(x)},
\ees
with
\bes
u^+(x):=\inf\flp{t\in\R:\,\lim_{r\to 0}\frac{\mathcal L^N(B(x,r)\cap\flp{u>t})}{r^N}=0},
\ees
and
\bes
u^-(x):=\sup\flp{t\in\R:\,\lim_{r\to 0}\frac{\mathcal L^N(B(x,r)\cap\flp{u<t})}{r^N}=0}.
\ees
We note that $\bar u$ is Borel measurable (see \cite{evans2015measure}, Lemma 1, page 210), and it can be shown that $\bar u=u$ $\mathcal L^N$-a.e. $x\in\Om$, and that 
\bes
(\bar u)^+(x)=u^+(x)\text{ and }(\bar u)^-(x)=u^-(x)
\ees
for $\hnmo$-a.e. $x\in\Om$ (see \cite{evans2015measure}, Corollary 1, page 216). Furthermore, we have that 
\be\label{neglesspos}
-<u^-(x)\leq u^+(x)<+\infty
\ee
for $\hnmo$-a.e. $x\in \Om$ (see  \cite{evans2015measure}, Theorem 2, page 211). The inequality \eqref{neglesspos} uniquely determines the sign of $\nu_u$ in \eqref{SBV_decom}.
\begin{define}\label{Muckenhoupt_Function_Space}(The weight function)
We say that $\omega$: $\Om\to (0,+\infty]$ belongs to ${\mathcal W}(\Om)$ if $\om\in L^1(\Om)$ and has a positive lower bound, i.e., there exists $l>0$ such that 
\be\label{pos_lower_ass}
\essinf\flp{\om(x),x\in\Om}\geq l.
\ee
\end{define}
Without loss of generality, we take $l=1$. Moreover, in this paper we will only consider the cases in which $\om$ is either a continuous function or a $SBV$ function. If $\om\in SBV$ then, in addition, we require that
\bes
\hnmo(S_\om)<\infty\text{ and }\hnmo(\overline {S_\om}\setminus S_\om)=0.
\ees
We next fix some notation which will be used throughout this paper. 
\begin{notation}\label{multi_not}
Let $\Gamma\subset \Om$ be a $\hnmo$-rectifiable set and let $x\in \Gamma$ be given.
\begin{enumerate}[1.]
\item
We denote by $\nu_\Gamma(x)$ a normal vector at $x$ with respect to $\Gamma$, and $Q_{\nu_\Gamma}(x,r)$ is the cube centered at $x$ with side length $r$ and two faces normal to $\nu_\Gamma(x)$;
\item
$T_{x,\nu_\Gamma}$ stands for the hyperplane normal to $\nu_\Gamma(x)$ and passing through $x$, and $\mathbb P_{x,\nu_\Gamma}$ stands for the projection operator from $\Gamma$ onto $T_{x,\nu_\Gamma}$;
\item
we define the hyperplane
\bes
T_{x,\nu_\Gamma}(t):=T_{x,\nu_\Gamma}+t\nu_\Gamma(x)
\ees
for $t\in\R$;
\item
we introduce the half-spaces
\bes
H_{\nu_\Gamma}(x)^+:=\flp{y\in\rn:\,\, \nu_\Gamma(x)\cdot(y-x)\geq 0}
\ees
and 
\bes
H_{\nu_\Gamma}(x)^-:=\flp{y\in\rn:\,\, \nu_\Gamma(x)\cdot(y-x)\leq 0}.
\ees
Moreover, we define the half-cubes
\bes
Q_{\nu_\Gamma}^\pm(x,r):=Q_{\nu_\Gamma}(x,r)\cap H_{\nu_\Gamma}(x)^\pm;
\ees
\item
for given $\tau>0$, we denote by $R_{\tau,\nu_\Gamma}(x,r)$ the part of $Q_{\nu_\Gamma}(x,r)$ which lies strictly between the two hyperplanes $T_{x,\nu_\Gamma}(-\tau r)$ and $T_{x,\nu_\Gamma}(\tau r)$;
\item
we set
\be\label{flatten_set}
A_\delta:=\flp{x\in\Om:\,\operatorname{dist}(x,A)<\delta}
\ee
for every $A\subset\Om$ and $\delta>0$.
\end{enumerate}
\end{notation}
\begin{theorem}[\cite{evans2015measure}, Theorem 3, page 213]\label{fine_properties_BV}
Assume that $u\in BV(\Om)$. Then
\begin{enumerate}[1.]
\item
for $\hnmo$-a.e. $x_0\in\Om\setminus S_u$,
\bes
\lim_{r\to 0}\fint_{B(x_0,r)}\abs{u(x)-\bar u(x_0)}^{\frac{N}{N-1}}dx=0;
\ees

\item
for $\hnmo$-a.e. $x_0\in S_u$, 
\bes
\lim_{r\to 0}\fint_{B(x_0,r)\cap H_{\nu_{S_u}}(x_0)^\pm}\abs{u(x)-u^\pm(x_0)}^{\frac{N}{N-1}}dx=0;
\ees
\item
for $\hnmo$ a.e. $x_0\in S_u$
\bes
\lime \frac{1}{\e^{N-1}} \int_{S_u\cap Q_{\nu_{S_u}}(x_0,\e)}\abs{u^+(x)-u^-(x)}d\hnmo(x)=\abs{u^+(x_0)-u^-(x_0)}.
\ees
\end{enumerate}
\end{theorem}

\begin{lemma}\label{jump_jiuehuishi}
Let $\om\in SBV(I)$ be such that $\mathcal H^0(S_\om)<\infty$. For every $x\in I$ the following statements hold:
\begin{enumerate}[1.]
\item
if $\seqn{x_n}$ and $\seqn{y_n}\subset I$ are such that $x_n< x<  y_n$, $n\in \mathbb N$, and $\limn x_n=\limn y_n=x$, then
\be\label{lower_bdd_1dnd}
\liminf_{n\to\infty}\essinf_{y\in(x_n,y_n)} \om(y)\geq \om^-(x);
\ee
\item
\be\label{lower_bdd_2dnd}
\lim_{\substack{z_n\to x\\ \seqn{z_n}\subset H_{\nu_{S_\om}(x)}^\pm}} \bar \om(z_n)=\om^\pm(x);
\ee
\item
\be\label{lower_bdd_3dnd}
\limsup_{d_{\mathcal H}(K_n,x)\to 0}\esssup_{\substack{z\in K_n\\ K_n\subset\subset H_{\nu_{S_\om}(x)}^\pm}} \om(z)=\om^\pm(x),
\ee
where $K_n\subset\subset H_{\nu_{S_\om}(x)}^\pm$ and $d_\mathcal{H}$ denotes the Hausdorff distance (see Definition \ref{hausdorff_distance}).
\end{enumerate}
\end{lemma}
\begin{proof}
If $x\notin S_\om$, then there exists $\delta>0$ such that
\bes
S_\om\cap(x-\delta,x+\delta)=\varnothing,
\ees
and so $\om$ is absolutely continuous in $(x-\delta,x+\delta)$, and \eqref{lower_bdd_1dnd}-\eqref{lower_bdd_3dnd} are trivially satisfied with $\om(x)=\om^-(x)$ and with equality in place of the inequality in \eqref{lower_bdd_1dnd}.\\\\
Let $x\in S_\om$ and, without loss of generality, assume that $x=0$, and let $x_n$, $y_n\to 0$ with $x_n<0<y_n$ for all $n\in\mathbb N$. Since $\mathcal H^0(S_\om)<\infty$, choose $\bar r>0$ such that 
\bes
S_\om\cap (0-\bar r,0+\bar r)=0.
\ees
As $\bar\om$ is absolutely continuous in $(-\bar r,0)$ and $(0,\bar r)$, we may extend $\bar \om$ uniquely to $x=0$ from the left and the right (see Exercise $3.7$ $(1)$ in \cite{leoni2009first}) to define
\be\label{lemma_jump_jiuehuishi}
\bar \om(0^+):=\lim_{x\searrow0^+}\bar \om(x)\text{ and }\bar \om(0^-):=\lim_{x\nearrow0^-}\bar \om(x).
\ee
Assume that (the case $\bar \om(0^-)\geq \bar \om(0^+)$ can be treated similarly)  
\be\label{lemma_jump_jiuehuishi_1}
\bar\om(0^-)\leq\bar \om(0^+).
\ee
We first claim that 
\be\label{jjszmhis_one}
\liminf_{n\to\infty}\inf_{x\in(x_n,y_n)} \bar \om(x)\geq \bar\om(0^-).
\ee
Let $\e>0$ be given. By \eqref{lemma_jump_jiuehuishi} find $\bar r>\delta>0$ small enough such that 
\be\label{lemma_jump_jiuehuishi_2}
\abs{\bar \om(x)-\bar \om(0^-)}\leq \frac12\e\text{ for all }x\in(-\delta,0)\text{,  and }\abs{\bar \om(x)-\bar \om(0^+)}\leq \frac12\e\text{ for all }x\in(0,\delta).
\ee
This, together with \eqref{lemma_jump_jiuehuishi_1}, yields
\bes
\bar \om(x)\geq \bar \om(0^-)-\frac12\e,
\ees
for all $x\in (-\delta,\delta)$. Since $x_n\to 0$ and $y_n\to 0$, we may choose $n$ large enough such that $(x_n,y_n)\subset (-\delta,\delta)$ and hence
\bes
\inf_{x\in(x_n,y_n)} \bar \om(x)\geq \bar \om(0^-)-\e.
\ees
Thus, \eqref{jjszmhis_one} follows by the arbitrariness of $\e>0$.\\\\
We next claim that 
\be\label{plus_{n(m)}eg_equal}
\bar \om(0^\pm)=\om^\pm(0).
\ee
By Theorem \ref{fine_properties_BV} part $2$ and the fact that $\bar \om=\om$ $\mathcal L^1$-a.e., we have 
\bes
\om^-(0)=\lim_{r\to 0}\frac1r\int_{-r}^0 \om(t)\,dt= \lim_{r\to 0}\frac1r\int_{-r}^0\bar  \om(t)\,dt = \bar \om(0^-),
\ees
where at the last equality we used the properties of absolutely continuous function and the definition of $\bar \om(0^-)$. The equation $\bar\om(0^+)=\om^+(0)$ can be proved similarly.\\\\
Therefore
\bes
\liminf_{n\to\infty}\essinf_{x\in(x_n,y_n)} \om(x) = \liminf_{n\to\infty}\inf_{x\in(x_n,y_n)} \bar \om(x) \geq \bar \om(0^-)=\om^-(0),
\ees
which concludes \eqref{lower_bdd_1dnd}, and \eqref{lower_bdd_2dnd} and \eqref{lower_bdd_3dnd} hold by \eqref{lemma_jump_jiuehuishi} and \eqref{plus_{n(m)}eg_equal}.
\end{proof}
\begin{define}\label{LpSobBVSBVdef}(Weighted function spaces)
Let $\omega\in {\mathcal W(\Om)}$ and $1\leq p<\infty$:
\begin{enumerate}[1.]
\item
$L_\om^p(\Om)$ is the space of functions $u\in L^p(\Om)$ such that
\bes
\into \abs{u}^p\om\,dx<\infty,
\ees
endowed with the norm 
\bes
\norm{u-v}_{L^p_\om}:=\fsp{\into\abs{u-v}^p\om\,dx}^{\frac1p}
\ees
if $u$, $v\in L^p_\om(\Om)$;\\
\item
$W^{1,p}_\om(\Om)$ is the space of functions $u\in W^{1,p}(\Om)$ such that
\bes
u\in L^p_\om(\Om)\,\,\text{ and }\nabla u\in L_\om^p(\Om;\rn),
\ees
endowed with the norm 
\bes
\norm{u-v}_{W^{1,p}_\om}:=\norm{u-v}_{L^p_\om}+\norm{\nabla u-\nabla v}_{L^p_\om}
\ees
if $u$, $v\in W^{1,p}_\om(\Om)$;\\
\item
$BV_\omega(\Omega)$ is the space of functions $u\in BV(\Om)$ such that
\bes
u\in L^1_\om(\Om)\text{ and }\into \om \,d\abs{Du}<\infty,
\ees
endowed with the norm
\bes
\norm{u-v}_{BV_\om}:=\norm{u-v}_{L^1_\om}+\into\om\,d\abs{Du-Dv}
\ees
if $u,v\in BV_\om(\Om)$;
\item
$u\in SBV_\om(\Om)$ if $u\in BV_\om(\Om)\cap SBV(\Om)$, and $u\in GSBV_\om(\Om)$ if $K\wedge u\vee -K\in SBV_\om(\Om)$ for all $K\in\mathbb N$.
\end{enumerate}
\end{define}

\begin{lemma}\label{blow_up_kill_jump}
Let $\om\in\WW(\Om)$ be given, and suppose that $u\in SBV_\om(\Om)$. Then 
\bes
\HH^{N-1}(S_u\cap \flp{\om=+\infty})=0.
\ees
\end{lemma}

\begin{proof}
By Definition \ref{LpSobBVSBVdef} we have 
\begin{align}\label{finite_jump_measure}
\begin{split}
+\infty>\into \om\, d\abs{Du}&= \into \abs{\nabla u}\om\,dx+\int_{S_u} \abs{u^+-u^-}\om\,d\HH^{N-1}\\
&\geq \int_{S_u\cap \flp{\om=+\infty}}\abs{u^+-u^-}\om\,d\HH^{N-1}.
\end{split}
\end{align}
Since $\abs{u^+-u^-}(x)>0$ for $\hnmo$-a.e. $x\in S_u$, it follows from \eqref{finite_jump_measure} that $\HH^{N-1}(S_u\cap \flp{\om=+\infty})=0$.
\end{proof}

\begin{lemma}
The space $L^2_\om$ is a Hilbert space endowed with the inner product 
\be\label{inner_prod_weighted}
(u,v)_{L^2_\om} :=(u,v\,\om)_{L^2} = \int u\,v\,\om\,dx.
\ee
\end{lemma}
\begin{proof}
It is clear that \eqref{inner_prod_weighted} is an inner product. Also, $(u,u)_{L^2_\om} = (u\sqrt{\om},u\sqrt\om)_{L^2} \geq 0$, and if $(u,u)_{L^2_\om}=0$ then by \eqref{pos_lower_ass}
\bes
\into u^2\om\,dx\geq \into u^2dx=0,
\ees
and thus $u= 0$ $a.e.$\\\\
To see that $L^2_\om$ is complete, and therefore a Hilbert space, let $\seqn{u_n}$ be a Cauchy sequence in $L^2_\om$ and notice that $\seqn{u_n\sqrt{\om}}$ is a Cauchy sequence in $L^2$. Hence, there is a function $v\in L^2$ such that $u_n\sqrt{\om}\to v$ in $L^2$. Defining $u:=v/\sqrt{\om}$, we have that $u\in L^2_\om$ and $u_n\to u$ in $L^2_\om$.
\end{proof}

\begin{lemma}\label{L_2_weighted}
Let $\seqn{u_n}\subset W^{1,2}_\om(\Om)$ be such that $u_n\to u$ in $L^1_\om$ and
\begin{align*}
\sup\into \abs{\nabla u_n}^2\om\,dx<\infty.
\end{align*}
Then, for every measurable set $A\subset \Om$
\bes
\liminf_{n\to \infty} \int_A \abs{\nabla u_n}^2\om\,dx\geq \int_A \abs{\nabla u}^2\om\,dx,
\ees
and $u\in W_\om^{1,2}(\Om)$.
\end{lemma}
\begin{proof}
By \eqref{pos_lower_ass} we have that $\seqn{\nabla u_n}$ is uniformly bounded in $L^2(\Om, \rn)$ and $u_n\to u$ in $L^1(\Om)$. Hence $\nabla u_n\wto \nabla u$ in $L^2(\Om;\rn)$, and using standard lower semi-continuity of convex energies (see \cite{fonseca2015modern}, Theorem 6.3.7), we conclude that 
\bes
+\infty>\liminf_{n\to \infty} \int_A \abs{\nabla u_n}^2\om\,dx\geq \int_A \abs{\nabla u}^2\om\,dx,
\ees
for every measurable subset $A\subset \Om$. In particular, with $A=\Om$ and using the fact that $1\leq \om$ a.e., we deduce that $u\in W_\om^{1,2}(\Om)$.
\end{proof}

\begin{lemma}\label{compact_energy}
Let $u\in L^1_\om(\Om)$ be such that 
\be\label{compact_energy1}
\int_\Om\abs{\nabla u}^2\om\,dx+\int_{S_u}\om\,d\hnmo<+\infty.
\ee
Then $\hnmo(S_u)<+\infty$ and $u\in GSBV_\om(\Om)$.
\end{lemma}
\begin{proof}
By \eqref{compact_energy1} and \eqref{pos_lower_ass}
\bes
\int_\Om\abs{\nabla u}^2\,dx+\hnmo(S_u)<+\infty,
\ees
and hence by \cite{ambrosio1990approximation} we have that $u\in GSBV(\Om)$. To show that $u\in GSBV_\om(\Om)$, we only need to verify that 
\bes
\int_{S_{u_K}}\abs{u_K^+-u_K^-}\om\,d\hnmo<+\infty
\ees
for every $K\in\mathbb N$ and with $u_K:=K\wedge u\vee -K$. Indeed, by \eqref{compact_energy1}
\bes
\int_{S_{u_K}}\abs{u_K^+-u_K^-}\om\,d\hnmo\leq 2K\int_{S_{u_K}}\om\,d\hnmo\leq 2K\int_{S_{u}}\om\,d\hnmo<+\infty.
\ees
\end{proof}

\section{The One Dimensional Case}\label{O_D_Toy_Sec}

\subsection{The Case $\om\in \mathcal W(I)\cap C(I)$}\label{ATAc_1d}$\,$\\

Let $\om\in\mathcal W(I)\cap  C(I)$ be given. Consider the functionals
\begin{align}\label{ATH_weighted_c}
E_{\om,\e}(u,v):=\int_I v^2\abs{u'}^2\om \,dx +\int_{I}\left[\frac{\e}{2}\abs{ v'}^2+\frac{1}{2\e}{(v-1)^2}\right]\om\, dx
\end{align}
for $(u,v)\in W_\om^{1,2}(I)\times W^{1,2}(I)$, and let
\be\label{AT_weighted_c}
E_\om(u):=\int_I \abs{u'}^2\om\,dx+ \sum_{x\in S_u}\om(x)
\ee
be defined for $u\in GSBV_\om(I)$ (Note that $E_{1,\e}(u,v)$ and $E_1(u)$ are, respectively, the non-weighted Ambrosio-Tortorelli approximation scheme and Mumford-Shah functional studied in \cite{ambrosio1990approximation}).
\begin{theorem}[$\Gamma$-Convergence]\label{gamma_c_1d}
Let $\mathcal{E}_{\om,\e}$: $L^1_\om(I)\times L^1(I)\rightarrow[0,+\infty]$ be defined by
\begin{align}\label{esd_weighted}
\mathcal{E}_{\om,\e}(u,v):=
\begin{cases}
E_{\om,\e}(u,v)&\text{if}\,\,(u,v)\in W_\om^{1,2}(I)\times W^{1,2}(I), \,0\leq v\leq 1,\\
+\infty &\text{otherwise}.
\end{cases}
\end{align}
Then the functionals $\mathcal{E}_{\om,\e}$ $\Gamma$-converge, with respect to the $L^1_\om\times L^1$ topology, to the functional
\begin{align}\label{etd_weighted}
\mathcal{E}_\om(u,v):=
\begin{cases}
E_\om(u)&\text{if}\,\,u\in GSBV_\om(I)\text{ and }v=1\,\,a.e.,\\
+\infty &\text{otherwise}.
\end{cases}
\end{align}
\end{theorem}

We begin with an auxiliary proposition.
\begin{proposition}\label{osl_small_contral}
Let $\flp{v_\e}_{\e>0}\subset W^{1,2}(I)$ be such that $0\leq v_\e\leq 1$, $v_\e\to1$ in $L^1(I)$ and pointwise a.e., and
\be\label{osl_small_contral_weighted}
\limsup_{\e\to0}\int_{I}\left[\frac{\e}{2}\abs{ v_\e'}^2+\frac{1}{2\e}(v_\e-1)^2\right]\,dx <\infty.
\ee
Then for arbitrary $0<\eta<1$ there exists an open set $H_\eta\subset I$ satisfying: 
\begin{enumerate}[1.]
\item
the set $I\setminus H_\eta$ is a collection of finitely many points in $I$;
\item
for every set $K$ compactly contained in $H_\eta$, we have $K\subset B_\e^\eta$ for $\e>0$ small enough, where 
\be\label{1_d_level_set}
B_\e^\eta:=\flp{x\in I:\,v^2_\e(x)\geq \eta}.
\ee
\end{enumerate}
\end{proposition}

Proposition \ref{osl_small_contral} is adapted from \cite{ambrosio1990approximation}, page 1020-1021 (see Lemma \ref{osl_small_contral_proof}).
\begin{proposition}\label{liminf_part_1d_c}\emph{($\Gamma$-$\liminf$)}
For $u\in L^1_\om(I)$, let
\begin{align*}
E_\om^-(u):=\inf&\flp{\liminf_{\e\to 0} E_{\om,\e}(u_\e,v_\e):\right.\\
&\left.\,\,\,(u_\e,v_\e)\in W^{1,2}_\om(I)\times W^{1,2}(I), u_\e\to u\text{ in }L^1_\om,\, v_\e\to1\text{ in }L^1,\,0\leq v_\e\leq 1}.
\end{align*}
We have
\bes
E_\om^-(u)\geq E_\om(u).
\ees
\end{proposition}
\begin{proof}
If $E_\om^-(u)=+\infty$ then there is nothing to prove. Assume that $M:=E_\om^-(u)<\infty$. Choose $u_\e$ and $v_\e$ admissible for $E_\om^-(u)$ such that 
\bes
\lim_{\e\to 0}E_{\om,\e}(u_\e,v_\e) = E_\om^-(u)<\infty,
\ees
and note that $v_\e\to 1$ in $L^1(I)$. Since $\inf_{x\in\Om}\om(x)\geq 1$, we have 
\bes
\liminf_{\e\to 0}E_{1,\e}(u_\e,v_\e)\leq \liminf_{\e\to 0}E_{\om,\e}(u_\e,v_\e)<+\infty,
\ees
and by \cite{ambrosio1990approximation} we obtain that
\be\label{use_la_nonsm0}
u\in GSBV(I)\text{ and }\mathcal H^0(S_u)<+\infty.
\ee
Let $\bar \e>0$ be sufficiently small so that, for all $0<\e<\bar \e$,
\bes
E_{\om,\e}(u_\e,v_\e)\leq M+1.
\ees
We claim, separately, that
\be\label{use_la_nonsm1}
\int_I\abs{u'}^2\om\,dx\leq \liminf_{\e\to 0}\int_I\abs{u_\e'}^2v_\e^2\,\om\,dx<+\infty,
\ee
and 
\be\label{use_la_nonsm2}
\sum_{x\in S_u} \om(x)\leq \liminf_{\e\to 0}\int_I \fmp{\frac{1}{2}\e\abs{v_\e'}^2+\frac{1}{2\e}(1-v_\e)^2}\om\,dx<+\infty.
\ee
Note that \eqref{use_la_nonsm1}, \eqref{use_la_nonsm2}, and Lemma \ref{compact_energy} will yield $u\in GSBV_\om(I)$.\\\\
Up to the extraction of a (not relabeled) subsequence, we have $u_\e\to u$ and $v_\e\to 1$ \emph{a.e.} in $I$ with
\bes
\limsup_{\e\to 0}\int_I \fmp{\frac{1}{2}\e\abs{v_\e'}^2+\frac{1}{2\e}(1-v_\e)^2}\,dx \leq \limsup_{\e\to 0}\int_I \fmp{\frac{1}{2}\e\abs{v_\e'}^2+\frac{1}{2\e}(1-v_\e)^2}\om\,dx <+\infty.
\ees
Therefore, up to the extraction of a (not relabeled) subsequence, we can apply Proposition \ref{osl_small_contral} and deduce that, for a fixed $\eta\in(1/2,1)$, there exists an open set $H_\eta$ such that the set $I\setminus H_\eta$ contains only a finite number of points, and for every compact subset $K\subset\subset H_\eta$, $K$ is contained in $B_\e^\eta$ for $0<\e<\e(K)$, where $B_\e^\eta$ is defined in \eqref{1_d_level_set}. We have
\begin{align}\label{liminf_2}
\begin{split}
&\int_{K}\abs{u'}^2\om\,dx\leq \liminf_{\e\to 0}\int_{K} \abs{u_\e'}^2\om\,dx\\
&\leq \frac{1}{\eta}\liminf_{\e\to 0}\int_{K} v_\e^2\abs{u_\e'}^2\om\,dx\leq \frac{1}{\eta}\liminf_{\e\to 0}\int_I v_\e^2\abs{u_\e'}^2\om\,dx,
\end{split}
\end{align}
where we used Lemma \ref{L_2_weighted} in the first inequality. By letting $K\nearrow H_\eta$ on the left hand side of \eqref{liminf_2} first and then $\eta\nearrow 1$ on the right hand side, we proved that
\be\label{contpart1}
\int_{I}\abs{u'}^2\om\,dx\leq\liminf_{\e\to 0}\int_I v_\e^2\abs{u_\e'}^2\om\,dx,
\ee
where we used the fact that $\abs{I\setminus H_\eta}=0$.\\\\
We claim that $S_u\subset I\setminus H_\eta$. Indeed, if there is $x_0\in S_u\cap  H_\eta$, since $H_\eta$ is open there exists an open interval $I'_0$ containing $x_0$ and compactly contained in $H_\eta$ such that for $0<\e<\e_0'$ 
\bes
\int_{I_0'} \abs{u'_\e}^2dx\leq\int_{I_0'} \abs{u'_\e}^2\om\,dx\leq \frac{1}{\eta} \int_I v_\e^2\abs{u_\e'}^2\om\,dx\leq 2(M+1).
\ees
Thus $u\in W^{1,2}(I_0')$, and hence is continuous at $x_0$, which contradicts the fact that $x_0\in S_u$.\\\\
Let $t\in S_u$, and for simplicity assume that $t=0$. We claim that there exist $\seqn{t_n^1}$, $\seqn{t_n^2}$, and  $\seqn{s_n}$ such that $-1<t_n^1<s_n<t_n^2<1$,
\begin{align*}
\limn t_n^1=\limn t_n^2 =\limn s_{n} =0,
\end{align*}
and, up to the extraction of a subsequence of $\seqe{v_{\e}}$,
\be\label{lisfinding}
 \limn v_{\e(n)}(t_n^1)=\limn v_{\e(n)}(t_n^2) = 1,\,\text{ and } \limn v_{\e(n)}(s_{n})=0.
\ee
Because $I\setminus H_\eta$ is discrete and $0\in I\setminus H_\eta$, we may choose $\delta_0>0$ small enough such that 
\bes
(-2\delta_0,2\delta_0)\cap (I\setminus H_\eta) = \flp{0}.
\ees 
We claim that
\be\label{keng_die_seq_1}
\limsup_{\delta\to 0^+}\limsup_{\e\to 0^+}\inf_{x\in I_\delta}v_\e(x)=0,
\ee
where $I_{\delta}:=(-\delta,\delta)$. Assume that 
\bes
\limsup_{\delta\to 0^+}\limsup_{\e\to 0^+}\inf_{x\in I_\delta}v_\e(x)=:\alpha>0.
\ees
Then there exists $0<\delta_{\alpha}<\delta_0$ such that
\bes
\limsup_{\e\to 0^+}\inf_{x\in I_{\delta_{\alpha}}}v_\e(x)\geq\frac{2}{3}\alpha>0.
\ees
Up to the extraction of a subsequence of $\flp{v_\e}_{\e>0}$, there exists $\e^{\delta_{\alpha}}_0>0$ such that 
\bes
\inf_{x\in I_{\delta_{\alpha}}}v_\e(x)\geq \frac1{2}{\alpha}>0,
\ees
for all $0<\e<\e^{\delta_{\alpha}}_0$, and we have
\begin{multline*}
\int_{I_{\delta_{\alpha}}}\abs{u'}^2dx\leq \int_{I_{\delta_{\alpha}}} \abs{u'}^2\om \,dx\\
\leq \liminf_{\e\to 0}\int_{I_{\delta_{\alpha}}} \abs{u'_\e}^2\om\,dx\leq  \liminf_{\e\to 0}\frac{2}{\alpha} \int_{I_{\delta_{\alpha}}} \abs{u_\e'}^2v_\e^2\,\om\,dx\leq \liminf_{\e\to 0} \frac{2}{\alpha} \int_I \abs{u_\e'}^2v_\e^2\,\om\,dx< \frac{2}{\alpha}(M+1).
\end{multline*}
Hence $u\in W^{1,2}({I_{\delta_{\alpha}}})$ and so $u$ is continuous at $0\in S_u$, and we reduce a contradiction. Therefore, in view of \eqref{keng_die_seq_1} we may find $\delta_n\to 0^+$, $\e(n)\to 0^+$, and $s_{n}\in (-\delta_n,\delta_n)$ such that 
\bes
\limn s_{n}=0\text{ and }\limn v_{\e(n)}(s_{n})=0.
\ees
We claim that for all $\tau\in(0,1/2)$,
\be\label{keng_die_seq_2}
\lim_{n\to \infty}\fmp{\inf_{x\in (s_{n}-\tau,s_{n})}(1-v_{\e(n)}(x))+\inf_{y\in (s_{n},s_{n}+\tau)}(1-v_{\e(n)}(x))}=0.
\ee
To reach a contradiction, assume that there exists $\tau\in(0,1/2)$ such that 
\bes
\limsup_{n\to \infty}\fmp{\inf_{x\in (s_{n}-\tau,s_{n})}(1-v_{\e(n)}(x))+\inf_{x\in (s_{n},s_{n}+\tau)}(1-v_{\e(n)}(x))}=:\beta>0.
\ees
Without loss of generality, suppose that
\bes
\limsup_{n\to\infty}\inf_{x\in (s_{n}-\tau,s_{n})}(1-v_{\e(n)}(x))\geq\frac12{\beta}>0.
\ees
Then
\bes
\liminf_{n\to \infty}\sup_{x\in (s_{n}-\tau,s_{n})} v_{\e(n)}(x) \leq 1-\frac12\beta,
\ees
which implies that
\be\label{shangbufengding}
\sup_{x\in (s_{n_k}-\tau,s_{n_k})} v_{\e(n_k)}(x)\leq1-\frac13\beta
\ee
for a subsequence $\seqk{\e(n_k)}\subset \seqn{\e(n)}$. However, \eqref{shangbufengding} contradicts the fact that $v_{\e(n_k)}(x)\to 1$ a.e. since for $k$ large enough so that $\abs{s_{n_k}}<\tau/4$ it holds
\bes
\fsp{s_{n_k}-\tau,s_{n_k}}\supset \fsp{-\frac34\tau,-\frac\tau4}.
\ees
Therefore, in view of \eqref{keng_die_seq_2} we may find $t_m^1\in(s_{n(m)}-1/m,s_{n(m)})$ and $t_m^2\in(s_{n(m)},s_{n(m)}+1/m)$ such that 
\bes
\limn t_m^1=\limn t_m^2=0\text{ and }\limn v_{\e({n(m)})}(t_m^1)=\limn v_{\e({n(m)})}(t_m^2)=1.
\ees
We next show that 
\bes
 \liminf_{m\to \infty}\int_{t_m^1}^{s_{n(m)}} \fmp{\frac{1}{2}{\e(n(m))}\abs{(v_{\e(n(m))})'}^2+\frac{1}{2\e(n(m))}(1- v_{\e(n(m))})^2}dx\geq \frac{1}{2}.
\ees
Indeed, we have
\begin{align*}
&\liminf_{m\to\infty}\int_{t_m^1}^{s_{n(m)}} \fmp{\frac{1}{2}{\e(n(m))}\abs{(v_\epsn)'}^2+\frac{1}{2\e(n(m))}(1- v_\epsn)^2}dx \\
&\geq \liminf_{m\to\infty}\int_{t_m^1}^{s_{n(m)}} (1-v_\epsn)\abs{v_\epsn'}dx
\geq \liminf_{m\to\infty} \abs{\int_{t_m^1}^{s_{n(m)}} (1-v_\epsn){v_\epsn'}dx}\\
 &= \liminf_{m\to\infty} \frac12 \abs{ \int_{t_m^1}^{s_{n(m)}} \frac{d}{dt}(1-v_\epsn)^2dx}\\
 &=\frac12 \limn \fmp{(1-v_{\e(n(m))}(s_{n(m)}))^2-(1-v_{\e(n(m))}(t_m^1))^2}=\frac12,
\end{align*}
where we used \eqref{lisfinding}. Similarly, we obtain 
\bes
 \liminf_{m\to \infty}\int_{s_{n(m)}}^{t_m^2} \fmp{\frac{1}{2}{\e(n(m))}\abs{(v_\epsn)'}^2+\frac{1}{2\epsn}(1- v_\epsn)^2}dx\geq \frac{1}{2}.
\ees
We observe that, since $\om$ is positive,
\begin{align}\label{revel_1d_c_lower}
\begin{split}
&\int_{t_m^1}^{t_m^2} \fmp{\frac{1}{2}\epsn\abs{v_\epsn'}^2+\frac{1}{2\epsn}(1-v_\epsn)^2}\om(x)\,dx\\
\geq& \fsp{ \inf_{r\in (t_m^1,t_m^2)}{\om(r)}}\cdot \left\{\int_{t_m^1}^{s_{n(m)}} \fmp{\frac{1}{2}\epsn\abs{v_\epsn'}^2+\frac{1}{2\epsn}(1-v_\epsn)^2}dx\right.\\
&\left.+\int_{s_{n(m)}}^{t_m^2} \fmp{\frac{1}{2}\epsn\abs{(v_\epsn)'}^2+\frac{1}{2\epsn}(1- v_\epsn)^2}dx\right\},
\end{split}
\end{align}
and so
\begin{align*}
\liminf_{m\to \infty}&\int_{t_m^1}^{t_m^2} \fmp{\frac{1}{2}\epsn\abs{v_\epsn'}^2+\frac{1}{2\epsn}(1-v_\epsn)^2}\om(x)\,dx\\
\geq & \fsp{\liminf_{m\to \infty}  \inf_{r\in (t_m^1,t_m^2)}{\om(r)}}\liminf_{n\to \infty}\flp{ \int_{t_m^1}^{s_{n(m)}} \fmp{\frac{1}{2\epsn}(1- v_\epsn)^2+\frac{\e}{2}\abs{(v_\epsn)'}^2}dx\right.\\
&\left.+\int_{s_{n(m)}}^{t_m^2} \fmp{\frac{1}{2}\epsn\abs{v_\epsn'}^2+\frac{1}{2\epsn}(1-v_\epsn)^2}dx}\\
\geq& \fsp{\frac12+\frac12} {\om(0)}=\om(0),
\end{align*}
where we used the fact that $\om$ is continuous at $0$.\\\\
Finally, since $S_u\subset I\setminus H_\eta$, by \eqref{use_la_nonsm0} we have that $S_u$ is a finite collection of points, and we may repeat the above argument for all $t\in S_u$ by partitioning $I$ into non-overlaping intervals where there is at most one point of $S_u$, to deduce that 
\be\label{contpart2}
\liminf_{\e\to 0}\int_I \fmp{\frac{1}{2}\e\abs{v_\e'}^2+\frac{1}{2\e}(1-v_\e)^2}\om(x)\,dx\geq \sum_{x\in S_u} \om(x).
\ee
In view of \eqref{contpart1} and \eqref{contpart2}, we conclude that 
\bes
\liminf_{\e\to 0} E_{\om,\e}(u_\e,v_\e)\geq E_\om(u).
\ees
\end{proof}

\begin{proposition}\label{limsup_1d_c}\emph{($\Gamma$-$\limsup$)}
For $u\in L^1_\om(I)\cap L^\infty(I)$, let
\begin{align*}
E_\om^+(u):=\inf&\flp{\limsup_{\e\to 0} E_{\om,\e}(u_\e,v_\e):\right.\\
&\left.\,\,\,(u_\e,v_\e)\in W^{1,2}_\om(I)\times W^{1,2}(I), u_\e\to u\text{ in }L^1_\om,\, v_\e\to1\text{ in }L^1,\,0\leq v_\e\leq 1}.
\end{align*}
We have
\be\label{cubnone1}
E_\om^+(u)\leq E_\om(u).
\ee
\end{proposition}
\begin{proof}
Without loss of generality, assume that $E_\om(u)<\infty$. Then by Lemma \ref{compact_energy} we have $u\in GSBV_\om(I)$  and $\mathcal H^0(S_u)<\infty$. To prove \eqref{cubnone1}, we show that there exist $\seqe{u_\e}\subset W_\om^{1,2}(I)$ and $\seqe{v_\e}\subset W^{1,2}(I)$ such that $u_\e\to u$ in $L_\om^1$, $v_\e\to 1$ in $L^1$, $0\leq v_\e\leq 1$, and 
\be\label{cubnone}
\limsup_{\e\to 0}E_{\om,\e}(u_\e,v_\e)\leq E_\om(u).
\ee
\underline{Step 1}: Assume that $S_u=\flp{0}$. \\\\
Fix $\eta>0$, and let $T>0$ and $v_0\in W^{1,2}(0,T)$ be such that 
\be\label{recovery_basic}
0\leq v_0\leq 1\,\,\,\,\,\text{ and }\int_0^T \fmp{(1-v_0)^2+\abs{v_0'}^2}\,dx\leq 1+\eta,
\ee
with $v_0(0)=0$ and $v_0(T)=1$.\\\\
For $\xi_\e = o(\e)$ we define
\be\label{recovery_seq}
v_\e(x):=
\begin{cases}
0 & \text{ if }\abs{x}\leq \xi_\e,\\
v_0\fsp{\frac{\abs{x}-\xi_\e}{\e}}& \text{ if }\xi_\e< \abs{x}<\xi_\e+\e T,\\
1 & \text{ if }\abs{x}\geq \xi_\e+\e T.
\end{cases}
\ee
Since $\norm{v_\e}_{L^\infty(I)}\leq 1$, by Lebesgue Dominated Convergence Theorem we have $v_\e\to 1$ in $L^1$. Let
\be\label{recovery_seq_u}
u_{\e}(x):=
\begin{cases}
u(x) & \text{ if }\abs{x}\geq \frac{1}{2}\xi_\e,\\
\text{affine from $u\fsp{-\frac12\xi_\e}$ to $u\fsp{\frac12\xi_\e}$}& \text{ if }\abs{x}<\frac12\xi_\e.
\end{cases}
\ee
and we observe that (recall in assumption we have $u\in L^\infty(I)$)
\bes
\norm{u_\e}_{L^\infty(I)}\leq \norm{u}_{L^\infty(I)},
\ees
and
\bes
\int_I \norm{u}_{L^\infty(I)}\om\,dx<\infty.
\ees
Therefore, by Lebesgue Dominated Convergence Theorem we deduce that $u_\e\to u$ in $L^1_\om$. Moreover, by \eqref{recovery_seq} and \eqref{recovery_seq_u} we observe that 
\be\label{zero_eating_jump}
v_\e^2\abs{u_\e'}^2=
\begin{cases}
v_\e^2\abs{u'}^2&\text{ if }x\geq \abs{\xi_\e},\\
0&\text{ if }x< \abs{\xi_\e},
\end{cases}
\ee
and so $v_\e^2\abs{u_\e'}^2\leq \abs{u'}^2$. Since $E_\om(u)<\infty$ we have $u'\in L_\om^2(I)$, by Lebesgue Dominated Convergence Theorem we obtain
\bes
\lim_{\e\to 0}\int_I {v_\e^2\abs{u_\e'}^2}\om\,dx = \int_I \abs{u'}^2\om\,dx.
\ees
Next, since $\om$ is positive we have
\begin{align*}
&\int_{I}\left[\frac{\e}{2}\abs{ v_\e'}^2+\frac{1}{2\e}{(v_\e-1)^2}\right]\om(x)\, dx \\
=& \int_{-\xi_\e-\e T}^{-\xi_\e}\left[\frac{\e}{2}\abs{ v_\e'}^2+\frac{1}{2\e}{(v_\e-1)^2}\right]\om(x)\, dx +  \int^{\xi_\e+\e T}_{\xi_\e}\left[\frac{\e}{2}\abs{ v_\e'}^2+\frac{1}{2\e}{(v_\e-1)^2}\right]\om(x)\, dx+\frac{1}{2\e}\int_{-\xi_\e}^{\xi_\e}\om(x)dx \\
\leq& \fsp{\sup_{t\in(-\xi_\e-\e T, \xi_\e+\e T)}\om(t)}\cdot\flp{ \int_{-\xi_\e-\e T}^{-\xi_\e}\left[\frac{\e}{2}\abs{ v_\e'}^2+\frac{1}{2\e}{(v_\e-1)^2}\right]\, dx \right.\\
&\left.+  \int^{\xi_\e+\e T}_{\xi_\e}\left[\frac{\e}{2}\abs{ v_\e'}^2+\frac{1}{2\e}{(v_\e-1)^2}\right]\, dx}+\frac{\xi_\e}{\e}\norm{\om}_{L^\infty}.
\end{align*}
We obtain
\begin{align}\label{limsup_cal}
\begin{split}
\limsup_{\e\to 0}&\int_{I}\left[\frac{\e}{2}\abs{ v_\e'}^2+\frac{1}{2\e}{(v_\e-1)^2}\right]\om(x)\, dx \\
\leq &\limsup_{\e\to 0}\fsp{\sup_{t\in(-\xi_\e-\e T, \xi_\e+\e T)}\om(t)}\cdot\\
&\limsup_{\e\to 0}\flp{ \int_{-\xi_\e-\e T}^{-\xi_\e}\left[\frac{\e}{2}\abs{ v_\e'}^2+\frac{1}{2\e}{(v_\e-1)^2}\right]\, dx +  \int^{\xi_\e+\e T}_{\xi_\e}\left[\frac{\e}{2}\abs{ v_\e'}^2+\frac{1}{2\e}{(v_\e-1)^2}\right]\, dx}\\
\leq&\, \om(0)(1+\eta),
\end{split}
\end{align}
where we used \eqref{recovery_basic}.\\\\
We conclude that
\bes
\limsup_{\e\to 0}E_{\om,\e}(u_\e,v_\e)\leq \int_I \abs{u'}^2\om\,dx + \om(0)(1+\eta),
\ees
and \eqref{cubnone} follows by the arbitrariness of $\eta$.\\\\
\underline{Step 2:} In the general case in which $S_u$ is finite, we obtain $u_\e$ by repeating the construction in Step 1 (see \eqref{recovery_seq_u}) in small non-overlapping intervals centered at each point in $S_u$. To obtain $v_\e$, we repeat the construction \eqref{recovery_seq} in those intervals and extend by $1$ in the complement of the union of those intervals. Hence, by Step 1 we have
\bes
\limsup_{\e\to 0}E_{\om,\e}(u_\e,v_\e)\leq \int_I \abs{u'}^2\om\,dx + (1+\eta)\sum_{x\in S_u}\om(x),
\ees
and again \eqref{cubnone} follows by letting $\eta\to 0^+$.
\end{proof}
\begin{proof}[Proof of Theorem \ref{gamma_c_1d}]
The $\liminf$ inequality follows from Proposition \ref{liminf_part_1d_c}. For the $\limsup$ inequality, we note that for any given $u\in GSBV_\om$ such that $E_\om(u)<+\infty$, by Lebesgue Monotone Convergence Theorem we have that
\bes
E_\om(u)=\lim_{K\to\infty}E_\om(K\wedge u\vee -K),
\ees
and hence a diagonal argument together with Proposition \ref{limsup_1d_c} conclude the proof.
\end{proof}


\subsection{The Case $\om\in\mathcal W(I)\cap SBV(I)$}\label{1_d_jump_intro}$\,$\\\\
Consider the functionals
\begin{align}\label{ATH_weighted}
E_{\om,\e}(u,v):=\int_I \abs{u'}^2v^2\om \,dx +\int_{I}\left[\frac{\e}{2}\abs{ v'}^2+\frac{1}{2\e}{(v-1)^2}\right]\om\, dx
\end{align}
for $(u,v)\in W_\om^{1,2}(I)\times W^{1,2}(I)$, and for $u\in GSBV_\om(I)$ let
\be\label{AT_weighted}
E_\om(u):=\int_I \abs{u'}^2\om\,dx+ \sum_{x\in S_u}\om^-(x).
\ee
We note that if $\om\in \mathcal W(I)\cap SBV(I)$ and $\om$ is continuous in a neighborhood of $S_u$, for $u\in GSBV_\om(I)$, then
\bes
\sum_{x\in S_u}\om^-(x)=\sum_{x\in S_u}\om(x)
\ees
and Theorem \ref{gamma_c_1d} still holds.\\\\
Here we study the case in which $\om$ is no longer continuous on a neighborhood of $S_u$. We recall that $\om\in SBV(I)$ implies that $\om\in L^\infty(I)$ and by definition of $\om\in\mathcal W(I)$, we have $\mathcal H^0( S_\om)<\infty$. Also, we note that $\om^-$ is defined $\mathcal H^0$-a.e, hence everywhere in $I$.
\begin{theorem}\label{gamma_1d_jump}
Let $\mathcal{E}_\e$: $L^1_\om(I)\times L^1(I)\rightarrow[0,+\infty]$ be defined by
\begin{align}\label{esd_weighted}
\mathcal{E}_{\om,\e}(u,v):=
\begin{cases}
E_{\om,\e}(u,v)&\text{if}\,\,(u,v)\in W_\om^{1,2}(I)\times W^{1,2}(I), \,0\leq v\leq 1,\\
+\infty &\text{otherwise}.
\end{cases}
\end{align}
Then the functionals $\mathcal{E}_{\om,\e}$ $\Gamma$-converge, with respect to the $L^1_\om\times L^1$ topology, to the functional
\begin{align}\label{etd_weighted}
\mathcal{E}_\om(u,v):=
\begin{cases}
E_\om(u)&\text{if}\,\,u\in GSBV_\om(I)\text{ and }v=1\,\,a.e.,\\
+\infty &\text{otherwise}.
\end{cases}
\end{align}
\end{theorem}
The proof of Theorem \ref{gamma_1d_jump} will be split into two propositions.

\begin{proposition}\label{liminf_part_1d}\emph{($\Gamma$-$\liminf$)}
For $u\in L^1_\om(I)$, let
\begin{align*}
E_\om^-(u):=\inf&\flp{\liminf_{\e\to 0} E_{\om,\e}(u_\e,v_\e):\right.\\
&\left.\,\,\,(u_\e,v_\e)\in W^{1,2}_\om(I)\times W^{1,2}(I),\, u_\e\to u \text{ in }L^1_\om, v_\e\to1\text{ in }L^1,\,0\leq v_\e\leq 1}.
\end{align*}
We have
\bes
E_\om^-(u)\geq E_\om(u).s_{n(m)}
\ees
\end{proposition}
\begin{proof}
Without lose of generality, assume that $E_\om^-(u)<+\infty$. We use the same arguments of the proof of Proposition \ref{liminf_part_1d_c} until \eqref{revel_1d_c_lower}. In particular, \eqref{use_la_nonsm0} and \eqref{use_la_nonsm1} still hold, that is
\bes
\mathcal H^0(S_u)<+\infty\text{ and }\int_I\abs{u'}^2\om\,dx\leq \liminf_{\e\to 0}\int_I\abs{u_\e'}^2v_\e^2\,\om\,dx.
\ees
Invoking \eqref{revel_1d_c_lower}, we have
\begin{align*}
&\liminf_{m\to \infty}\int_{t_m^1}^{t_m^2} \fmp{\frac{1}{2}\epsn\abs{v_\epsn'}^2+\frac{1}{2\epsn}(1-v_\epsn)^2}\om(x)\,dx\\
&\geq \fsp{\liminf_{m\to \infty}  \essinf_{r\in (t_m^1,t_m^2)}{\om(r)}}\cdot\liminf_{n\to \infty}\flp{ \int_{t_m^1}^{s_{n(m)}} \fmp{\frac{1}{2}\epsn\abs{(v_\epsn)'}^2+\frac{1}{2\epsn}(1- v_\epsn)^2}dx\right.\\
&\left.\,\,\,\,\,\,\,\,+\int_{s_{n(m)}}^{t_m^2} \fmp{\frac{1}{2}\epsn\abs{(v_\epsn)'}^2+\frac{1}{2\epsn}(1- v_\epsn)^2}dx}\\
&\geq \,\om^-(0)\fsp{\frac12+\frac12}=\om^-(0),
\end{align*}
where the last step is justified by \eqref{lower_bdd_1dnd}.\\\\
Since $S_u$ is finite, we may repeat the above argument for all $t\in S_u$ by partitioning $I$ into finitely many non-overlapping intervals where there is at most one point of $S_u$, to conclude that 
\bes
\liminf_{\e\to 0}\int_I \fmp{\frac{1}{2}\e\abs{v_\e'}^2+\frac{1}{2\e}(1-v_\e)^2}\om(x)\,dx\geq \sum_{x\in S_u} \om^-(x),
\ees
as desired.
\end{proof}
The construction of the recovery sequence uses a reflection argument nearby points of $S_\om\cap S_u$.

\begin{proposition}\emph{($\Gamma$-$\limsup$)}\label{1d_jump_reco}
For $u\in L^1_\om(I)\cap L^\infty(I)$, let
\begin{align*}
E_\om^+(u):=\inf&\flp{\limsup_{\e\to 0} E_{\om,\e}(u_\e,v_\e):\right.\\
&\left.\,\,\,(u_\e,v_\e)\in W^{1,2}_\om(I)\times W^{1,2}(I), u_\e\to u\text{ in }L^1_\om,\, v_\e\to1\text{ in }L^1,\,0\leq v_\e\leq 1}.
\end{align*}
We have
\be\label{jump_case_woyaosisisis1}
E_\om^+(u)\leq E_\om(u).
\ee
\end{proposition}
\begin{proof}
To prove \eqref{jump_case_woyaosisisis1}, we only need to explicitly construct a sequence $\seqe{(u_\e,v_\e)}\subset W^{1,2}_\om(I)\times W^{1,2}(I)$ such that $u_\e\to u$\text{ in }$L^1_\om$, $v_\e\to1$\text{ in }$L^1$, $0\leq v_\e\leq 1$, and 
\be\label{jump_case_woyaosisisis}
\limsup_{\e\to 0}E_{\om,\e}(u_\e,v_\e)\leq E_\om(u).
\ee
\underline{Step 1}: Assume that $\flp{0}=S_u \subset S_\om $.\\\\
Recall that we always identify $\om$ with its approximation representative $\bar \om$, and by \eqref{lower_bdd_2dnd} we may assume that (the converse situation may be dealt with similarly)
\be\label{direction_negative}
\lim_{t\nearrow 0^-}\om(t)=\om^-(0)\text{ and }\lim_{t\searrow 0^+}\om(t)=\om^+(0).
\ee
Fix $\eta>0$. For $\e>0$ small enough, and with $\xi_\e =o(\e)$, as in \eqref{recovery_basic}, \eqref{recovery_seq} let
\bes
\tilde v_\e(x):=
\begin{cases}
0 & \text{ if }\abs{x}\leq \xi_\e\\
 v_0\fsp{\frac{\abs{x}-\xi_\e}{\e}}& \text{ if }\xi_\e< \abs{x}<\xi_\e+\e T\\
1 & \text{ if }\abs{x}\geq \xi_\e+\e T,
\end{cases}
\ees
and define
\be\label{one_side_recovery_v}
v_\e(x):=\tilde v_\e(x+2\xi_\e+\e T).
\ee
Note that from \eqref{recovery_seq} $v_\e\to 1$ a.e., and since $0\leq v_\e\leq 1$, by Lebesgue Dominated Convergence Theorem we have $v_\e\to v$ in $L^1$. We also note that 
\be\label{0zrea}
\frac{\e}2\abs{v'_\e(x)}^2+\frac{1}{2\e}(1-v_\e(x))^2=0 
\ee
if $x\in (-1, -3\xi_\e-2\e T)\cup(-\xi_\e,1)$, and if $x\in(-3\xi_\e-\e T,-\xi_\e-\e T)$ then
\be\label{00zrea}
v_\e(x)=0.
\ee
Set
\bes
\tilde u_\e(x):=
\begin{cases}
u(x)&\text{ if }x\in (-1, -2\xi_\e-\e T)\cup (0,1),\\
u(-x)&\text{ if } x\in [-2\xi_\e-\e T,0].
\end{cases}
\ees 
Observe that $\tilde u_\e(x)$ is continuous at $0$ since $\tilde u_\e^+(0) = \tilde u_\e^-(0)=u^+(0)$ by the definition of $\tilde u_\e(x)$, and $\tilde u_\e$ may only jump at $t={-2\xi_\e-\e T}$ but not at $t={0}$ where $u$ jumps.\\\\
We define the recovery sequence
\be\label{one_side_recovery_u}
u_\e(x):=
\begin{cases}
\tilde u_\e(x)&\text{ if }x\in I\setminus [-2.5\xi_\e-\e T,-1.5\xi_\e-\e T],\\
\text{affine from }\tilde u_\e(-2.5\xi_\e-\e T)\text{ to }\tilde u_\e(-1.5\xi_\e-\e T)&\text{ if } x\in [-2.5\xi_\e-\e T,-1.5\xi_\e-\e T].
\end{cases}
\ee
We claim that 
\be\label{1djrecovl1}
\lime\int_I\abs{u_\e-u}\om \,dx=0
\ee
and
\be\label{1djrecovl2}
\limsup_{\e\to 0}\int_I \abs{u_\e'}^2v_\e^2\,\om\,dx \leq \int_I \abs{u'}^2\om\,dx.
\ee
To show \eqref{1djrecovl1}, we observe that 
\bes
\lime\int_I\abs{u_\e-u}\om \,dx\leq \lime\int_{-2.5\xi_\e-\e T}^{0}\abs{u_\e-u}\om \,dx \leq \lime2\norm{u}_{L^\infty}\norm{\om}_{L^\infty}(2.5\xi_\e+\e T)=0.
\ees
We next prove \eqref{1djrecovl2}. By \eqref{0zrea} we have
\bes
\int_I \abs{u_\e'}^2 v_\e^2\,\om \,dx\leq\int_I \abs{u'}^2\om\, dx + \norm{\om}_{L^\infty}\int_{-\xi_\e-\e T}^0 \abs{u'(-x)}^2dx,
\ees
and so
\be\label{also_used_in_multi}
\limsup_{\e\to 0}\int_I \abs{u_\e'}^2 v_\e^2\,\om \,dx\leq \int_I \abs{u'}^2\om \,dx,
\ee
since $u'\in L_\om^2(I)$, and we conclude that $u'\in L^2(I)$.\\\\
On the other hand, by \eqref{0zrea} and \eqref{00zrea},
\begin{align*}
\int_{I}&\left[\frac{\e}{2}\abs{ v_\e'}^2+\frac{1}{2\e}{(v_\e-1)^2}\right]\om(x)\, dx \\
&= \int_{-3\xi_\e-2\e T}^{-\xi_\e}\left[\frac{\e}{2}\abs{ v_\e'}^2+\frac{1}{2\e}{(v_\e-1)^2}\right]\om(x)\, dx \\
&\leq \fsp{\esssup_{t\in\fsp{{-3\xi_\e-2\e T},{-\xi_\e}}}\om(t)}  \int_{-3\xi_\e-2\e T}^{-\xi_\e}\left[\frac{\e}{2}\abs{ v_\e'}^2+\frac{1}{2\e}{(v_\e-1)^2}\right]\, dx\\
&=\fsp{\esssup_{t\in\fsp{{-3\xi_\e-2\e T},{-\xi_\e}}}\om(t)}  \int_{-\xi_\e-\e T}^{\xi_\e+\e T}\left[\frac{\e}{2}\abs{ \tilde v_\e'}^2+\frac{1}{2\e}{(\tilde v_\e-1)^2}\right]\, dx.
\end{align*}
Therefore,
\begin{align*}
&\limsup_{\e\to 0}\int_{I}\left[\frac{\e}{2}\abs{ v_\e'}^2+\frac{1}{2\e}{(v_\e-1)^2}\right]\om(x)\, dx \\
\leq &\limsup_{\e\to 0}\fsp{\esssup_{t\in\fsp{{-3\xi_\e-2\e T},{-\xi_\e}}}\om(t)}\flp{\limsup_{\e\to0}\int_{-\xi_\e-\e T}^{\xi_\e+\e T}\left[\frac{\e}{2}\abs{ \tilde v_\e'}^2+\frac{1}{2\e}{(\tilde v_\e-1)^2}\right]\, dx}\\
\leq &\,\om^-(0)(1+\eta),
\end{align*}
where at the last inequality we used the definition of $\tilde v_\e$, \eqref{recovery_basic}, and \eqref{lower_bdd_2dnd}.\\\\
We conclude that
\bes
\limsup_{\e\to 0}E_{\om,\e}(u_\e,v_\e)
\leq \int_I \abs{u'}^2\om\,dx + \om^-(0)(1+\eta),
\ees
and \eqref{jump_case_woyaosisisis} follows due to the arbitrariness of $\eta$.\\\\
\underline{Step 2:} In the general case, we recall that $S_u$ is finite. We may obtain $u_\e$ and $v_\e$ by repeating the construction in Step 1 in small non-overlapping intervals centered at every point of $S_u\cap S_\om$, and by repeating the construction in Step 1 in Lemma  \ref{limsup_1d_c} in those non-overlaping intervals centered at points of $S_u\setminus S_\om$. Hence, we have 
\bes
\limsup_{\e\to 0}E_{\om,\e}(u_\e,v_\e)\leq \int_I \abs{u'}^2\om\,dx + (1+\eta)\sum_{x\in S_u}\om^-(x),
\ees
and \eqref{jump_case_woyaosisisis} follows due to the arbitrariness of $\eta$.
\end{proof}
\begin{proof}[Proof of Theorem \ref{gamma_1d_jump}]
The proof follows that of Theorem \ref{gamma_c_1d}, using Proposition \ref{liminf_part_1d} and Proposition \ref{1d_jump_reco}, in place of Proposition \ref{liminf_part_1d_c} and \ref{limsup_1d_c}, respectively.
\end{proof}
\section{The Multi-Dimensional Case}\label{T_M_D_C}
\subsection{One-Dimensional Restrictions and Slicing Properties}$\,$\\\\
Let $\mathcal S^{N-1}$ be the unit sphere in $\rn$ and let $\nu\in \mathcal S^{N-1}$ be a fixed direction. We set
\begin{align}\label{slicing_notation}
\begin{cases}
\Pi_\nu: = \flp{x\in\rn:\,\fjp{x,\nu}=0};\\
\Om_{x,\nu}^1:=\flp{t\in\R:\, x+t\nu\in\Om}\,\,\text{ for }x\in\Pi_\nu;\\
\Om_{x,\nu}:=\flp{y=x+t\nu:\,\,t\in\R}\cap\Om;\\
\Om_\nu: = \flp{x\in\Pi_\nu:\,\Om_{x,\nu}\neq \varnothing}.
\end{cases}
\end{align}
We also define the 1-d restriction function $u_{x,\nu}$ of the function $u$ as
\be\label{1_d_restric}
u_{x,\nu}(t) := u(x+t\nu),\,\,x\in\Om_\nu,\,\,t\in\Om_{x,\nu}^1.
\ee
We recall the result below from \cite{ambrosio1990approximation}, Theorem $3.3$.
\begin{theorem}\label{slice_deri_dirc}
Let $\nu\in \mathcal S^{N-1}$ be given, and assume that $u\in W^{1,2}(\Omega)$. Then, for $\hnmo$-a.e. $x\in\Omega_\nu$, $u_{x,\nu}$ belongs to $W^{1,2}(\Om_{x,\nu})$ and 
\bes
u_{x,\nu}'(t)= \fjp{\nabla u(x+t\nu),\nu}.
\ees
\end{theorem}
\begin{lemma}\label{lower_drop_energy}
Let $\om\in\mathcal W(\Om)$ and $u\in W^{1,p}_\om(\Om)$, for $p\in [1,\infty)$, be given. If $\nu\in\mathcal S^{N-1}$ and $v\in W^{1,p}(\Om)$ is nonnegative, then
\bes
\into \abs{\nabla u}^pv^p\,\om\,dx \geq  \int_{\Om_\nu}\int_{\Om^1_{x,\nu}}\abs{ u'_{x,\nu}(t)}^pv^p_{x,\nu}(t)\,\om_{x,\nu}(t)\,dtdx.
\ees
\end{lemma}
\begin{proof}
Since $\essinf_{\Om}\om\geq 1$, we have $W_\om^{1,p}(\Om)\subset W^{1,p}(\Om)$. Given $\nu\in\mathcal S^{N-1}$ and a nonnegative function $v\in W^{1,p}(\Om)$, by Fubini's Theorem and Theorem \ref{slice_deri_dirc} we have 
\begin{align*}
\into \abs{\nabla u}^pv^p\,\om\,dx&= \int_{\Om_\nu}\int_{\Om^1_{x,\nu}}\abs{\nabla u}^pv^p\,\om \,dt\,d\hnmo(x)\\
&\geq \int_{\Om_\nu}\int_{\Om^1_{x,\nu}}\abs{ \fjp{\nabla u(x+tv),\nu}}^pv^p_{x,\nu}(t)\,\om_{x,\nu}(t) \,dtd\hnmo(x) \\
&= \int_{\Om_\nu}\int_{\Om^1_{x,\nu}}\abs{u_{x,\nu}'(t)}^pv^p_{x,\nu}(t)\,\om_{x,\nu}(t)\,dtd\hnmo(x),
\end{align*}
where we used the fact that 
\bes
\abs{u_{x,\nu}'(t)}=\abs{\fjp{\nabla u(x+t\nu),\nu}}\leq \abs{\nabla u(x+t\nu)}
\ees
$\hnmo$-a.e. $x\in\Om_\nu$.
\end{proof}



\begin{proposition}\label{project_lemma_lip_h}
Let $\nu\in \mathcal S^{N-1}$ be a fixed direction, $\Gamma\subset \rn$ be such that $\hnmo(\Gamma)<\infty$, and $\mathbb P_\nu$: $\rn\to \Pi_\nu$ be a projection operator, where by \eqref{slicing_notation} $\Pi_\nu\subset \rn$ is a hyperplane in $\R^{N-1}$. Then
\be\label{project_lemma_lip_1}
\hnmo(\mathbb P_\nu(\Gamma)) \leq \hnmo(\Gamma),
\ee
and for $\hnmo$-a.e. $x\in\Pi_\nu$,
\be\label{project_lemma_lip_2}
\mathcal H^0(\Om_{x,\nu}\cap \Gamma)<+\infty.
\ee
\end{proposition}
\begin{proof}
Note that \eqref{project_lemma_lip_1} follows immediately from Theorem $7.5$ in \cite{mattila1999geometry} since $\mathbb P_\nu$ is a Lipschitz map with Lipschitz constant less or equal to one. To show \eqref{project_lemma_lip_2}, we apply co-area formula (see \cite{ambrosio2000functions}, Theorem 2.93) with $\mathbb P_\nu$ and again since $\mathbb P_\nu$ is a Lipschitz map with Lipschitz constant less or equal to one, we are done.
\end{proof}
Set $x=(x',x_N)\in\rn$, where 
\be\label{N-1fromN}
x'\in\R^{N-1}\text{ denotes the first }N-1\text{ component of }x\in\rn,
\ee
and given $u$: $\R^{N-1}\to \R$ and $G\subset \R^{N-1}$, we define the graph of $u$ over $G$ as 
\bes
F(u;G):=\flp{(x',x_N)\in\rn:\,x'\in G,\,x_N=u(x')}.
\ees
If $u$ is Lipschitz, then we call $F(u;G)$ a Lipschitz -$(N-1)$-graph.

\begin{lemma}\label{lim_project_density_1}
Let $\Gamma\subset \rn$ be a $\hnmo$-rectifiable set, and let $\mathbb P_{x,\nu_\Gamma}$: $\rn\to T_{x,\nu_\Gamma}$ be a projection operator for $x\in\Gamma$. Then
\be\label{project_density_1}
\lim_{r\to 0}\frac{\hnmo(\mathbb P_{x_0,\nu_\Gamma}(\Gamma\cap Q_{\nu_\Gamma}(x_0,r)))}{r^{N-1}}=1
\ee
for $\hnmo$-a.e. $x_0\in\Gamma$.
\end{lemma}
\begin{proof}
By Proposition \ref{project_lemma_lip_h} we have 
\be\label{limsup_project_density1}
\limsup_{r\to 0}\frac{\hnmo(\mathbb P_{x_0,\nu_\Gamma}(\Gamma\cap Q_{\nu_\Gamma}(x_0,r)))}{r^{N-1}}\leq \limsup_{r\to 0}\frac{\hnmo(\Gamma\cap Q_{\nu_\Gamma}(x_0,r))}{r^{N-1}}=1
\ee
for a.e. $x_0\in\Gamma$. By Theorem $2.76$ in \cite{ambrosio2000functions} we may write 
\bes
\Gamma = \Gamma_0\cup\bigcup_{i=1}^\infty \Gamma_i
\ees
as a disjoint union with $\hnmo(\Gamma_0)=0$, $\Gamma_i=(N_i,l_i(N_i))$ where $l_i:$ $\R^{N-1}\to \R$ is of class $C^1$ and $N_i\subset \R^{N-1}$.\\\\
Let $x_0\in \Gamma_{i_0}$ for some $i_0\in\mathbb N$ and, without loss of generality, let  $(-\nabla l_{i_0}(x'_0),1)=\nu_\Gamma(x_0)$, with $x_0$ a point of density one in $\Gamma_0$ (see Exercise 10.6 in \cite{maggi2012sets}). Up to a rotation and a translation, we may assume that $\nabla l_{i_0}(x'_0)=(0,0,\ldots, 0)\in\R^{N-1}$, $x_0=(0,0,\ldots,0)$, and $\mathbb P_{x_0,\nu_\Gamma}$: $\Gamma_{i_0}\to \R^{N-1}\times\flp{0}$. Therefore, for $r>0$ small enough, 
\bes
\Gamma_{i_0}\cap Q_{\nu_\Gamma}(x_0,r)=(\mathbb P_{x_0,\nu_\Gamma}\fsp{\Gamma_{i_0}\cap Q_{\nu_\Gamma}(x_0,r)},\,l_{i_0}((\mathbb P_{x_0,\nu_\Gamma}\fsp{\Gamma_{i_0}\cap Q_{\nu_\Gamma}(x_0,r)})')),
\ees
and by Theorem $9.1$ in \cite{mattila1999geometry} we obtain that, 
\bes
\hnmo(\Gamma_{i_0}\cap Q_{\nu_\Gamma}(x_0,r))=\int_{\mathbb P_{x_0,\nu_\Gamma}\fsp{\Gamma_{i_0}\cap Q_{\nu_\Gamma}(x_0,r)}}\sqrt{1+\abs{\nabla l_{i_0}(x')}^2}d\hnmo(x').
\ees
Since $l_{i_0}$ is of class $C^1$ and $\nabla l_{i_0}(x_0)={0}$, for $\e>0$ choose $r_\e>0$ such that $\abs{\nabla l_{i_0}(x)}<\e$ for all $0<r<r_\e$. Therefore, we have that 
\begin{align*}
\hnmo({\mathbb P_{x_0,\nu_\Gamma}\fsp{\Gamma\cap Q_{\nu_\Gamma}(x_0,r)}})&\geq \hnmo({\mathbb P_{x_0,\nu_\Gamma}\fsp{\Gamma_{i_0}\cap Q_{\nu_\Gamma}(x_0,r)}})\\
&\geq \frac{1}{\sqrt{1+\e^2}}\int_{\mathbb P_{x_0,\nu_\Gamma}\fsp{\Gamma_{i_0}\cap Q_{\nu_\Gamma}(x_0,r)}}\sqrt{1+\abs{\nabla l_{i_0}(x')}^2}dx' \\
&= \frac{1}{\sqrt{1+\e^2}}\hnmo(\Gamma_{i_0}\cap Q_{\nu_\Gamma}(x_0,r)).
\end{align*}
We obtain
\begin{align*}
\liminf_{r\to 0}\frac{\hnmo({\mathbb P_{x_0,\nu_\Gamma}\fsp{\Gamma\cap Q_{\nu_\Gamma}(x_0,r)}})}{r^{N-1}}
\geq \liminf_{r\to 0}\frac{1}{\sqrt{1+\e^2}}\frac{\hnmo(\Gamma_{i_0}\cap Q_{\nu_\Gamma}(x_0,r))}{r^{N-1}} =\frac{1}{\sqrt{1+\e^2}}.
\end{align*}
By the arbitrariness of $\e>0$, we deduce that
\bes
\liminf_{r\to 0}\frac{\hnmo({\mathbb P_{x_0,\nu_\Gamma}\fsp{\Gamma\cap Q_{\nu_\Gamma}(x_0,r)}})}{r^{N-1}}\geq 1,
\ees  
and, in view of \eqref{limsup_project_density1}, we conclude that 
\bes
\lim_{r\to 0}\frac{\hnmo({\mathbb P_{x_0,\nu_\Gamma}\fsp{\Gamma\cap Q_{\nu_\Gamma}(x_0,r)}})}{r^{N-1}}=1.
\ees
\end{proof}
\begin{lemma}\label{single_selection}
Let $Q:=(-1,1)^N$ and let $\Gamma\subset Q$ be a $\hnmo$-rectifiable set such that $\hnmo(\Gamma)<\infty$ and 
\be\label{single_selection_nonempty}
\mathcal H^0(\Gamma\cap \fsp{\flp{x'}\times (-1,1)})\geq 1
\ee
for $\hnmo$-a.e. $x'\in  (-1,1)^{N-1}$. Then there exists a $\hnmo$-measurable subset $\Gamma'\subset \Gamma$ such that
\be\label{single_selection_one}
\mathcal H^0(\Gamma'\cap \fsp{\flp{x'}\times (-1,1)})= 1.
\ee
for $\hnmo$-a.e. $x'\in (-1,1)^{N-1}$.
\end{lemma}
\begin{proof}
By Lemma \ref{project_lemma_lip_h} we have 
\be
\mathcal H^0(\Gamma'\cap \fsp{\flp{x'}\times (-1,1)})<+\infty
\ee
for $\hnmo$-a.e. $x'\in (-1,1)^{N-1}$. Thus, for $\hnmo$-a.e. $x'\in (-1,1)^{N-1}$, the set 
\bes
\Gamma_{x'}:=\Gamma\cap \fsp{\flp{x'}\times (-1,1)}
\ees
is a finite collection of singletons, hence closed, and by \eqref{single_selection_nonempty} is non-empty. Applying Corollary 1.1 in \cite{ekeland1976convex}, page $237$, we obtain a $\hnmo$ measurable subset $\Gamma'\subset \Gamma$ which satisfies \eqref{single_selection_one}.
\end{proof}

\begin{lemma}\label{slicing_single}
Let $\tau>0$ and $\eta>0$ be given. Let $u\in SBV(\Om)$ and assume that $\mathcal H^{N-1}(S_u)<\infty$. The following statements hold:
\begin{enumerate}[1.]
\item
there exist a set $S\subset S_u$ with $\mathcal H^{N-1}(S_u\setminus S)<\eta$, and a countable collection $\mathcal Q$ of mutually disjoint open cubes centered on elements of $S_u$ such that 
\bes
\bigcup_{Q\in\mathcal Q}Q\subset \Om,
\ees
and
\bes
\mathcal H^{N-1}\fsp{S\setminus \bigcup_{Q\in \mathcal Q}Q}=0;
\ees 
\item
for every $Q\in \mathcal Q$ there exists a direction vector $\nu_Q\in\mathcal S^{N-1}$ such that 
\be\label{slicing_property_single}
\mathcal H^0(S\cap Q_{x,\nu_Q})=1,
\ee
for $\mathcal H^{N-1}$ a.e. $x\in Q\cap S$;
\item
for every $Q\in\mathcal Q$, $S\cap Q$ is contained in a Lipschitz $(N-1)$- graph $\Gamma_Q$ with Lipschitz constant less than $\tau$.
\end{enumerate}
\end{lemma}

\begin{proof}
Let $\tau,\eta>0$ be given. By Theorem $2.76$ in \cite{ambrosio2000functions}, there exist countably many Lipschitz $(N-1)$- graphs $\Gamma_i\subset \rn$ such that  (up to a rotation and a translation)
\be
\Gamma_i=\flp{(x',x_N):\,\,x'\in N_i,\,\, x_N=l_i(x')} 
\ee
with $N_i\subset \R^{N-1}$, $l_i$: $\R^{N-1}\to \R$ of class $C^1$, $\abs{\nabla l_i}<\tau$ for all $i\in\mathbb N$, and
\be\label{lip_graph_0}
\hnmo\fsp{S_u\setminus \bigcup_{i=1}^\infty\Gamma_i}=0.
\ee

Without lose of generality, we assume that 
\be\label{lip_notinter}
\hnmo(\Gamma_i\cap \Gamma_{i'})=0\text{ if }i\neq {i'}\in\mathbb N\text{, and }\hnmo(\Gamma_i)>0.
\ee
We denote by $\mathcal P$ the collection of Lipschitz $(N-1)$-graphs $\Gamma_i$ in \eqref{lip_graph_0}-\eqref{lip_notinter}. By \eqref{lip_notinter}, for $\hnmo$- a.e. $x\in S_u$ there exists only one $\Gamma\in \mathcal P$ such that $x\in \Gamma$, and we denote such $\Gamma$ by $\Gamma_x$ and we write
\bes
\Gamma_x=\flp{(y',y_N):\,\,y'\in N_x\subset \R^{N-1},\,\,y_N=l_x(y')}.
\ees
For simplicity of notation, in what follows we will abbreviate $\nu_{\Gamma_x}(x)=\nu_{S_u}(x)$ by $\nu(x)$, $Q_{\nu_{S_u}}(x,r)$ by $Q(x,r)$, and $T_{x,\nu_{S_u}}$ by $T_x$.\\\\
We also note that $\hnmo(\Gamma\cap S_u)<\hnmo(S_u)<\infty$ for each $\Gamma\in\mathcal P$. Then $\hnmo$- a.e. $x$ has density 1 in $\Gamma_x\cap S_u$ (see Theorem $2.63$ in \cite{ambrosio2000functions}). Denote by $S_1$ the set of points such that $S_u$ has density $1$ at $x$ and 
\be\label{lip_xiao}
\lim_{r\to 0}\frac{\mathcal H^{N-1}(S_u\cap\Gamma_x\cap Q_{\nu_{\Gamma_x}}(x,r))}{r^{N-1}}=1. 
\ee
Then $\hnmo(S_u\setminus S_1)=0$.\\\\
Define
\be\label{density_one_function}
f_r(x) := \frac{\mathcal H^{N-1}(S_1\cap Q(x,r))}{r^{N-1}}.
\ee
Since $f_r(x)\to 1$ as $r\to 0^+$ for $x\in S_1$, by Egoroff's Theorem there exists a set $S_2\subset S_1$ such that $\mathcal H^{N-1}(S_1\setminus S_2)<\eta/4$ and $f_r\to 1$ uniformly on $S_2$. Find $r_1>0$ such that 
\bes
\frac{\mathcal H^{N-1}(S_1\cap Q(x,r))}{r^{N-1}}\geq \frac{1}{2},
\ees
i.e.,
\be\label{big_enough_dadada_cover}
{\mathcal H^{N-1}(S_1\cap Q(x,r))}\geq \frac{1}{2}{r^{N-1}}
\ee
for all $0<r<r_1$ and $x\in S_2$. Since $S_2\subset S_1$, $S_2$ is also $\hnmo$-rectifiable and so $\hnmo$ a.e. $x\in S_2$ has density one. Without loss of generality, we assume that every point in $S_2$ has density one and satisfies \eqref{project_density_1} in Lemma \ref{lim_project_density_1}.\\\\
%
%
Let $x_0\in S_2$ be given and recall \eqref{slicing_notation}. We define 
\begin{align}\label{bad_good_points}
\begin{split}
&T_b(x_0,r):=\flp{x\in Q(x_0,r)\cap T_{x_0}:\,\, \mathcal H^{0}(\fmp{Q(x_0,r)}_{x,\nu(x_0)}\cap S_2)\geq 2},\\
&T_g(x_0,r):=\flp{x\in Q(x_0,r)\cap T_{x_0}:\,\, \mathcal H^{0}(\fmp{Q(x_0,r)}_{x,\nu(x_0)}\cap S_2)=1},\\
&S_b(x_0,r):=\bigcup_{x\in T_b(x_0,r)}\fsp{S_2\cap \fmp{Q(x_0,r)}_{x,\nu(x_0)}},\\
&S_g(x_0,r):=\bigcup_{x\in T_g(x_0,r)}\fsp{S_2\cap \fmp{Q(x_0,r)}_{x,\nu(x_0)}}.
\end{split}
\end{align}
Note that 
\be\label{disjoint_slice}
T_b(x_0,r)\cap T_g(x_0,r)=\emptyset\text{ and }S_b(x_0,r)\cap S_g(x_0,r)=\emptyset,
\ee
and by Proposition \ref{project_lemma_lip_h} we have
\be\label{jiuduoyidian}
\mathcal H^{N-1}(S_g(x_0,r)) \geq \mathcal H^{N-1}(T_g(x_0,r)).
\ee
We claim that
\be\label{duoleliangbei}
\mathcal H^{N-1}(S_b(x_0,r))\geq 2\mathcal H^{N-1}(T_b(x_0,r)).
\ee
By Lemma \ref{single_selection} there exists a measurable selection $S_b^1\subset S_b(x_0,r)$ such that 
\bes
\hnmo(S_b^1(x_0,r)\cap \fmp{Q(x_0,r)}_{x,\nu(x_0)})=1
\ees
for $\hnmo$-a.e. $x\in T_b(x_0,r)$. We define 
\be\label{def_1jump_2jump}
S_b^2(x_0,r):=S_b(x_0,r)\setminus S_b^1(x_0,r).
\ee
By the definition of $S_b(x_0,r)$ in \eqref{bad_good_points}, we have 
\bes
 \mathcal H^{0}(\fmp{Q(x_0,r)}_{x,\nu(x_0)}\cap S_b^1(x_0,r))\geq 1\text{ and }  \mathcal H^{0}(\fmp{Q(x_0,r)}_{x,\nu(x_0)}\cap S_b^2(x_0,r))\geq 1
\ees
for all $x\in T_b(x_0,r)$. We observe that 
\bes
\mathcal H^{N-1}(S_b(x_0,r))=\mathcal H^{N-1}(S_b^1(x_0,r))+\mathcal H^{N-1}(S_b^2(x_0,r))\geq 2\mathcal H^{N-1}(T_b(x_0,r))
\ees
by Proposition \ref{project_lemma_lip_h} and we deduce \eqref{duoleliangbei}.\\\\
We next show that 
\be\label{still_density_1}
\lim_{r\to 0}\frac{\mathcal H^{N-1}(S_b(x_0,r))}{r^{N-1}}=0.
\ee
Indeed, since $T_{x_0}$ is the tangent hyperplane to $S_2$ at $x_0$,
\bes
T_b(x_0,r)\cup T_g(x_0,r)=\mathbb P_{x_0,\nu_{S_u}}(S_2\cap Q(x_0,r)),
\ees
and by Lemma \ref{lim_project_density_1} it follows that 
\be\label{tangent_density_1}
\lim_{r\to 0} \frac{\mathcal H^{N-1}(T_b(x_0,r)\cup T_g(x_0,r))}{r^{N-1}} =1.
\ee
On the other hand, in view of \eqref{disjoint_slice}, \eqref{jiuduoyidian}, and \eqref{duoleliangbei}, we deduce that
\begin{align*}
\mathcal H^{N-1}(S_b(x_0,r)\cup S_g(x_0,r)) &= \mathcal H^{N-1}(S_b(x_0,r))+ \mathcal H^{N-1}(S_g(x_0,r)) \\
&\geq  2\mathcal H^{N-1}(T_b(x_0,r))+ \mathcal H^{N-1}(T_g(x_0,r)).
\end{align*}
That is, 
\begin{align}
\mathcal H^{N-1}(T_b(x_0,r))
&\leq \mathcal H^{N-1}(S_b(x_0,r)\cup S_g(x_0,r)) - \fmp{\mathcal H^{N-1}(T_b(x_0,r))+ \mathcal H^{N-1}(T_g(x_0,r))}\\
&=\mathcal H^{N-1}(S_b(x_0,r)\cup S_g(x_0,r)) -\mathcal H^{N-1}(T_b(x_0,r)\cup T_g(x_0,r)).\label{more_point}
\end{align}
Since $x_0\in S_2$ has density $1$, we have
\be\label{density_one_s_2}
\lim_{r\to 0} \frac{\mathcal H^{N-1}(S_b(x_0,r)\cup S_g(x_0,r)) }{r^{N-1}} =\lim_{r\to 0}\frac{\hnmo(S_2\cap Q(x_0,r))}{r^{N-1}}=1.
\ee
In view of \eqref{tangent_density_1}, \eqref{more_point}, and \eqref{density_one_s_2}, we conclude that 
\begin{align*}
&\limsup_{r\to 0}\frac{\mathcal H^{N-1}(T_b(x_0,r))}{r^{N-1}}\\
&\leq\lim_{r\to 0}\frac{\mathcal H^{N-1}(S_b(x_0,r)\cup S_g(x_0,r)) }{r^{N-1}}-\lim_{r\to 0}\frac{\mathcal H^{N-1}(T_b(x_0,r)\cup T_g(x_0,r))}{r^{N-1}}=0,
\end{align*}
which implies that
\bes
\lim_{r\to 0} \frac{\mathcal H^{N-1}(T_b(x_0,r))}{r^{N-1}}=0.
\ees
This, together with \eqref{disjoint_slice} and \eqref{tangent_density_1}, yields
\bes
\lim_{r\to 0}\frac{\mathcal H^{N-1}(T_g(x_0,r))}{r^{N-1}}=1,
\ees
and so by \eqref{jiuduoyidian} we have
\bes
\liminf_{r\to 0}\frac{\mathcal H^{N-1}(S_g(x_0,r))}{r^{N-1}}\geq \lim_{r\to 0}\frac{\mathcal H^{N-1}(T_g(x_0,r))}{r^{N-1}}=1,
\ees
while by \eqref{density_one_s_2}
\bes
\limsup_{r\to 0}\frac{\mathcal H^{N-1}(S_g(x_0,r))}{r^{N-1}}\leq \lim_{r\to 0}\frac{\mathcal H^{N-1}(S_b(x_0,r)\cup S_g(x_0,r)) }{r^{N-1}} =1,
\ees
and we conclude that
\bes
\lim_{r\to 0}\frac{\mathcal H^{N-1}(S_g(x_0,r))}{r^{N-1}}=1.
\ees
Now, also in view of \eqref{disjoint_slice} and \eqref{density_one_s_2}, we deduce \eqref{still_density_1}.\\\\
We define, for $x\in S_2$,
\bes
g_r(x):=\frac{\mathcal H^{N-1}(S_b(x,r))}{r^{N-1}}.
\ees
By \eqref{still_density_1} we have $\lim_{r\to 0} g_r(x)=0$ for all $x\in S_2$, therefore by Egoroff's Theorem there exists a set $S_3\subset S_2$ such that 
\bes
\hnmo(S_2\setminus S_3)<\frac{\eta}{4}
\ees
and $g_r\to 0$ uniformly on $S_3$. Choose $0<r_2<r_1$ such that 
\be\label{bad_small_enough}
\frac{\mathcal H^{N-1}(S_b(x,r))}{r^{N-1}}<\frac{\eta}{16}\frac{1}{\mathcal H^{N-1}(S_u)}
\ee
for all $x\in S_3$ and $0<r<r_2$. We claim that, for $x\in S_3$ and the corresponding $\Gamma_x\in \mathcal P$,
\be\label{lip_daxiao}
\lim_{r\to 0}\frac{\hnmo\fsp{{S_g(x,r)}\setminus\fmp{S_u\cap\Gamma_x\cap Q(x,r)}}}{r^{N-1}}=0.
\ee
Suppose that 
\bes
0<\limsup_{r\to 0}\frac{\hnmo\fsp{{S_g(x,r)}\setminus\fmp{S_u\cap\Gamma_x\cap Q(x,r)}}}{r^{N-1}}=:\delta.
\ees
By \eqref{lip_xiao}, and the fact that $\Gamma_x\subset S_u$, we have that
\begin{align*}
1=&\lim_{r\to 0}\frac{\hnmo(S_u\cap Q(x,r))}{r^{N-1}}\\
=&\lim_{r\to 0}\frac{\hnmo\fsp{\fsp{\fmp{S_u\cap Q(x,r)}\setminus\fmp{S_u\cap\Gamma_x\cap Q(x,r)}}\cup \fmp{S_u\cap\Gamma_x\cap Q(x,r)}}}{r^{N-1}}\\
\geq & \limsup_{r\to 0}\frac{\hnmo\fsp{\fmp{S_g(x,r)}\setminus\fmp{S_u\cap\Gamma_x\cap Q(x,r)}}}{r^{N-1}}\\
&+\lim_{r\to 0}\frac{\hnmo{\fmp{S_u\cap\Gamma_x\cap Q(x,r)}}}{r^{N-1}}\\
= &\, \delta+1>1,
\end{align*}
which is a contradiction.\\\\
We define, for $x\in S_3$, 
\bes
h_r(x):=\frac{\hnmo\fsp{{S_g(x,r)}\setminus\fmp{S_u\cap\Gamma_x\cap Q(x,r)}}}{r^{N-1}}.
\ees
By \eqref{lip_daxiao} $\lim_{r\to 0} h_r(x)=0$ for all $x\in S_3$, therefore by Egoroff's Theorem there exists a set of $S_4\subset S_3$ such that 
\bes
\hnmo(S_3\setminus S_4)<\frac\eta4,
\ees
and $h_r\to0$ uniformly on $S_4$. Choose $0<r_3<r_2$ such that 
\be\label{out_sm_enough}
\frac{\hnmo\fsp{{S_g(x,r)}\setminus\fmp{S_u\cap\Gamma_x\cap Q(x,r)}}}{r^{N-1}}<\frac{\eta}{16}\frac{1}{\hnmo(S_u)}
\ee
for all $x\in S_4$ and $0<r<r_3$, and let
\begin{align}\label{fine_cover_S3}
\mathcal Q':=\flp{Q(x,r):\,\,x\in S_4,\,0<r<r_3}.
\end{align}
By Besicovitch's Covering Theorem we may extract a countable collection $\mathcal Q$ of mutually disjoint cubes from $\mathcal Q'$ such that 
\be
\bigcup_{Q\in \mathcal Q}Q\subset \Om\text{ and }\mathcal H^{N-1}\fsp{S_4\setminus \fsp{\bigcup_{Q\in\mathcal Q}Q}}=0.
\ee
Define
\be\label{define_S}
S:=S_4\setminus \fmp{\fsp{\bigcup_{Q\in\mathcal Q}S_b(x_Q, r_Q)}\cup\fsp{\bigcup_{Q\in\mathcal Q}\fmp{S_g(x_Q, r_Q)\setminus \fsp{S_u\cap\Gamma_{x_Q}\cap Q}}}},
\ee
where $x_Q$ is the center of cube $Q$ and $r_Q$ is the side length of $Q$. Note that the set $S$ satisfies properties 2 and 3 in the statement of Lemma \ref{slicing_single}. Finally, we show that
\bes
\hnmo(S_u\setminus S)<\eta.
\ees
Indeed, in view of \eqref{bad_small_enough} and \eqref{out_sm_enough}, and using the fact that the cubes $Q\in\mathcal Q$ are mutually disjoint, we have
\be\label{step_mid_eps1}
\hnmo\fsp{\bigcup_{Q\in\mathcal Q}S_b(x_Q,r_Q)} = \sum_{Q\in\mathcal Q}\hnmo(S_b(x_Q,r_Q))\leq \frac{\eta}{16\hnmo(S_u)}\sum_{Q\in\mathcal Q} r_Q^{N-1},
\ee
and 
\begin{multline}\label{step_mid_eps2}
\hnmo\fsp{\bigcup_{Q\in\mathcal Q}\fmp{S_g(x_Q, r_Q)\setminus \fsp{S_u\cap\Gamma_{x_Q}\cap Q}}}\\ = \sum_{Q\in\mathcal Q}\hnmo({S_g(x_Q, r_Q)\setminus \fsp{S_u\cap\Gamma_{x_Q}\cap Q}})\leq \frac{\eta}{16\hnmo(S_u)}\sum_{Q\in\mathcal Q} r_Q^{N-1}.
\end{multline}
By \eqref{big_enough_dadada_cover} we obtain
\be\label{step_mid_fin_eps}
\sum_{Q\in\mathcal Q} \frac{1}{2}r_Q^{N-1}\leq \sum_{Q\in \mathcal Q}{\mathcal H^{N-1}(S_1\cap Q)} = \hnmo\fsp{\bigcup_{Q\in \mathcal Q} S_u\cap Q}\leq\hnmo(S_u).
\ee
Using \eqref{step_mid_eps1}, \eqref{step_mid_eps2}, and \eqref{step_mid_fin_eps}, we deduce that 
\bes
\hnmo\fsp{\bigcup_{Q\in\mathcal Q}S_b(x_Q,r_Q)} \leq \frac{\eta}{8},
\ees
and
\bes
\hnmo\fsp{\bigcup_{Q\in\mathcal Q}\fmp{S_g(x_Q, r_Q)\setminus \fsp{S_u\cap\Gamma_{x_Q}\cap Q}}}\leq \frac\eta8,
\ees
and so by \eqref{define_S} we get
\bes
\hnmo(S_4\setminus S)\leq\frac{\eta}{4}.
\ees
Since $S\subset S_4\subset S_3\subset S_2\subset S_1\subset S_u$, we conclude that
\begin{align}
\begin{split}
&\hnmo(S_u\setminus S)\\
\leq& \hnmo(S_u\setminus S_1)+\hnmo(S_1\setminus S_2)+\hnmo(S_2\setminus S_3)+\hnmo(S_3\setminus S_4)+\hnmo(S_4\setminus S)\\
\leq&\frac\eta4+\frac\eta4+\frac\eta4+\frac\eta4 = \eta.
\end{split}
\end{align}
\end{proof}
\begin{lemma}\label{cont_arb_dic_slice}
Let $\om\in C(\Om)$ be nonnegative, let $\Gamma\subset\Om$ be a $\hnmo$-rectifiable set, and let $\tau\in(0,1)$ be given. Then for $\hnmo$-a.e. $x_0\in \Gamma$, there exists $r_0:=r_0(x_0)>0$ such that for each $0<r<r_0$ there exist $t_0\in (-\tau r/4, \tau r/4)$ and $0<t_{0,r}=t_{0,r}(t_0,\tau,x_0,r)<\abs{t_0}$ such that 
\begin{multline}\label{up_est_cont_lem}
\sup_{0<t\leq t_{0,r}}\frac{1}{\abs{I(t_0,t)}}\int_{I(t_0,t)}\int_{Q_{\nu_{\Gamma}}(x_0,r)\cap T_{x_0,\nu_{\Gamma}}(l)}{\om(x)}\,d\hnmo d l\\
\leq \int_{ Q_{\nu_{\Gamma}}(x_0,r)\cap \Gamma}\om(x)\,d\hnmo+(1+\om(x_0))O(\tau) r^{N-1},
\end{multline}
where $I(t_0,t):=(t_0-t, t_0+t)$,  $T_{x_0,\nu_\delta}( l):= T_{x_0,\nu_{\Gamma}}+l\nu_\Gamma$.
\end{lemma}
\begin{proof}
Fix $x_0\in \Gamma$ with density $1$ and let $\tau>0$ be given. There exists $r_1>0$ such that 
\be\label{length_a_equal_cont}
 \frac{1}{1+\tau^2}\leq\frac{\mathcal H^{N-1}(\Gamma\cap Q_{\nu_\Gamma}(x_0,r))}{r^{N-1}}\leq 1+\tau^2,
\ee
for all $0<r<r_1$. Since by continuity of $\om$ we have that 
\bes
\lim_{r\to 0}\fint_{Q_{\nu_\Gamma}(x_0,r)}\abs{ \om(x)- \om(x_0)}dx=0,
\ees
and 
\bes
\lim_{r\to 0}\fint_{Q_{\nu_\Gamma}(x_0,r)\cap \Gamma}\abs{ \om(x)- \om(x_0)}d\hnmo=0,
\ees
we may choose $0<r_2<r_1$ such that for all $0<r<r_2$
\begin{align*}
\fint_{Q_{\nu_\Gamma}(x_0,r)}\abs{ \om(x)- \om(x_0)}dx\leq \tau^2, 
\end{align*}
and
\be\label{phase_two_cont}
\int_{Q_{\nu_\Gamma}(x_0,r)\cap \Gamma} \abs{ \om(x)- \om(x_0)}d\hnmo\leq \frac{\tau}{1+\tau^2} \hnmo(Q_{\nu_\Gamma}(x_0,r)\cap\Gamma)\leq O(\tau) r^{N-1},
\ee
where we used \eqref{length_a_equal_cont}. \\\\
Therefore
\bes
\int_{-\tau r/4}^{\tau r/4} \int_{Q_{\nu_\Gamma}(x_0,r)\cap T_{x_0,\nu_\Gamma}(t)}\abs{ \om(x)- \om(x_0)}d\hnmo dt\leq \int_{Q_{\nu_\Gamma}(x_0,r)}\abs{ \om(x)- \om(x_0)}dx\leq \tau^2 r^N,
\ees
and by the Mean Value Theorem there exists a set $A\subset (-\tau r/4,\tau r/4)$ with positive 1 dimensional Lebesgue measure such that for every $t\in A$, 
\be\label{phase_four_cont}
\int_{Q_{\nu_\Gamma}(x_0,r)\cap T_{x_0,\nu_\Gamma}(t)}\abs{ \om(x)- \om(x_0)}d\hnmo \leq 2\tau r^{N-1}.
\ee
If $t_0\in A$ then we have, by the continuity of $\om$,
\be\label{lebesgue_point_line}
\lim_{t\to 0}\frac{1}{\abs{{I(t_0,t)}}}\int_{{I(t_0,t)}}\int_{Q_{\nu_\Gamma}(x_0,r)\cap T_{x_0,\nu_\Gamma}(l)}{ \om(x)}d\hnmo dl= \int_{Q_{\nu_\Gamma}(x_0,r)\cap T_{x_0,\nu_\Gamma}(t_0)}{ \om(x)}d\hnmo,
\ee
hence there exists $t_{0,r}>0$, depending on $r$, $t_0$, $\tau$, and $x_0$, such that ${I(t_0,t_{0,r})}\subset (-\tau r/2,\tau r/2)$ and
\begin{multline}\label{phase_one_cont}
\sup_{0<t\leq t_{0,r}}\frac{1}{\abs{{I(t_0,t)}}}\int_{{I(t_0,t)}}\int_{Q_{\nu_\Gamma}(x_0,r)\cap T_{x_0,\nu_\Gamma}(l)}{\om(x)}d\hnmo dl\\
\leq \int_{Q_{\nu_\Gamma}(x_0,r)\cap T_{x_0,\nu_\Gamma}(t_0)}{ \om(x)}d\hnmo + O(\tau) r^{N-1}.
\end{multline}
Moreover, since
\bes
\hnmo\fmp{{Q_{\nu_\Gamma}(x_0,r)\cap T_{x_0,\nu_\Gamma}(t_0)}}= \hnmo\fmp{Q_{\nu_\Gamma}(x_0,r)\cap T_{x_0,\nu_\Gamma}},
\ees
we have
\begin{align}\label{phase_three_cont}
\begin{split}
{\int_{Q_{\nu_\Gamma}(x_0,r)\cap T_{x_0,\nu_\Gamma}(t_0)} \om(x_0)d\hnmo} &= {\int_{Q_{\nu_\Gamma}(x_0,r)\cap T_{x_0,\nu_\Gamma}} \om(x_0)d\hnmo}\\
&= \om(x_0) r^{N-1}\leq (1+\tau^2)\int_{Q_{\nu_\Gamma}(x_0,r)\cap \Gamma} \om(x_0)d\hnmo\\
&\leq \int_{Q_{\nu_\Gamma}(x_0,r)\cap \Gamma} \om(x_0)d\hnmo + O(\tau)r^{N-1},
\end{split}
\end{align}
where in the last inequality we used \eqref{length_a_equal_cont}, the non-negativeness of $\om$.\\\\
By \eqref{phase_one_cont}, \eqref{phase_four_cont}, in this order, for every $r\leq  r_2$ there exist $t_0\in (-\tau r/4,\tau r/4)$ and $0<t_{0,r}<\abs{t_0}$, depending on $t_0$, $\tau$, $x_0$ and $r$, such that 
\begin{align}
\sup_{0<t\leq t_{0,r}}\frac{1}{\abs{{I(t_0,t)}}}&\int_{{I(t_0,t)}}\int_{Q_{\nu_\Gamma}(x_0,r)\cap T_{x_0,\nu_\Gamma(x_0)}(l)}{ \om(x)}\,d\hnmo dl\\
&\leq \om(x_0)\hnmo\fsp{Q_{\nu_\Gamma}(x_0,r)\cap T_{x_0,\nu_\Gamma(x_0)}}+O(\tau)r^{N-1}=\om(x_0)r^{N-1}+O(\tau)r^{N-1}\\
&\leq\om(x_0)(1+\tau^2)\hnmo\fsp{\Gamma\cap Q_{\nu_\Gamma}(x_0,r)}+O(\tau)r^{N-1}\\
\end{align}
where we used \eqref{length_a_equal_cont} in the last inequality. Finally, by \eqref{phase_two_cont} we conclude that 
\begin{multline}\label{easy_part_done_here}
\sup_{0<t\leq t_{0,r}}\frac{1}{\abs{{I(t_0,t)}}}\int_{{I(t_0,t)}}\int_{Q_{\nu_\Gamma}(x_0,r)\cap T_{x_0,\nu_\Gamma(x_0)}(l)}{ \om(x)}\,d\hnmo dl\\
\leq \int_{Q_{\nu_\Gamma}(x_0,r)\cap \Gamma} \om(x)\,d\hnmo+(1+\om(x_0))O(\tau) r^{N-1},
\end{multline}
as desired.
\end{proof}

\begin{proposition}\label{cont_control_jump}
Let $\om\in C(\Om)$ be nonnegative, let $\Gamma\subset\Om$ be a $\hnmo$-rectifiable set with $\hnmo(\Gamma)<+\infty$, and let $\tau\in(0,1)$ be given. Then there exist a set $S\subset\Om$ and a countable family of disjoint cubes $\mathcal F=\flp{Q_{\nu_\Gamma}(x_n,r_n)}_{n=1}^{\infty}$ with $r_n\leq \tau$, for all $n\in\mathbb N$, such that the following hold:
\begin{enumerate}[1.]
\item
$\hnmo(\Gamma\setminus S)<\tau$, $S\subset \bigcup_{n=1}^\infty Q_{\nu_\Gamma}(x_n,r_n)$;
\item
$\hnmo\fsp{S\cap Q_{\nu_\Gamma}(x_n,r)}\leq (1+\tau^2)r^{N-1}$ for all $0<r<r_n$;
\item
$S\cap Q_{\nu_\Gamma}(x_n,r_n)\subset R_{\tau/2,\nu_\Gamma}(x_n,r_n)$;
\item
if $0<\kappa<1$ then for every $n\in\mathbb N$ there exist $t^\kappa_n\in (-\kappa r_n/4,\kappa r_n/4)$ and $0<t^\kappa_{x_n,r_n}<\abs{t_n^\kappa}$, depending on $\tau$, $x_n$, and $\kappa r_n$, such that
\begin{multline}\label{upper_sup_est_cont}
\sup_{0<t\leq t^\kappa_{x_n,r_n}}\frac{1}{\abs{{I(t^\kappa_n,t)}}}\int_{{I(t^\kappa_n,t)}}\int_{Q_{\nu_\Gamma}(x_n,\kappa r_n)\cap T_{x_n,\nu_\Gamma}( l)}{\om(x)}d\hnmo d l\\
\leq \int_{\Gamma\cap Q_{\nu_\Gamma}(x_n,\kappa r_n)}\om(x)d\hnmo+  (1+\om(x_n))O(\tau) (\kappa r_n)^{N-1},
\end{multline}
where ${I(t^\kappa_n,t)}:=(t_n^\kappa-t, t_n^\kappa+t)$. 
\end{enumerate}
\end{proposition}
\begin{proof}
Let $\tau\in(0,1)$ and $\kappa\in(0,1)$ be given. Since $\hnmo(\Gamma)<\infty$, there exists  $S_1\subset \Gamma$ such that $\hnmo(\Gamma\setminus S_1)<\tau/3$, $S_1$ is compact and contained in a finite union of $(N-1)$-Lipschitz graphs $\Gamma_i$, $i=1,\ldots, M$, with Lipschitz constants less than $\tau/(2\sqrt{N})$. \\\\
Moreover, since $\hnmo$ a.e. $x\in S_1$ a point of density one, by Egorov's Theorem, we may find $S_2\subset S_1$ such that $\hnmo(S_1\setminus S_2)<\tau/3$ and there exists $r_1>0$ such that for all $0<r<r_1$ and $x\in S_2$,
\bes
\hnmo\fsp{S_1\cap Q_{\nu_\Gamma}(x,r)}\leq (1+\tau^2)r^{N-1}.
\ees
Let $L_i:=S_2\cap \Gamma_i$ and without lose of generality we assume that $L_i$ are mutually disjoint. Let $L_i'\subset L_i$ be such that 
\bes
\hnmo(L_i\setminus L_i')<\frac{\tau}{3}\frac{1}{2^i}\text{ and }d_{ij}:=\operatorname{dist}(L_i',L_j')>0
\ees
for $i\neq j$. We observe that
\bes
\hnmo\fsp{S_2\setminus \bigcup_{i=1}^M L_i'}<\frac\tau3\text{ and }d:=\min_{i\neq j}\flp{d_{ij}}>0.
\ees 
Define
\bes
S:=\bigcup_{i=1}^M L_i'.
\ees
We claim that there exists $0<r_2<\min\flp{\tau^2,d/2,r_1}$ such that for every $0<r<r_2$ and every $x$, $y\in S$ with $\abs{x-y}<\sqrt{N}r$ we have
\bes
 S\cap Q_{\nu_\Gamma}(x,r)\subset R_{\tau/2,\nu_\Gamma}(x,r),
\ees
where we are using the notation introduced in Notation \ref{multi_not}. Indeed, to verify this inclusion, we write (up to a rotation)
\bes
S\cap Q_{\nu_\Gamma}(x,r)=\flp{(y', \l_x(y')):\,\, y\in T_{x,\nu_\Gamma}\cap Q_{\nu_\Gamma}(x,r)}\subset \Gamma_{x}
\ees
where $y'$. Assuming, without loss of generality, that $x=0$ and $l_x(0)=0$, we have for all $y\in T_{0,\nu_\Gamma}\cap Q_{\nu_\Gamma}(0,r)$ 
\bes
\abs{l_0(y)}\leq\norm{\nabla l_0}_{L^\infty}\leq\frac{1}{2}\tau r
\ees
because for every $y\in S\cap Q_{\nu_\Gamma}(0,r)$ we have $\abs{y}<\sqrt{N}r$.\\\\
Next, for $\hnmo$-a.e. $x\in S$ we may find $r_2(x)>0$ such that $Q_{\nu_\Gamma}(x,r_3)\subset \Om$ and $\kappa r_2(x)\leq r_0(x)$ where $r_0(x)$ is determined in Lemma \ref{cont_arb_dic_slice}. Let $\bar r_0(x):=\min\flp{r_1,r_2(x)}$. The collection
\bes
\mathcal F':=\flp{Q_{\nu_\Gamma}(x,r):\,\,x\in S, \,\,r<\bar r_0(x)}
\ees
is a fine cover for $S$, and so by Besicovitch's Covering Theorem we may obtain a countable sub-collection $\mathcal F\subset \mathcal F'$ with pairwise disjoint cubes such that
\bes
S\subset \bigcup_{Q_{\nu_\Gamma}(x_n,r_n)\in\mathcal F}Q_{\nu_\Gamma}(x_n,r_n).
\ees
For each $Q_{\nu_\Gamma}(x_n,r_n)\in\mathcal F$ we apply Lemma \ref{cont_arb_dic_slice} to obtain $t^\kappa_n\in (-\kappa r_n/4,\kappa r_n/4)$ and $t^\kappa_{x_n,r_n}>0$, depending on $t_{x_n}^\kappa$, $\tau$, $\kappa r_n$, and $x_n$, such that \eqref{upper_sup_est_cont} hold.\\\\
Finally, we observe that 
\bes
\hnmo(\Gamma\setminus S)\leq \hnmo(\Gamma\setminus S_1)+\hnmo(S_1\setminus S_2)+\hnmo(S_2\setminus S)\leq \tau,
\ees
and this completes the proof.
\end{proof}

\begin{proposition}\label{cont_ready_coro_limsup}
Let $\om\in C(\Om)$ be nonnegative, let $\Gamma\subset\Om$ be $\hnmo$-rectifiable with $\hnmo(\Gamma)<+\infty$, and let $\tau\in(0,1)$ be given. There exists a $\hnmo$-rectifiable set $S\subset \Gamma$ and a countable family of disjoint cubes $\mathcal F=\flp{Q_{\nu_\Gamma}(x_n,r_n)}_{n=1}^{\infty}$ with $r_n< \tau$ such that the following hold:
\begin{enumerate}[1.]
\item
\be\label{cont_sdfsfsdfs}
\hnmo(\Gamma\setminus S)<\tau,\,S\subset \bigcup_{n=1}^{\infty}Q_{\nu_\Gamma}(x_n,r_n),\text{ and }S\cap Q_{\nu_\Gamma}(x_n,r_n)\subset R_{\tau/2,\nu_\Gamma}(x_n,r_n);
\ee
\item
\be\label{cont_sdfsfsdfs2}
\hnmo\fsp{S\cap Q_{\nu_\Gamma}(x_n,r_n)}\leq (1+\tau^2)r_n^{N-1};
\ee
\item
for $n\neq m$
\be\label{disjoint_col_c}
\operatorname{dist}(Q_{\nu_\Gamma}(x_n,r_n),Q_{\nu_\Gamma}(x_m,r_m))>0;
\ee
\item
\be\label{cont_coro_finitesum}
\sum_{n=1}^{+\infty} r_n^{N-1}\leq 4\hnmo(\Gamma)
\ee
\item
for each $n\in\mathbb N$ there exist $t_{n}\in (-\tau { r_n}/4,\tau{ r_n}/4)$ and $0<t_{x_n,r_n}<\abs{t_n}$, depending on $\tau$, $r_n$, and $x_n$, such that $T_{x_n,\nu_\Gamma}(t_n\pm t_{x_n,r_n})\subset R_{\tau/2,\nu_\Gamma}(x_n,{ r_n})$ and
\begin{multline}\label{upper_sup_ready_limsup_cont}
\sup_{0<t\leq t_{x_n,r_n}}\frac{1}{\abs{{I(t_n,t)}}}\int_{{I(t_n,t)}}\int_{Q_{\nu_\Gamma}(x_n,{r_n})\cap T_{x_n,\nu_\Gamma}( l)}{\om(x)}d\hnmo d l\\
\leq \int_{S\cap Q_{\nu_\Gamma}(x_n,{ r_n})}\om d\hnmo+  (1+\om(x_n))\tau { r_n}^{N-1},
\end{multline}
where $I(t_n,t):=(t_n-t, t_n+t)$.
\end{enumerate}
\end{proposition}

\begin{proof}
We apply items 1, 2, and 3 in Proposition \ref{cont_control_jump} to obtain a countable collection $\flp{Q_{\nu_\Gamma}(x_n,r'_n)}_{n=1}^\infty$ and a set $S'\subset \Gamma$ such that 
\bes
\hnmo(\Gamma\setminus S')<\frac\tau2,\,\,S' \subset\bigcup_{n=1}^\infty Q_{\nu_\Gamma}(x_n,r'_n),\,\, S'\cap Q_{\nu_\Gamma}(x_n,r'_n)\subset R_{\tau/2,\nu_\Gamma}(x_n,r_n'),
\ees
and 
\bes
\hnmo\fsp{S\cap Q_{\nu_\Gamma}(x_n,r)}\leq (1+\tau^2)r^{N-1}
\ees
for all $0<r<r'_n$. Find $0<\kappa<1$ such that 
\bes
\hnmo\fsp{S'\setminus \bigcup_{n=1}^\infty Q_{\nu_\Gamma}(x_n,\kappa r'_n)}<\frac\tau2,
\ees
and let
\bes
S:=S'\cap \fsp{\bigcup_{n=1}^\infty Q_{\nu_\Gamma}(x_n,\kappa r'_n)}.
\ees
Then
\bes
S\subset \bigcup_{n=1}^\infty Q_{\nu_\Gamma}(x_n,\kappa r'_n)
\ees
and
\bes
\hnmo(\Gamma\setminus S)\leq \hnmo(\Gamma\setminus S')+\hnmo(S'\setminus S)\leq \frac\tau2+\frac\tau2=\tau.
\ees
Note that the collection $\flp{Q_{\nu_\Gamma}(x_n,\kappa r'_n)}_{n=1}^{\infty}$ satisfies \eqref{disjoint_col_c}. Next, we apply item $4$ in Proposition \ref{cont_control_jump} with such $\kappa>0$ to find $t_n^\kappa$, $t_{x_n,r'_n}^\kappa$ such that \eqref{upper_sup_est_cont} holds. It suffices to set $r_n:=\kappa r_n'$, $t_n:=t_n^\kappa$, and $t_{x_n,r_n}:=t_{x_n,r_n'}^\kappa$.\\
\end{proof}

\subsection{The Case $\om\in \mathcal W(\Om)\cap C(\Om)$}\label{ATAc_nd}$\,$\\\\
Consider the functionals
\bes
E_{\om,\e}(u,v):=\int_\Omega v^2\abs{\nabla u}^2\om\,dx +\int_{\Omega}\left[\e\abs{\nabla v}^2+\frac{1}{4\e}{(v-1)^2}\right]\om\,dx
\ees
for $(u,v)\in W_\om^{1,2}(\Om)\times W^{1,2}(\Om)$, and let
\bes
E_\om(u):=\int_\Omega \abs{\nabla u}^2\om\,dx+ \int_{S_u} {\om(x)}\,d\mathcal H^{N-1}, 
\ees
be defined for $u\in GSBV_\om(\Om)$.
\begin{theorem}\label{ATc_n_case}
Let $\om\in\mathcal W(\Om)\cap C(\Om)\cap L^\infty(\Om)$ be given. Let $\mathcal{E}_{\om,\e}$: $L^1_\om(\Om)\times L^1(\Om)\rightarrow[0,+\infty]$ be defined by
\begin{align}\label{esd_weighted}
\mathcal{E}_{\om,\e}(u,v):=
\begin{cases}
E_{\om,\e}(u,v)&\text{if}\,\,(u,v)\in W_\om^{1,2}(\Om)\times W^{1,2}(\Om), \,0\leq v\leq 1,\\
+\infty &\text{otherwise}.
\end{cases}
\end{align}
Then the functionals $\mathcal{E}_{\om,\e}$ $\Gamma$-converge, with respect to the $L^1_\om\times L^1$ topology, to the functional
\begin{align}\label{etd_weighted}
\mathcal{E}_\om(u,v):=
\begin{cases}
E_\om(u)&\text{if}\,\,u\in GSBV_\om(\Om)\text{ and }v=1\,\,a.e.,\\
+\infty &\text{otherwise}.
\end{cases}
\end{align}

\end{theorem}
Theorem \ref{ATc_n_case} will be proved in two propositions.
\begin{proposition}\label{liminf_part_c}\emph{($\Gamma$-$\liminf$)}
For $\om\in \mathcal W(\Om)\cap C(\Om)$ and $u\in L^1_\om(\Om)$, let
\begin{align*}
E_\om^-(u):=\inf&\flp{\liminf_{\e\to 0} E_{\om,\e}(u_\e,v_\e):\right.\\
&\,\,\,\left.(u_\e,v_\e)\in W^{1,2}_\om(\Om)\times W^{1,2}(\Om), \,u_\e\to u \text{ in }L^1_\om,\, v_\e\to1 \text{ in }L^1,\,0\leq v_\e\leq 1}.
\end{align*}
We have 
\bes
E_\om^-(u)\geq E_\om(u).
\ees
\end{proposition}

\begin{proof}
Without loss of generality, we assume that $M:=E_\om^-(u)<\infty$. Let $\seqe{(u_\e,v_\e)}\subset W^{1,2}_\om(\Om)\times W^{1,2}(\Om)$ be such that 
\be\label{lower_conv_seq}
u_\e\to u\text{ in }L^1_\om,\,\, v_\e\to 1\text{ in }L^1(\Om),\text{ and}\lim_{\e\to 0}E_{\om,\e}(u_\e,v_\e) = E_\om^-(u)<\infty.
\ee
Since $\inf_{x\in\Om}\om(x)\geq 1$, we have 
\bes
\liminf_{\e\to 0} E_{1,\e}(u_\e,v_\e)\leq \liminf_{\e\to 0} E_{\om,\e}(u_\e,v_\e)<\infty,
\ees
and by \cite{ambrosio1990approximation} we deduce that 
\be\label{famitsfinite}
u\in GSBV(\Om)\text{ and }\hnmo(S_u)<\infty.
\ee
We prove separately that 
\be\label{first_part_ATCw_m}
\liminf_{\e\to 0} \int_\Om\abs{\nabla u_\e}^2v_\e\,\om\,dx\geq \into \abs{\nabla u}^2\om\,dx,
\ee
and 
\be\label{second_part_ATCw_m}
\liminf_{\e\to 0} \into \fsp{\e\abs{\nabla v_\e}^2+\frac{1}{4\e}(1-v_\e)^2}\om\,dx\geq \int_{S_u}\om(x)d\hnmo.
\ee
Let $A$ be an open subset of $\Omega$. Fix $\nu\in \mathcal S^{N-1}$, and define $A_{x,\nu}$, $A^1_{x,\nu}$, and $A_\nu$ as in \eqref{slicing_notation}. For $K\in\R^+$, set $u_K:=K\wedge u\vee -K$, and observe that, by Fubini's Theorem,
\begin{align*}
\liminf_{\e\to 0} \int_A\abs{\nabla u_\e}^2v_\e^2\,\om\,dx
&\geq \liminf_{\e\to 0}\int_{A_\nu}\int_{A^1_{x,\nu}}  \abs{(u_\e)'_{x,\nu}}^2(v_\e)_{x,\nu}^2\,\om_{x,\nu}\,dt \,d\hnmo(x)\\
&\geq \int_{A_\nu}  \liminf_{\e\to 0}  \int_{A^1_{x,\nu}} \abs{(u_\e)'_{x,\nu}}^2(v_\e)_{x,\nu}^2\,\om_{x,\nu}\,dt\,d\hnmo(x)\\
&\geq\int_{A_\nu}\int_{A^1_{x,\nu}} \abs{u'_{x,\nu}}^2\om_{x,\nu}\,dt\,d\hnmo(x) \\
&\geq \int_{A_\nu}\int_{A^1_{x,\nu}} \abs{(u_K)'_{x,\nu}}^2\om_{x,\nu}\,dt\,d\hnmo(x),
\end{align*}
where in the first inequality we used Lemma \ref{lower_drop_energy}, in the second inequality we used Fatou's Lemma, and in the third inequality we used \eqref{use_la_nonsm1}. Since $u_K\in L^\infty(\Om)\cap SBV_\om(\Om)\subset L^\infty(\Om)\cap SBV(\Om)$, we may apply  Theorem 2.3 in \cite{ambrosio1990approximation} to $u_K$ to obtain
\bes
\liminf_{\e\to 0} \int_A\abs{\nabla u_\e}^2v_\e^2\,\om\,dx\geq \int_{A_\nu}\int_{A^1_{x,\nu}} \abs{(u_K)'_{x,\nu}}^2\om_{x,\nu}\,dt\,d\hnmo(x)\geq \int_A \abs{\fjp{\nabla u_K(x),\nu}}^2\om\,dx.
\ees
Letting $K\to\infty$ and using Lebesgue Monotone Convergence Theorem we have 
\be\label{lfslicecont}
\liminf_{\e\to 0} \int_A\abs{\nabla u_\e}^2v_\e^2\,\om\,dx\geq \int_A \abs{\fjp{\nabla u(x),\nu}}^2\om\,dx.
\ee
Let
\bes
\phi_n(x):= \abs{\fjp{\nabla u(x),\nu_n}}^2\om\text{ for }\mathcal L^N \text{-a.e. }x\in\Om,
\ees
where $\seqn{\nu_n}$ is a dense subset of $\mathcal S^{N-1}$, and let
\bes
\mu(A):=\liminf_{\e\to 0} \int_A\abs{\nabla u_\e}^2v_\e^2\,\om\,dx.
\ees
Then $\mu$ is a positive function, super-additivity on open sets $A$, $B$, with disjoint closures, since 
\begin{align*}
\mu(A\cup B)&=\liminf_{\e\to 0} \int_{A\cup B}\abs{\nabla u_\e}^2v_\e^2\,\om\,dx =\liminf_{\e\to 0} \fsp{  \int_A\abs{\nabla u_\e}^2v_\e^2\,\om\,dx+\int_B\abs{\nabla u_\e}^2v_\e^2\,\om\,dx}\\
&\geq \liminf_{\e\to 0}\int_A\abs{\nabla u_\e}^2v_\e^2\,\om\,dx+\liminf_{\e\to 0}\int_B\abs{\nabla u_\e}^2v_\e^2\,\om\,dx=\mu(A)+\mu(B).
\end{align*}
Hence by Lemma $15.2$ in \cite{braides2002gamma}, together with \eqref{lfslicecont}, we conclude \eqref{first_part_ATCw_m}.\\\\
Now we prove \eqref{second_part_ATCw_m}. Assume first that $\om\in L^\infty(\Om)$. For any open set $A\subset \Om$ and $\nu\in\mathcal S^{N-1}$, by Fubini's Theorem and Fatou's Lemma we have
\begin{align}\label{liminf_cont_later_slice}
\begin{split}
\liminf_{\e\to 0}& \int_A \fsp{{\e}\abs{\nabla v_\e}^2+\frac{1}{4\e}(1-v_\e)^2}\om\,dx\\
&\geq\liminf_{\e\to 0}\int_{A_\nu}\int_{A^1_{x,\nu}}  \left[\e\abs{(v_\e)'_{x,\nu}}^2+\frac{1}{4\e}{\fsp{1-(v_\e)_{x,\nu}}^2}\right]\om_{x,\nu}\,dtd\hnmo(x)\\
&\geq 
\int_{A_\nu} { \liminf_{\e\to 0}  \int_{A^1_{x,\nu}}\left[\e\abs{ (v_\e)'_{x,\nu}}^2+\frac{1}{4\e}{\fsp{1-(v_\e)_{x,\nu}}^2}\right]\om_{x,\nu}\,dt }d\hnmo(x)\\
&\geq  \int_{A_\nu}\fmp{\sum_{t\in S_{u_{x,\nu}}\cap A^1_{x,\nu}} {\om_{x,\nu}(t)}}d\hnmo(x),
\end{split}
\end{align}
where the last inequality follows from \eqref{contpart2}.\\\\
%
%
Next, given arbitrary $\tau>0$ and $\eta>0$ we choose a set $S\subset S_u$ and a collection $\mathcal Q$ of mutually disjoint cubes according to Lemma \ref{slicing_single} with respect to $S_u$. Fix one such cube $Q_{\nu_{S}}(x_0,r_0)\in \mathcal Q$. By Lemma \ref{slicing_single} we have
\bes
\hnmo(\fmp{Q_{\nu_{S}}(x_0,r_0)}_{x,\nu_S}\cap S)=1
\ees
for $\hnmo$-a.e. $x\in Q_{\nu_{S}}(x_0,r_0)\cap S$, and $Q_{\nu_{S}}(x_0,r_0)\cap S\subset \Gamma_{x_0}$ such that, up to a rotation and a translation,
\be\label{C1curvec}
\Gamma_{x_0}=\flp{(y',l_{x_0}(y')):\,\, y\in T_{x_0,\nu_{S}}\cap Q_{\nu_{S}}(x_0,r_0)}\text{ and }\norm{\nabla l_{x_0}}_{L^\infty}<\tau,
\ee
where $y'$ denotes the first $N-1$ components  of $y\in T_{x_0,\nu_{S}}\cap Q_{\nu_{S}}(x_0,r_0)$.\\\\
In \eqref{liminf_cont_later_slice} set $A=Q_{\nu_{S}}(x_0,r_0)$ and $\nu=\nu_S(x_0)$ and, using the same notation as in the proof of Lemma \ref{slicing_single}, we obtain
\begin{align}\label{single_int_tmp}
\int_{\fmp{Q_{\nu_{S}}(x_0,r_0)}_{\nu_S(x_0)}} &\fsp{\sum_{t\in S_{u_{x,\nu_S(x_0)}}\cap \fmp{Q_{\nu_{S}}(x_0,r_0)}_{x,\nu_S(x_0)}} \om_{x,\nu_S(x_0)}(t)}d\mathcal H^{N-1}(x) \\
&\geq\int_{T_g(x_0,r_0)} \fsp{\sum_{t\in S_{u_{x,\nu_S(x_0)}}\cap \fmp{Q_{\nu_{S}}(x_0,r_0)}_{x,\nu_S(x_0)}\cap S} \om_{x,\nu_{S}(x_0)}(t)}d\mathcal H^{N-1}(x)\\
&=\int_{T_g(x_0,r_0)}  \om(x)\,d\mathcal H^{N-1}(x)=\int_{T_g(x_0,r_0)} \om(x',l_{x_0}(x'))d\mathcal L^{N-1}(x'),
\end{align}
where the first inequality is due to the positivity of $\om$ and the last equality is because $ Q_{\nu_{S}}(x_0,r_0)\cap S\subset \Gamma_{x_0}$ which is defined in \eqref{C1curvec}.\\\\
Next, by Theorem 9.1 in \cite{maggi2012sets} and since $\om\in C(\Om)$, we have that
\begin{align}
\int_{Q_{\nu_{S}}(x_0,r_0)\cap S}\om \,d\hnmo&=\int_{T_{x_0,\nu_{S}}\cap Q_{\nu_{S}}(x_0,r_0)}\om(x',l_{x_0}(x'))\sqrt{1+\abs{\nabla l_{x_0}(x')}^2}dx'\\
&\leq \sqrt{1+\tau^2} \int_{T_{x_0,\nu_{S}}\cap Q_{\nu_{S}}(x_0,r_0)}\om(x',l_{x_0}(x'))dx',
\end{align}
which, together with \eqref{single_int_tmp}, yields
\begin{multline}\label{single_int}
\int_{\fmp{Q_{\nu_{S}}(x_0,r_0)}_{\nu_S(x_0)}} \fsp{\sum_{t\in S_{u_{x,\nu_S(x_0)}}\cap \fmp{Q_{\nu_{S}}(x_0,r_0)}_{x,\nu_S(x_0)}} \om_{x,\nu_S(x_0)}(t)}\,d\mathcal H^{N-1}(x)\\
\geq \frac1{\sqrt{1+\tau^2}}\int_{Q_{\nu_{S}}(x_0,r_0)\cap S}\om \,d\hnmo.
\end{multline}
Since cubes in $\mathcal Q$ are pairwise disjoint and $\hnmo(S\setminus \cup_{Q\in\mathcal Q}Q)=0$, by \eqref{liminf_cont_later_slice}, \eqref{single_int_tmp}, and \eqref{single_int} we have
\begin{align*}
\liminf_{\e\to0}& \int_{\cup_{Q\in\mathcal Q} Q} \left[\e\abs{\nabla v_\e}^2+\frac{1}{4\e}{(v_\e-1)^2}\right]\om\,dx\\&\geq  \sum_{Q\in\mathcal Q} \liminf_{\e\to 0}  \int_{Q} \left[\e\abs{\nabla v_\e}^2+\frac{1}{4\e}{(v_\e-1)^2}\right]\om\,dx\\
&\geq\frac{1}{\sqrt{1+\tau^2}} \sum_{Q\in\mathcal Q} \int_{S\cap Q} {\om}\,d\mathcal H^{N-1} =\frac{1}{\sqrt{1+\tau^2}}\int_S {\om}\,d\mathcal H^{N-1}\\
&\geq \frac{1}{\sqrt{1+\tau^2}}\fsp{\int_{S_u} {\om}\,d\mathcal H^{N-1} - \norm{\om}_{L^\infty}\eta}.
\end{align*}
Therefore
\begin{align*}
&\liminf_{\e\to0} \into \left[\e\abs{\nabla v_\e}^2+\frac{1}{4\e}{(v_\e-1)^2}\right]\om\,dx\\
&\geq \liminf_{\e\to0} \int_{\cup_{Q\in\mathcal Q} Q}\left[\e\abs{\nabla v_\e}^2+\frac{1}{4\e}{(v_\e-1)^2}\right]\om\,dx\geq
\frac{1}{\sqrt{1+\tau^2}}\fsp{ \int_{S_u} {\om(x)}d\mathcal H^{N-1} - \norm{\om}_{L^\infty}\eta},
\end{align*}
and \eqref{second_part_ATCw_m} follows from the arbitrariness of $\eta$ and $\tau$, and the fact that $\eta$ and $\tau$ are independent. \\\\
We now remove the assumption that $\om\in L^\infty$. Define for each $k>0$,
\bes
\om_k(x):=
\begin{cases}
\om&\text{ if }\om\leq k,\\
k&\text{ otherwise. }
\end{cases}
\ees
We have 
\begin{align*}
\liminf_{\e\to0}& \into \left[\e\abs{\nabla v_\e}^2+\frac{1}{4\e}{(v_\e-1)^2}\right]\om\,dx\\
&\geq \liminf_{\e\to0} \into \left[\e\abs{\nabla v_\e}^2+\frac{1}{4\e}{(v_\e-1)^2}\right]\om_k\,dx\geq \int_{S_u} {\om_k(x)}d\mathcal H^{N-1},
\end{align*}
and we conclude 
\bes
\liminf_{\e\to0} \into \left[\e\abs{\nabla v_\e}^2+\frac{1}{4\e}{(v_\e-1)^2}\right]\om\,dx\geq \int_{S_u} {\om(x)}d\mathcal H^{N-1}
\ees
by letting $k\nearrow \infty$ and using Lebesgue Monotone Convergence Theorem.
\end{proof}
\begin{proposition}\label{limsup_n_c}\emph{($\Gamma$-$\limsup$)}
For $\om\in \mathcal W(\Om)\cap C(\Om)\cap L^\infty(\Om)$ and $u\in L^1_\om(\Om)\cap L^\infty(\Om)$, let
\begin{align*}
E_\om^+(u):=\inf&\flp{\limsup_{\e\to 0} E_{\om,\e}(u_\e,v_\e):\right.\\
&\left.\,\,\,(u_\e,v_\e)\in W^{1,2}_\om(\Om)\times W^{1,2}(\Om), u_\e\to u\text{ in }L^1_\om,\, v_\e\to1\text{ in }L^1,\,0\leq v_\e\leq 1}.
\end{align*}
We have
\be\label{cont_upper_bdd_need}
E_\om^+(u)\leq E_\om(u).
\ee
\end{proposition}
\begin{proof}
If $E_\om(u)=\infty$ then there is nothing to prove. Assume that $E_\om(u)<+\infty$ so that by Lemma \ref{compact_energy} we have that $u\in GSBV_\om(\Om)$ and $\hnmo(S_u)<\infty$. By assumption $u\in L^\infty(\Om)$, thus $u\in SBV_\om(\Om)$.\\\\
Let $\tau\in(0,2/9)$ be given. Apply Proposition \ref{cont_ready_coro_limsup} to $\om$ and $\Gamma=S_u$ to obtain a set $S_\tau\subset S_u$, a countable collection $\mathcal F_\tau=\seqn{Q_{\nu_{S_u}}(x_n,r_n)}$ of mutually disjoint cubes with $r_n<\tau$, and corresponding 
\be\label{range_tn_c}
t_n\in(-\tau r_n/4,\tau r_n/4)
\ee
and $t_{x_n,r_n}$ so that items 1-5 in Proposition \ref{cont_ready_coro_limsup} hold. Extract a finite collection $\mathcal T_\tau=\flp{{Q_{\nu_{S_u}}(x_n,r_n)}}_{n=1}^{M_\tau}$ from $\mathcal F_\tau$ with $M_\tau>0$ large enough such that 
\bes
\hnmo\fmp{S_\tau\setminus \bigcup_{n=1}^{M_\tau} Q_{\nu_{S_u}}(x_n,r_n)}<\tau,
\ees
and we define
\be\label{what_left_F_cont}
F_\tau:=S_\tau\cap \fmp{\bigcup_{n=1}^{M_\tau} Q_{\nu_{S_u}}(x_n,r_n)},
\ee
which implies that 
\be\label{qu_left_F_cont}
\hnmo\fsp{S_u\setminus F_\tau}\leq \hnmo(S_u\setminus S_\tau)+\hnmo(S_\tau\setminus F_\tau)<2\tau.
\ee
Let $U_n$ be the part of $Q_{\nu_{S_u}}(x_n, r_n)$ which lies between $T_{x_n,\nu_{S_u}}(\pm \tau r_n)$, $U_n^+$ be the part above $T_{x_n,\nu_{S_u}}(\tau r_n)$ and $U_n^-$ be the part below $T_{x_n,\nu_{S_u}}(-\tau r_n)$. Moreover, let $U_{t_n}^+$ be the part of $U_n$ which lies above $T_{x_n,\nu_{S_u}}(t_n)$, and $U_{t_n}^-$ be the part below $T_{x_n,\nu_{S_u}}(t_n)$. \\\\
We claim that if $x\in U_{t_n}^\pm$,
\be\label{upper_right_place}
x\pm2\,{\operatorname{dist}\fsp{x,T_{x_n,\nu_{S_u}}(\pm\tau r_n)}\nu_{S_u}(x_n)}\in U_n^\pm\subset Q_{\nu_{S_u}}(x_n,r_n)\setminus R_{\tau/2,\nu_{S_u}}(x_n,r_n).
\ee
Let $x\in U_{t_n}^+$ (the case in which $x\in U_{t_n}^-$ can be handled similarly), we need to prove that 
\be
\tau r_n <\dist\fsp{x+2\,{\operatorname{dist}(x,T_{x_n,\nu_{S_u}}(\tau r_n)) \nu_{S_u}(x_n)},\, T_{x_n,\nu_{Su}}}<\frac{r_n}2.
\ee
Note that
\bes
\operatorname{dist}\fsp{x,T_{x_n,\nu_{S_u}}(\tau r_n)}=\tau r_n-(x-\mathbb P_{x_n, \nu_{S_u}}(x))\nu_{S_u}(x_n),
\ees
and since $x\in U_{t_n}^+$, we have that
\be\label{pos_dis_proj}
\fsp{x-\mathbb P_{x_n, \nu_{S_u}}(x)}\nu_{S_u}(x_n)\in (t_n,\tau r_n).
\ee
Hence, together with \eqref{range_tn_c}, we observe that
\begin{align}\label{range_x_proj}
\tau r_n&\leq2\operatorname{dist}\fsp{x,T_{x_n,\nu_{S_u}}(\tau r_n)}+ \fsp{x-\mathbb P_{x_n, \nu_{S_u}}(x)}\nu_{S_u}(x_n) \\
&=2\tau r_n-(x-\mathbb P_{x_n, \nu_{S_u}}(x))\nu_{S_u}(x_n)\leq 2\tau r_n-\fsp{-\frac{\tau r_n}{4}}= \frac94\tau r_n<\frac12 r_n.
\end{align}
From the definition of projection operator $\mathbb P_{x_n,\nu_{S_u}}$ we have
\bes
\mathbb P_{x_n, \nu_{S_u}}\fmp{x+2\,\dist\fsp{T_{x_n,\nu_{S_u}}(\tau r_n)}\nu_{S_u}(x_n)} = \mathbb P_{x_n, \nu_{S_u}}(x),
\ees
and hence
\begin{align}\label{big_proj_com}
&\operatorname{dist}\fsp{x+2\,{\operatorname{dist}(x,T_{x_n,\nu_{S_u}}(\tau r_n)) \nu_{S_u}(x_n)},\, T_{x_n,\nu_{Su}}}\\
&=\fsp{\fmp{x+2\,{\operatorname{dist}\fsp{x,T_{x_n,\nu_{S_u}}(\tau r_n)}}\nu_{S_u}(x_n)}\right.\\
&\,\,\,\,\,\,\left.-\mathbb P_{x_n, \nu_{S_u}}\fmp{x+2\,\dist\fsp{T_{x_n,\nu_{S_u}}(\tau r_n)}\nu_{S_u}(x_n)} } \nu_{S_u}(x_n)\\
&=\fsp{2\,{\operatorname{dist}\fsp{x,T_{x_n,\nu_{S_u}}(\tau r_n)}}\nu_{S_u}(x_n)}\nu_{S_u}(x_n) \\
&\,\,\,\,\,\,+ \fsp{x-\mathbb P_{x_n, \nu_{S_u}}\fmp{x+2\,\dist\fsp{T_{x_n,\nu_{S_u}}(\tau r_n)}\nu_{S_u}(x_n)} } \nu_{S_u}(x_n)\\
&=2\operatorname{dist}\fsp{x,T_{x_n,\nu_{S_u}}(\tau r_n)} + \fsp{x-\mathbb P_{x_n, \nu_{S_u}}(x)}\nu_{S_u}(x_n),
\end{align}
and by \eqref{range_x_proj} we conclude \eqref{upper_right_place}.\\\\
We define $\bar u_\tau$ in $Q_{\nu_{S_u}}(x_n,r_n)$ as follows (see Figure \ref{figue_approx_u_simplex_jump_cont}): 
\be\label{approx_u_simplex_jump_cont}
\bar u_\tau(x):=
\begin{cases}
u(x) & \text{ if }x\in U_n^+\cup U_n^-\\
u\fsp{x+2\,{\operatorname{dist}(x,T_{x_n,\nu_{S_u}}(\tau r_n)) \nu_{S_u}(x_n)}}&\text{ if }x\in U_{t_n}^+,\\
u\fsp{x-2\,{\operatorname{dist}(x,T_{x_n,\nu_{S_u}}(-\tau r_n)) \nu_{S_u}(x_n)}}&\text{ if }x\in U_{t_n}^-,
\end{cases}
\ee
and 
\be\label{only_diff_set}
\bar u_\tau(x):=u(x)\text{ if }x\in \Om\setminus \fsp{\bigcup_{n=1}^{M_\tau} Q_{\nu_{S_u}}(x_n,r_n)}.
\ee
\begin{figure}[h]
\includegraphics[width=0.8\textwidth]{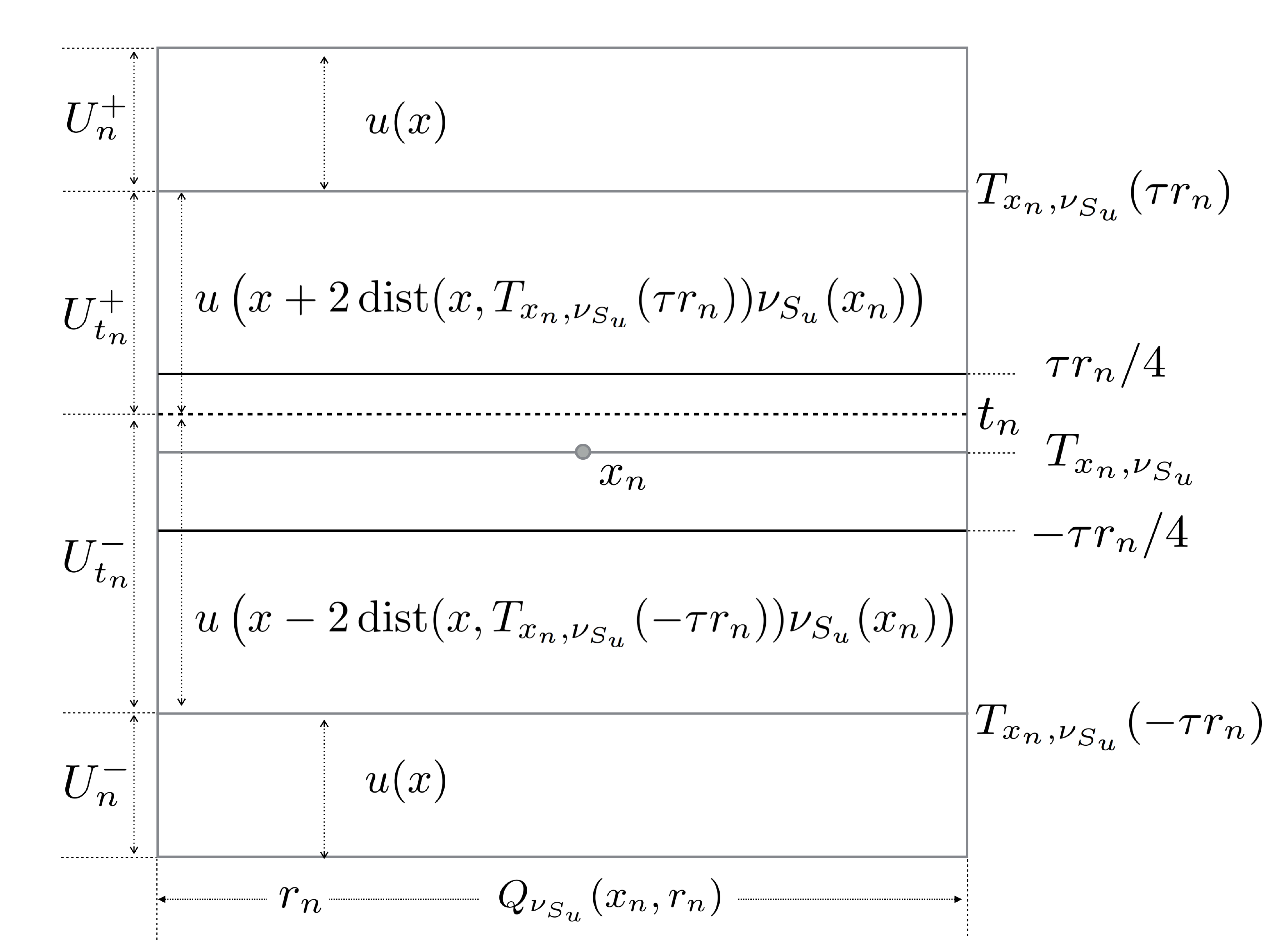}
\caption{Construction of $\bar u_\tau(x)$ in (\ref{approx_u_simplex_jump_cont})}\label{figue_approx_u_simplex_jump_cont}
\end{figure}

We observe that, as $\tau\to 0$, and since $0<r_n<\tau$,
\begin{align}\label{diff_goes_z}
\begin{split}
\mathcal L^N&\fsp{\flp{x\in\Om,\,\, u(x)\neq \bar u_\tau(x)}}=\mathcal L^N\fsp{\bigcup_{n=1}^{M_\tau}U_{t_n}^+\cup U_{t_n}^-}\\
&\leq \sum_{n=1}^{M_\tau}\mathcal L^N\fsp{U_{t_n}^+\cup U_{t_n}^-}= \sum_{n=1}^{M_\tau} \fsp{r_n^{N-1}2\tau r_n}\leq 2\tau^2\sum_{n=1}^{M_\tau} r_n^{N-1}\leq  8\tau^2\hnmo(S_u)\to 0,
\end{split}
\end{align}
where the last inequality follows from \eqref{cont_coro_finitesum}. Moreover, using the same computation, we deduce that 
\be\label{bdiff_goes_z}
\mathcal L^N \fsp{\bigcup_{n=1}^{M_\tau} Q_{\nu_{S_u}}(x_n,r_n)}\leq \tau \sum_{n=1}^{M_\tau}r_n^{N-1}\leq 4\tau\hnmo(S_u)=O(\tau)\to 0.
\ee
Hence, in view of \eqref{diff_goes_z}, we have 
\be\label{bar_u_ready_cont}
\bar u_\tau\to u\,\text{ and }\nabla \bar u_\tau\to\nabla u \text{ in measure},
\ee
and, since in $U_{t_n}^+\cup U_{t_n}^-$ $\bar u_\tau$ is the reflection of $u$ from $Q_{\nu_{S_u}}(x_n,r_n)\setminus U_{t_n}^+\cup U_{t_n}^-$, we observe that 
\begin{align}\label{bar_u_ready_cont2}
\begin{split}
\into \abs{\nabla \bar u_\tau}^2\om\,dx&\leq \int_{\Om\setminus \flp{u(x)\neq \bar u_\tau(x)}}\abs{\nabla u}^2\om\,dx+\norm{\om}_{L^\infty}\int_{\flp{u(x)\neq \bar u_\tau(x)}}\abs{\nabla \bar u_\tau}^2dx\\
&\leq \int_{\Om\setminus \flp{u(x)\neq \bar u_\tau(x)}}\abs{\nabla u}^2\om\,dx+2\norm{\om}_{L^\infty}\sum_{n=1}^{M_\tau}\int_{Q_{\nu_{S_u}}(x_n,r_n)}\abs{\nabla u}^2dx\\
&= \int_{\Om\setminus \flp{u(x)\neq \bar u_\tau(x)}}\abs{\nabla u}^2\om\,dx+2\norm{\om}_{L^\infty}\int_{\bigcup_{n=1}^{M_\tau}Q_{\nu_{S_u}}(x_n,r_n)}\abs{\nabla u}^2dx\\
&\leq \into\abs{\nabla u}^2\om dx + O(\tau)
\end{split}
\end{align}
where the last inequality follows from \eqref{bdiff_goes_z} and from the fact that because $E_{1}(u)\leq E_\om(u)<+\infty$, $\nabla u$ is $L^2$ integrable. Moreover, in view of \eqref{bar_u_ready_cont} and by Lebesgue Dominated Convergence Theorem we conclude that
\be\label{L1omcont}
\lim_{\tau\to 0}\into\abs{\bar u_\tau -u}\om\,dx\leq \norm{\om}_{L^\infty}\lim_{\tau\to 0}\into\abs{\bar u_\tau -u}dx=0
\ee
because $\norm{\bar u_\tau}_{L^\infty}\leq\norm{u}_{L^\infty}<+\infty$.\\\\
For simplicity of notation, in the rest of the proof of this lemma we shall abbreviate $Q_{\nu_{S_u}}(x_n, r_n)$ by $Q_n$ and $T_{x_n,\nu_{S_u}}$ by $T_{x_n}$. Note that the jump set of $\bar u_\tau$ is contained by (recall item $4$ in Proposition \ref{cont_ready_coro_limsup})
\begin{enumerate}[1.]
\item
\be\label{major_part_m_jump_cont}
\bigcup_{n=1}^{M_\tau}\fmp{T_{x_n}(t_n)\cap Q_n};
\ee
\item
$$\bigcup_{n=1}^{M_\tau}\partial Q_n\cap \overline{U_n};$$
\item
$S_u\setminus F_\tau$, where $F_\tau$ is defined in \eqref{what_left_F_cont}.\\
\end{enumerate}
The contributions to $S_u$ from $2$ and $3$ are negligible. To be precise,
\begin{align}\label{AT_nimenzhongyuyouyongle}
&\hnmo\fmp{(S_u\setminus F_\tau)\cup\fsp{\bigcup_{n=1}^{M_\tau}\partial Q_n\cap \overline{U_n}}}\\
&\leq \hnmo (S_u\setminus F_\tau)+\sum_{n=1}^{M_\tau}\hnmo(\partial Q_n\cap \overline{U_n})\leq 2\tau+C\tau \sum_{n=1}^\infty r_n^{N-1} \tau\leq O(\tau).
\end{align}
where we used \eqref{cont_sdfsfsdfs}, \eqref{cont_coro_finitesum}, \eqref{qu_left_F_cont}, and the fact that 
\be\label{AT_nimenzhongyuyouyongle_2}
\sum_{n=1}^{M_\tau}\hnmo(\partial Q_n\cap \overline{U_n})\leq 2\tau \sum_{n=1}^{M_\tau}r_n^{N-1}\leq 8\tau \hnmo(S_{u}).
\ee
Hence, again by \eqref{cont_coro_finitesum},
\bes
\hnmo(S_{\bar u_\tau})\leq \sum_{n=1}^{M_\tau} \hnmo(T_{x_n}\cap Q_n)+O(\tau)\leq \sumn r_n^{N-1}+O(\tau)<\infty.
\ees
By \eqref{disjoint_col_c}, let $a_\tau$ denote a quarter of the minimum distance between all cubes in $\mathcal T_\tau$. Let $\e>0$ be such that
\be\label{finite_small_pos}
\e^2+\sqrt{\e}<<\frac{1}{4}\min\flp{\tau,\,a_\tau,\, t_{x_n,r_n} \text{ for } 1\leq n\leq M_\tau}.
\ee
Hence, by item 5 in Proposition \ref{cont_ready_coro_limsup} we have 
\be\label{finite_small_pos2}
\e^2+\sqrt{\e}< t_{x_n,r_n}<\abs{t_n}<\frac{1}4\tau r_n < r_n.
\ee
We set
\bes
u_{\tau,\e}:=(1-\vp_{\e})\bar u_\tau,
\ees
where $\vp_\e$ is such that
\begin{align*}
\vp_\e\in C_c^\infty(\Om;[0,1]),&&\,\,\vp_\e\equiv1\text{ on }(\overline{S_{\bar u_\tau}})_{{\e^2}/4}, &&\text{and  } \vp_\e\equiv0\text{ in }\Om\setminus(\overline{S_{\bar u_\tau}})_{{\e^2}/2}.
\end{align*}
Since $\bar u_\tau\in W^{1,2}(\Om\setminus \overline{S_{\bar u_\tau}})$, we have $\seqe{u_{\tau,\e}}\subset W^{1,2}(\Om)$ because $(1-\vp_\e)(x)= 0$ if $x\in (\overline{S_{\bar u_\tau}})_{\e^2/4}$. Moreover, $\seqe{u_{\tau,\e}}\subset W^{1,2}_\om(\Om)$ and, using Lebesgue Dominated Convergence Theorem and \eqref{L1omcont},
\be\label{recov_goes_1}
\lim_{\tau\to 0}\lime\into \abs{u_{\tau,\e}-u}\om=0
\ee
because $\om\in L^\infty$, $u\in L^\infty$, and $\vp_\e\to 0$ a.e.\\\\
Consider the sequence $\seqe{v_{\tau,\e}}\subset W^{1,2}(\Om)$ given by
\be\label{recovery_two_side_AT}
v_{\tau,\e}(x):=\tilde v_{\e}\circ d_\tau(x)
\ee
where $d_\tau(x):=\operatorname{dist}(x, S_{\bar u_\tau})$ and $\tilde v_{\e}$ is defined by
\be\label{benshen}
\tilde v_\e(t):=
\begin{cases}
0&\text{ if }t\leq \e^2,\\
1-e^{-\frac{1}{2\sqrt{\e}}}&\text{ if }t>\sqrt{\e}+\e^2,
\end{cases}
\ee
and for $\e^2\leq t\leq \sqrt{\e}+\e^2$ we define $\tilde v_\e$ as the solution of the differential equation
\be\label{lo_or_swith}
{\tilde v_\e'(t)}=\frac{1}{2\e}(1-\tilde v_\e(t)).
\ee
with initial condition $\tilde v_\e(\e^2)=0$. An explicit computation shows that 
\be\label{benzun}
\tilde v_\e(t)=-e^{-\frac{1}{2}\frac{t-\e^2}{\e}}+1
\ee
for $\e^2\leq t\leq \sqrt\e+\e^2$ and $\tilde v_\e(\sqrt{\e}+\e^2)=1-\exp\fsp{-1/2\sqrt{\e}}$, and we remark that
\be\label{ele_multi_ago1}
\lime \frac1{\e}e^{-\frac1{2\sqrt\e}}=0,
\ee
and
\be\label{ele_multi_ago2}
-\frac{d}{dt}\fsp{\frac12\fsp{1-\tilde v_\e(t)}^2}=\fsp{1-\tilde v_\e(t)}\tilde v_\e'(t)\geq 0.
\ee
Next, since $\abs{\nabla d_\tau}=1$ a.e. (see \cite{federer2014geometric}, Section $3.2.34$), we have $\seqe{v_{\tau,\e}}\subset W^{1,2}(\Om)$, $0\leq v_{\tau,\e}\leq 1-\operatorname{exp}(-1/2\sqrt{\e})$,
and 
\be\label{recov_goes_o}
v_{\tau,\e}\to 1\text{ in }L^1 \text{ as }\e\to0
\ee
by Lebesgue Dominated Convergence Theorem since $v_{\tau,\e}\to 1$ a.e. by \eqref{benzun}. By \eqref{bar_u_ready_cont2} and since if $\vp_\e(x)\neq 0$ then $d_\tau(x)<\e^2/2$ and so $v_{\tau,\e}(x)=0$,
\be\label{cont_maj_est2}
\into \abs{\nabla u_{\tau,\e}}^2v^2_{\tau,\e}\,\om\,dx \leq  \into \abs{\nabla \bar u_\tau}^2\om\,dx\leq \into\abs{\nabla u}^2\om\,dx+O(\tau).
\ee
Next we prove that
\be\label{cont_maj_est}
\into\fmp{\e\abs{\nabla v_{\tau,\e}}^2+\frac{1}{4\e}(1-v_{\tau,\e})^2}\om\,dx\leq \int_{S_u}\om\,d\hnmo+{O(\e)+O(\tau)}.
\ee
Define
\begin{align*}
L_n:={T_{x_n}\cap Q_n},&&L_n(\e):=\fsp{T_{x_n}\cap Q_n}_{\e},
\end{align*}
and observe that, using Fubini's Theorem,
\begin{align*}
&\int_{L_n(\e^2+\sqrt{\e})}\fmp{\e\abs{\nabla v_{\tau,\e}}^2+\frac{1}{4\e}(1-v_{\tau,\e})^2}\om\,dx \\
&= \int_{\e^2}^{\e^2+\sqrt{\e}}\fmp{\e\abs{\tilde v'_\e(l)}^2+\frac{1}{4\e}(1-\tilde v_\e(l))^2}\int_{\flp{d_\tau(y) = l}\cap L_n(\e^2+\sqrt{\e})}\om(y)\,d\hnmo(y)\,dl\\
&\,\,\,\,\,\,\,+\frac{1}{4\e}\int_{L_n(\e^2)}\om(x)dx,
\end{align*}
where the latter term in the right hand side is of the order $O(\e)$. Next, in view of \eqref{lo_or_swith}, using integration by parts, we have that
\begin{align}\label{whatsleftover}
\begin{split}
\int_{\e^2}^{\e^2+\sqrt{\e}}&\fmp{\e\abs{\tilde v'_\e(l)}^2+\frac{1}{4\e}(1-\tilde v_\e(l))^2}\int_{\flp{d_\tau(y) = l}\cap L_n(\e^2+\sqrt{\e})}\om(y)\,d\hnmo(y)\,dl\\
=& \int_{\e^2}^{\e^2+\sqrt{\e}}\frac{1}{2\e}(1-\tilde v_\e(l))^2\int_{\flp{d_\tau(y) = l}\cap L_n(\e^2+\sqrt{\e})}\om(y)\,d\hnmo(y)\,dl\\
=&- \int_{\e^2}^{\e^2+\sqrt{\e}}\frac{1}{2\e}\frac{d}{dt}\fmp{(1-\tilde v_\e(l))^2}\int_{\flp{d_\tau(y)\leq l}\cap L_n(\e^2+\sqrt{\e})}\om(y)\,dy\,dl \\
&+ \mathcal A^n_\om({{\e^2}+\sqrt\e}) - \mathcal A^n_\om({{\e^2}}),
\end{split}
\end{align}
where 
\begin{align*}
\mathcal A^n_\om(t):=\frac1{2\e}(1-\tilde v_\e(t))^2\int_{\flp{d_\tau(x)\leq t}\cap L_n(\e^2+\sqrt{\e})}\om(y)\,dy.
\end{align*}
By \eqref{finite_small_pos2} and \eqref{ele_multi_ago1} we have
\begin{align}\label{leftover_sm}
\mathcal A^n_\om(\e^2+\sqrt{\e}) &= \frac1{2\e}e^{-\frac1{2\sqrt\e}}\int_{\flp{d_\tau(x)\leq \e^2+\sqrt{\e}}\cap L_n(\e^2+\sqrt{\e})}\om(x)\,dx\\
&\leq \frac1{2\e}e^{-\frac1{2\sqrt\e}}\norm{\om}_{L^\infty}\mathcal L^N(L_n(\e^2+\sqrt{\e}))= \frac1{2\e}e^{-\frac1{2\sqrt\e}}\norm{\om}_{L^\infty}\mathcal L^N((T_{x_n}\cap Q_n)_{\e^2+\sqrt{\e}})\\
&\leq \frac1{2\e}e^{-\frac1{2\sqrt\e}}\norm{\om}_{L^\infty}\fmp{2(\e^2+\sqrt{\e})[r_n+(\e^2+\sqrt{\e})]^{N-1}} \leq O(\e)r_n^{N-1}.
\end{align}
We write
\begin{multline*}
- \int_{\e^2}^{\e^2+\sqrt{\e}}\frac{1}{2\e}\frac{d}{dt}\fmp{(1-\tilde v_\e(l))^2}\int_{\flp{d_\tau(y)\leq l}\cap L_n(\e^2+\sqrt{\e})}\om(y)\,dy\,dl \\
=\int_{\e^2}^{\e^2+\sqrt{\e}}2l\fsp{-\frac{1}{2\e}\frac{d}{dt}\fmp{(1-\tilde v_\e(l))^2}} \fmp{\frac{1}{2l}\int_{\flp{d_\tau(y)\leq l}\cap L_n(\e^2+\sqrt{\e})}\om(x)\,dx}dl.
\end{multline*}
Recalling the notation from Proposition \ref{cont_ready_coro_limsup} and the fact that $\om(x_n)\leq \norm{\om}_{L^\infty}$, we have 
\begin{align*}
\begin{split}
\frac{1}{2l}\int_{\flp{d_\tau(y)\leq l}\cap L_n(\e^2+\sqrt{\e})}\om(x)\,dx&\leq \sup_{t\leq \e^2+\sqrt{\e}}\fsp{\frac{1}{\abs{I(t_n,t)}}\int_{I(t_n,t)}\int_{Q(x_n,r_n)\cap T_{x_n}(l)}{\om(x)}d\hnmo dl}\\
&\leq \int_{S_\tau\cap Q(x_n,r_n)}\om(x)\,d\hnmo+  O(\tau) r_n^{N-1}.
\end{split}
\end{align*}
where by \eqref{finite_small_pos} we could use \eqref{upper_sup_ready_limsup_cont} in the last inequality. Therefore, by \eqref{ele_multi_ago2}
\begin{align}\label{cont_key_step}
\begin{split}
-&\int_{\e^2}^{\e^2+\sqrt{\e}}\frac{1}{2\e}\frac{d}{dt}\fmp{(1-\tilde v_\e(l))^2}\int_{\flp{d_\tau(y)\leq l}\cap L_n(\e^2+\sqrt{\e})}\om(x)\,dx\,dl\\
\leq& 2\fsp{\int_{\e^2}^{\e^2+\sqrt{\e}}-\frac{1}{2\e}\frac{d}{dt}\fmp{(1-\tilde v_\e(l))^2}l\,dl}\fsp{\int_{S_\tau\cap Q(x_n,r_n)}\om(x)\,d\hnmo+  O(\tau) r_n^{N-1}}.
\end{split}
\end{align}
A new integration by parts and by using \eqref{benzun} yields
\begin{align*}
&\int_{\e^2}^{\e^2+\sqrt{\e}}-\frac{1}{2\e}\frac{d}{dt}\fmp{(1-\tilde v_\e(l))^2}l\,dl\\
&=\int_{\e^2}^{\e^2+\sqrt{\e}}\frac{1}{2\e}{(1-\tilde v_\e(l))^2}dl-\frac{1}{2\e}(\e^2+\sqrt{\e})(1-\tilde v_\e(\e^2+\sqrt{\e}))^2+\frac{\e^2}{2\e}(1-\tilde v_\e(\e^2))^2\\
&\leq \int_{\e^2}^{\e^2+\sqrt{\e}}2\e\abs{\tilde v_\e'(l)}^2dl+\frac{\e^2}{2\e}(1-\tilde v_\e(\e^2))^2=\frac12\fsp{1-e^{-\frac1{\sqrt{\e}}}}+\frac{\e}{2}(1-\tilde v_\e(\e^2))^2\\
&\leq \frac12+\frac12\e,
\end{align*}
which, together with \eqref{cont_key_step} and \eqref{cont_sdfsfsdfs2}, gives
\begin{align}\label{whatsleftover2}
-&\int_{\e^2}^{\e^2+\sqrt{\e}}\frac{1}{2\e}\frac{d}{dt}\fmp{(1-\tilde v_\e(l))^2}\int_{\flp{d_\tau(y)\leq l}\cap L_n(\e^2+\sqrt{\e})}\om(x)\,dx\,dl\\
&\leq \int_{S_\tau\cap Q(x_n,r_n)}\om(x)\,d\hnmo+  O(\tau) r_n^{N-1} + \e\norm{\om}_{L^\infty}\hnmo(S_\tau\cap Q(x_n,r_n))+\e O(\tau)r_n^{N-1}\\
&\leq \int_{S_\tau\cap Q(x_n,r_n)}\om(x)\,d\hnmo+O(\tau)r_n^{N-1}+O(\e)O(\tau)r_n^{N-1}.
\end{align}
Hence, in view of \eqref{whatsleftover}, \eqref{leftover_sm}, \eqref{whatsleftover2}, and since $A_\om(\e^2)\geq 0$, we obtain that
\begin{multline}\label{muti_cont_phase1}
\int_{L_n(\e^2+\sqrt{\e})}\fmp{\e\abs{\nabla v_{\tau,\e}}^2+\frac{1}{4\e}(1-v_{\tau,\e})^2}\om\,dx\\
\leq\int_{S_\tau\cap Q(x_n,r_n)}\om(x)d\hnmo +O(\tau)r_n^{N-1}+O(\e)O(\tau) r_n^{N-1}+O(\e)r_n^{N-1}.
\end{multline}
Next we define
\begin{align*}
L_0:={\fsp{S_u\setminus F_\tau}\cup\fsp{\bigcup_{n=1}^{M_\tau}\partial Q_n\cap \overline{U}_n}}&&\text{and}&&L_0(\e):=\fmp{\fsp{S_u\setminus F_\tau}\cup\fsp{\bigcup_{n=1}^{M_\tau}\partial Q_n\cap \overline{U_n}}}_{\e}.
\end{align*}
Since $\om\in L^\infty(\Om)$, we have 
\begin{multline*}
\int_{L_0(\e^2+\sqrt{\e})}\fmp{\e\abs{\nabla v_{\tau,\e}}^2+\frac{1}{4\e}(1-v_{\tau,\e})^2}\om\,dx\\
\leq \norm{\om}_{L^\infty}\int_{L_0(\e^2+\sqrt{\e})}\fmp{\e\abs{\nabla v_{\tau,\e}}^2+\frac{1}{4\e}(1-v_{\tau,\e})^2}\,dx,
\end{multline*}
and we note that the term 
\bes
\int_{L_0(\e^2+\sqrt{\e})}\fmp{\e\abs{\nabla v_{\tau,\e}}^2+\frac{1}{4\e}(1-v_{\tau,\e})^2}\,dx
\ees
is the recovery sequence constructed in \cite{ambrosio1990approximation}, page 1034, Added in Proof. Therefore, recalling that by assumption that $u\in SBV_\om(\Om)\cap L^\infty(\Om)\subset SBV(\Om)\cap L^\infty(\Om)$ and invoking Proposition $5.1$ and $5.3$ in \cite{ambrosio1990approximation} and calculation within, we conclude that 
\begin{align*}
\limsup_{\e\to 0}\int_{L_0(\e^2+\sqrt{\e})}\fmp{\e\abs{\nabla v_{\tau,\e}}^2+\frac{1}{4\e}(1-v_{\tau,\e})^2}\,dx&\leq \limsup_{\e\to 0}\frac{\hnmo\fsp{x\in\Om:\,\,\operatorname{dist}(x, L_0)<\e}}{2\e}\\
&\leq \hnmo(L_0).
\end{align*}
Thus, 
\begin{align}\label{we_left_a_bit_cont}
\begin{split}
\int_{L_0(\e^2+\sqrt{\e})}\fmp{\e\abs{\nabla v_{\tau,\e}}^2+\frac{1}{4\e}(1-v_{\tau,\e})^2}\om\,dx&\leq \norm{\om}_{L^\infty}\fsp{\hnmo(L_0)+O(\e)}\\
&\leq {O(\tau)+O(\e)}.
\end{split}
\end{align}
Furthermore, by \eqref{benshen}
\be\label{whatsleftover3}
\int_{\Om\setminus ({S_{\bar u_\tau}})_{\e^2+\sqrt{\e}}} \fmp{\e\abs{\nabla v_{\tau,\e}}^2+\frac{1}{4\e}(1-v_{\tau,\e})^2}\om\,dx\leq \frac1{4\e}e^{-\frac1{\sqrt{\e}}}\norm{\om}_{L^\infty}\mathcal L^N(\Om)\leq O(\e),
\ee
where in the last inequality we used \eqref{ele_multi_ago1}.\\\\
Since cubes in $\mathcal T_\tau$ are pairwise disjoint, in view of \eqref{muti_cont_phase1}, \eqref{we_left_a_bit_cont}, and \eqref{whatsleftover3} we have that
\begin{align*}
&\into \fmp{\e\abs{\nabla v_{\tau,\e}}^2+\frac{1}{4\e}(1-v_{\tau,\e})^2}\om\,dx\\
=&\int_{(S_{\bar u_\tau})_{\e^2+\sqrt{\e}}}\fmp{\egamma\abs{\nabla v_{\tau,\e}}^2+\frac{1}{4\egamma}(1-v_{\tau,\e})^2}\om\,dx \\
&+ \int_{\Om\setminus (S_{\bar u_\tau})_{\e^2+\sqrt{\e}}} \fmp{\e\abs{\nabla v_{\tau,\e}}^2+\frac{1}{4\e}(1-v_{\tau,\e})^2}\om\,dx\\
\leq& \int_{L_0(\e^2+\sqrt{\e})}\fmp{\e\abs{\nabla v_{\tau,\e}}^2+\frac{1}{4\e}(1-v_{\tau,\e})^2}\om\,dx\\
&+ \sum_{n=1}^{M_\tau}\int_{L_n(\e^2+\sqrt{\e})}\fmp{\e\abs{\nabla v_{\tau,\e}}^2+\frac{1}{4\e}(1-v_{\tau,\e})^2}\om\,dx+O(\e)\\
\leq& O(\e)+O(\tau)+\sum_{n=1}^{M_\tau} \fsp{\int_{S_\tau\cap Q(x_n,r_n)}\om(x)d\hnmo + \fmp{O(\e)+O(\tau)+O(\e)O(\tau)} r_n^{N-1}}\\
\leq& \int_{\bigcup_{n=1}^{M_\tau}(S_\tau\cap Q(x_n,r_n))}\om(x)\,d\hnmo+  \fmp{O(\e)+O(\tau)+O(\e)O(\tau)}\sum_{n=1}^{M_\tau} r_n^{N-1}+O(\tau)+O(\e)\\
\leq& \int_{S_u} \om(x)\,d\hnmo +{O(\e)+O(\tau)}+O(\e)O(\tau),
\end{align*}
where in the last inequality we used \eqref{cont_coro_finitesum}, and this concludes the proof of \eqref{cont_maj_est}. Hence, also in view of \eqref{cont_maj_est2} and \eqref{cont_maj_est}, for each $\tau>0$, we may choose $\e(\tau)$ such that 
\bes
\into \abs{\nabla u_{\tau,\e(\tau)}}^2v^2_{\tau,\e(\tau)}\,\om\,dx \leq \into\abs{\nabla u}^2\om\,dx+O(\tau),
\ees
and
\bes
\into \fmp{\e\abs{\nabla v_{\tau,\e(\tau)}}^2+\frac{1}{4\e}(1-v_{\tau,\e(\tau)})^2}\om\,dx\leq \int_{S_u}\om(x) d\hnmo+O(\tau),
\ees
and we thus constructed a recovery sequence $\flp{(u_\tau,v_\tau)}_{\tau>0}$ given by  
\bes
u_\tau:=u_{\tau,\e(\tau)}\text{ and }v_\tau:=v_{\tau,\e(\tau)}
\ees
which satisfies \eqref{cont_upper_bdd_need} and by \eqref{recov_goes_1} and \eqref{recov_goes_o} we have
\bes
\norm{u_{\tau,\e(\tau)}-u}_{L^1_\om}<\tau\text{ and }\norm{v_{\tau,\e(\tau)}-v}_{L^1}<\tau.
\ees
Hence, we proved Proposition \ref{limsup_n_c}.
\end{proof}

\begin{proof}[Proof of Theorem \ref{ATc_n_case}]
The $\liminf$ inequality follows from Lemma \ref{liminf_part_c}. On the other hand, for any given $u\in GSBV_\om$ such that $E_\om(u)<\infty$, we have, by Lebesgue Monotone Convergence Theorem,
\bes
E_\om(u)=\lim_{K\to\infty} E_\om(K\wedge u\vee -K),
\ees
and a diagonal argument, together with Proposition \ref{limsup_n_c}, concludes the proof.
\end{proof}
\subsection{The Case $\om\in \mathcal W(\Om)\cap SBV(\Om)$}\label{ATA_nd}$\,$\\\\
Consider the functionals
\bes
E_{\om,\e}(u,v):=\int_\Omega v^2\abs{\nabla u}^2\om\,dx +\int_{\Omega}\left[\e\abs{\nabla v}^2+\frac{1}{4\e}{(v-1)^2}\right]\om\,dx,
\ees
for $(u,v)\in W_\om^{1,2}(\Om)\times W^{1,2}(\Om)$, and
\bes
E_\om(u):=\int_\Omega \abs{\nabla u}^2\om\,dx+ \int_{S_u} {\om^-(x)}\,d\mathcal H^{N-1}
\ees
defined for $u\in GSBV_\om(\Om)$.

\begin{theorem}\label{AT_n_case}
Let $\om\in \mathcal W(\Om)\cap SBV(\Om)\cap L^\infty(\Om)$ be given. Let $\mathcal{E}_{\om,\e}$: $L^1_\om(\Om)\times L^1(\Om)\rightarrow[0,+\infty]$ be defined by
\begin{align}\label{esd_weighted_jump}
\mathcal{E}_{\om,\e}(u,v):=
\begin{cases}
E_{\om,\e}(u,v)&\text{if}\,\,(u,v)\in W_\om^{1,2}(\Om)\times W^{1,2}(\Om), \,0\leq v\leq 1,\\
+\infty &\text{otherwise}.
\end{cases}
\end{align}
Then the functionals $\mathcal{E}_{\om,\e}$ $\Gamma$-converge, with respect to the $L^1_\om\times L^1$ topology, to the functional
\begin{align}\label{etd_weighted_jump}
\mathcal{E}_\om(u,v):=
\begin{cases}
E_\om(u)&\text{if}\,\,u\in GSBV_\om(\Om)\text{ and }v=1\,\,a.e.,\\
+\infty &\text{otherwise}.
\end{cases}
\end{align}
\end{theorem}
We start by proving the $\Gamma$-$\liminf$.
\begin{proposition}\label{liminf_part}\emph{($\Gamma$-$\liminf$)}
For $\om\in\mathcal W(\Om)\cap SBV(\Om)$ and $u\in L^1_\om(\Om)$, let
\begin{align*}
E_\om^-(u):=\inf&\flp{\liminf_{\e\to 0} E_{\om,\e}(u_\e,v_\e):\right.\\
&\left.\,\,\,(u_\e,v_\e)\in W^{1,2}_\om(\Om)\times W^{1,2}(\Om), u_\e\to u, v_\e\to1\text{ in }L^1_\om\times L^1,\,0\leq v_\e\leq 1}.
\end{align*}
We have 
\bes
E_\om^-(u)\geq E_\om(u).
\ees
\end{proposition}

\begin{proof}
Without lose of generality, we assume that $E_\om^-(u)<+\infty$. The proof of this lemma uses the same arguments of the proof of Proposition \ref{liminf_part_c} until the beginning of \eqref{liminf_cont_later_slice}, and we obtain
\bes
\liminf_{\e\to 0}\into \abs{\nabla u_\e}^2v_\e^2\om\,dx\geq \into \abs{\nabla u}^2\om\,dx.
\ees

 By applying Proposition \ref{liminf_part_1d} to the last inequality of \eqref{liminf_cont_later_slice}, we have 
\bes
\liminf_{\e\to 0}\int_A \fsp{{\e}\abs{\nabla v_\e}^2+\frac{1}{4\e}(1-v_\e)^2}\om\,dx\geq  \int_{A_\nu}\sum_{t\in S_{u_{x,\nu}}\cap A_{x,\nu}} {\om^-_{x,\nu}(t)}dx.
\ees
The rest of the proof follows that of Proposition \ref{liminf_part_c} with $\om^-_{x,\nu}$ in place of $\om_{x,\nu}$ and taking into consideration of the fact that $\om^{-}_{x,\nu}(t)=\om^-(x+t\nu)$ (see Remark 3.109 in \cite{ambrosio2000functions}).
\end{proof}

The next lemma is the $SBV$ version of Lemma \ref{cont_arb_dic_slice}. We recall that $I(t_0,t):=(t_0-t,t_0+t)$.
\begin{proposition}\label{SBV_Simple_jump_approx}
let $\tau\in(0,1/4)$ be given, and let $\om\in SBV(\Om)\cap L^\infty(\Om)$ be nonnegative. Then for $\hnmo$ a.e. $x_0\in S_\om$ a point of density one, there exists $r_0:=r_0(x_0)>0$ such that for each $0<r<r_0$ there exist $t_0\in(2\tau r, 4\tau r)$ and $0<t_{0,r}=t_{0,r}(t_0,\tau,x_0,r)<t_0$ such that
\begin{multline}\label{upper_sup_est}
\sup_{0<t\leq t_{0,r}}\frac{1}{\abs{{I(t_0,t)}}}\int_{{I(t_0,t)}}\int_{Q_{\nu_{S_\om}}^\pm(x_0,r)\cap T_{x_0,\nu_{S_\om}}(\pm l)}{\om(x)}d\hnmo(x) d l\\
\leq \int_{S_\om\cap Q_{\nu_{S_\om}}^\pm(x_0,r)}\om^\pm(x)\,d\hnmo+  O(\tau) r^{N-1}.
\end{multline}
\end{proposition}

\begin{proof}
For simplicity of notation, in what follows we abbreviate $Q_{\nu_{S_\om}}(x_0,r)$ as $Q(x_0,r)$ and $T_{x_0,\nu_{S_\om}}$ as $T_{x_0}$.\\\\
Since $\hnmo(S_\om)<\infty$, and so $\mu:=\hnmo\lfloor S_\om$ is a nonnegative radon measure, and since $\om^-\in L^1(\Om,\mu)$, it follows that for $\hnmo$ a.e. $x_0\in S_\om$
\be\label{densityoneomneg}
\lim_{r\to 0}\fint_{Q(x_0,r)\cap S_\om}\abs{\om^-(x)-\om^-(x_0)}d\hnmo (x)=0.
\ee
Choose one such $x_0\in S_\om$, also a point of density $1$ of $S_\om$, and let $\tau>0$ be given. Select $r_1>0$ such that for all $0<r<r_1$,
\be\label{length_a_equal_j}
\frac{1}{1+\tau^2}\leq\frac{\mathcal H^{N-1}(S_\om\cap Q(x_0,r))}{r^{N-1}}\leq 1+\tau^2.
\ee
Let $0<r_2<r_1$ be such that, in view of \eqref{densityoneomneg},
\be\label{phase_two}
\int_{Q(x_0,r)\cap S_\om} \abs{ \om^-(x)- \om^-(x_0)}d\hnmo\leq \tau^2 r^{N-1}
\ee
for all $0<r<r_2$, and we observe that
\begin{align}\label{phase_three}
\begin{split}
 \om^-(x_0){\hnmo\fmp{Q(x_0,r)\cap T_{x_0}(-t_0)}}& = \om^-(x_0) \,r^{N-1}\\
 &\leq (1+\tau^2)\,\om^-(x_0)\hnmo\fmp{{Q(x_0,r)\cap S_\om}}.
\end{split}
\end{align}
Since by Theorem \ref{fine_properties_BV} 
\bes
\lim_{r\to 0}\fint_{Q^-(x_0,r)}\abs{ \om(x)- \om^-(x_0)}dx=0,
\ees
we may choose $0<r_3<r_2$ such that 
\bes
\fint_{Q^-(x_0,r)}\abs{ \om(x)- \om^-(x_0)}dx\leq \tau^2,
\ees
for all $0<r<r_3$, and so, since $3.5\tau r<r$, we have
\bes
\int_{2.5\tau r}^{3.5\tau r} \int_{Q^-(x_0,r)\cap T_{x_0}(-t)}\abs{ \om(x)- \om^-(x_0)}d\hnmo(x) dt\leq \int_{Q^-(x_0,r)}\abs{ \om(x)- \om^-(x_0)}dx\leq \tau^2 r^N.
\ees
There exists a set $A\subset(2.5\tau r, 3.5\tau r)$ with positive 1 dimensional Lebesgue measure such that for every $t\in A$, 
\be\label{phase_four}
\int_{Q^-(x_0,r)\cap T_{x_0}(-t)}\abs{ \om(x)- \om^-(x_0)}d\hnmo(x) \leq \frac{\tau^2 r^{N}}{\tau r}=\tau r^{N-1}.
\ee
and choose $t_0\in A$ a Lebesgue point for 
\bes
l\in(-r/2,r/2)\longmapsto \int_{Q^{-}(x_0,r)\cap T_{x_0}(l)}\om\,d\hnmo(x)
\ees
so that 
\be\label{lebesgue_point_line}
\lim_{t\to 0}\frac{1}{\abs{{I(t_0,t)}}}\int_{{I(t_0,t)}}\int_{Q^-(x_0,r)\cap T_{x_0}(-l)}{ \om(x)}d\hnmo(x) dl= \int_{Q^-(x_0,r)\cap T_{x_0}(-t_0)}{ \om(x)}d\hnmo(x) .
\ee
Hence, there exists $t_{0,r}>0$, depending on $t_0$, $\tau$, $r$, and $x_0$, such that $I(t_0,t_{0,r})\subset (2.5\tau r,3.5\tau r)$ and
\begin{multline}\label{phase_one}
\sup_{0<t\leq t_{0,r}}\frac{1}{\abs{{I(t_0,t)}}}\int_{{I(t_0,t)}}\int_{Q^-(x_0,r)\cap T_{x_0}(-l)}{ \om(x)}d\hnmo(x) dl\\
\leq \int_{Q^-(x_0,r)\cap T_{x_0}(-t_0)}{ \om(x)}d\hnmo + \tau r^{N-1}.
\end{multline}
%
In view of \eqref{phase_one}, \eqref{phase_four}, \eqref{phase_three}, and \eqref{phase_two}, in this order, we have that for every $0<r<r_3$ there exist $t_0\in(2.5\tau r,3.5\tau r)$ and $0<t_{0,r}<t_0$ such that
\begin{align*}
\sup_{0<t\leq t_{0,r}}&\frac{1}{\abs{{I(t_0,t)}}}\int_{{I(t_0,t)}}\int_{Q^-(x_0,r)\cap T_{x_0}(-l)}{ \om(x)}d\hnmo(x) dl \\
\leq&\int_{Q^-(x_0,r)\cap T_{x_0}(-t_0)}{\om(x)}d\hnmo + \tau r^{N-1} \\
\leq&\int_{Q^-(x_0,r)\cap T_{x_0}(-t_0)}\abs{\om(x)-\om^-(x_0)}d\hnmo \\
&+\om^-(x_0)\hnmo\fmp{Q^-(x_0,r)\cap T_{x_0}(-t_0)}+ \tau r^{N-1} \\
\leq& O(\tau) r^{N-1}+(1+\tau^2)\,\om^-(x_0)\hnmo\fmp{{Q(x_0,r)\cap S_\om}}\\
\leq& O(\tau) r^{N-1} +(1+\tau^2)\int_{Q(x_0,r)\cap S_\om} \om^-(x)d\hnmo.
\end{align*}
Since $\om\in L^\infty(\Om)$, we have $\om^-\in L^\infty(S_\om)$ and thus, invoking \eqref{length_a_equal_j},
\bes
\tau^2\int_{Q(x_0,r)\cap S_\om} \om^-(x)d\hnmo\leq O(\tau) \norm{\om}_{L^\infty} \hnmo\fmp{Q(x_0,r)\cap S_\om}\leq O(\tau) r^{N-1},
\ees
and we deduce the $\om^-$ version of \eqref{upper_sup_est}.\\\\
Similarly, we may refine $t_0$, $r_0>0$, and $0<t_{0,r}<t_0$ such that
\bes
\sup_{0<t\leq t_{0,r}}\frac{1}{\abs{{I(t_0,t)}}}\int_{{I(t_0,t)}}\int_{Q^+(x_0,r)\cap T_{x_0}(l)}{ \om(x)}d\hnmo dl\leq  \int_{Q(x_0,r)\cap S_\om} \om^+(x)d\hnmo+O(\tau) r^{N-1}.
\ees
\end{proof}

\begin{proposition}\label{jump_ready_coro_limsup}
Let $\om\in SBV(\Om)\cap L^\infty(\Om)$ be nonnegative and let $\tau\in(0,1/4)$ be given. Then, there exist a set $S\subset S_\om$ and a countable family of disjoint cubes $\mathcal F=\flp{Q_{\nu_{S_\om}}(x_n, r_n)}_{n=1}^{\infty}$, with $r_n< \tau$, such that the following hold:
\begin{enumerate}[1.]
\item
$\hnmo(S_\om\setminus S)<\tau$ and $S\subset \bigcup_{n=1}^\infty Q_{\nu_{S_\om}}(x_n, r_n)$;
\item
\be\label{disjoint_col}
\operatorname{dist}(Q_{\nu_{S_\om}}(x_n, r_n),Q_{\nu_{S_\om}}(x_m,r_m))>0
\ee
for $n\neq m$;
\item
\be\label{jump_coro_finitesum}
\sum_{n=1}^\infty r_n^{N-1}\leq 4\hnmo(S_\om);
\ee
\item
$S\cap Q_{\nu_{S_\om}}(x_n, r_n)\subset R_{\tau/2,\nu_{S_\om}}(x_n,r_n)$;
\item
for each $n\in\mathbb N$, there exists $t_{n}\in(2.5\tau r_n,3.5\tau r_n)$ and $0<t_{x_n,r_n}<t_n$, depending on $\tau$, $r_n$, and $x_n$, such that 
\bes
T_{x_n,\nu_{S_\om}}(-t_n\pm t_{x_n,r_n})\subset Q_{\nu_{S_\om}}^-(x_n,r_n)\setminus R_{\tau/2,\nu_{S_\om}}(x_n,r_n)
\ees
and
\begin{multline}\label{upper_sup_ready_limsup_jump}
\sup_{0<t\leq t_{x_n,r_n}}\frac{1}{\abs{I(t_n,t)}}\int_{I(t_n,t)}\int_{Q_{\nu_{S_\om}}(x_n, r_n)\cap T_{x_n,\nu_\Gamma}( l)}{\om(x)}d\hnmo d l\\
\leq \int_{S\cap Q_{\nu_{S_\om}}(x_n, r_n)}\om^-\,d\hnmo+  C\tau r_n^{N-1},
\end{multline}
where $I(t_n,t):=(-t_n-t, -t_n+t)$.
\end{enumerate}
\end{proposition}

\begin{proof}
The proof of this proposition uses the same arguments of the proof of Proposition \ref{cont_control_jump} and Proposition \ref{cont_ready_coro_limsup} where we apply Lemma \ref{SBV_Simple_jump_approx} in place of Lemma \ref{cont_arb_dic_slice}.
\end{proof}

\begin{proposition}\label{limsup_n_j}\emph{($\Gamma$-$\limsup$)}
For $\om\in\mathcal W(\Om)\cap SBV(\Om)\cap L^\infty(\Om)$ and $u\in L^1_\om(\Om)\cap L^\infty(\Om)$, let
\begin{align*}
E_\om^+(u):=\inf&\flp{\limsup_{\e\to 0} E_{\om,\e}(u_\e,v_\e):\right.\\
&\left.\,\,\,(u_\e,v_\e)\in W^{1,2}_\om(\Om)\times W^{1,2}(\Om), u_\e\to u\text{ in }L^1_\om,\, v_\e\to1\text{ in }L^1,\,0\leq v_\e\leq 1}.
\end{align*}
We have
\be\label{jump_upper_bdd_need}
E_\om^+(u)\leq E_\om(u).
\ee
\end{proposition}
\begin{proof}
\underline{Step 1:} Assume $\hnmo((S_\om\setminus S_u)\cup(S_u\setminus S_\om))=0$, i.e., $S_\om$ and $S_u$ coincide $\hnmo$ a.e.\\\\
If $E_\om(u)=\infty$ then there is nothing to prove. If $E_\om(u)<+\infty$ then by Lemma \ref{compact_energy} we have that $u\in GSBV_\om(\Om)$ and $\hnmo(S_u)<+\infty$.\\\\
Fix $\tau\in(0,2/21)$. Applying Proposition \ref{jump_ready_coro_limsup} to $\om$ we obtain a set $S_\tau\subset S_\om$, a countable collection of mutually disjoint cubes $\mathcal F_\tau=\seqn{Q_{\nu_{S_\om}}(x_n, r_n)}$, and corresponding 
\be\label{whereistn}
t_n\in(2.5\tau r_n,3.5\tau r_n)
\ee
and $t_{x_n,r_n}$ for which \eqref{upper_sup_ready_limsup_jump} holds. Extract a finite collection $\mathcal T_\tau=\flp{{Q_{\nu_{S_\om}}(x_n, r_n)}}_{n=1}^{M_\tau}$ from $\mathcal F_\tau$ with $M_\tau>0$ large enough such that 
\be\label{ext_fint_part}
\hnmo\fmp{S_\tau\setminus \bigcup_{n=1}^{M_\tau} Q_{\nu_{S_\om}}(x_n, r_n)}<\tau,
\ee
and we define
\be\label{what_left_F_jump}
F_\tau:=S_\tau\cap \fmp{\bigcup_{n=1}^{M_\tau} Q_{\nu_{S_\om}}(x_n, r_n)}.
\ee
Let $U_n$ be the part of $Q_{\nu_{S_\om}}(x_n, r_n)$ which lies between $T_{x_n,\nu_{S_\om}}(\pm t_n)$, $U_n^+$ be the part above $T_{x_n,\nu_{S_\om}}(t_n)$, and $U_n^-$ be the part below $T_{x_n,\nu_{S_\om}}(-t_n)$.\\\\
We claim that if $x\in U_n$, 
\be\label{jump_right_place}
x+2\,{\operatorname{dist}(x,T_{x_n,\nu_{S_\om}}(t_n))\nu_{S_\om}(x_n)}\in U_n^+.
\ee
Note that 
\bes
\dist\fsp{x, T_{x_n,\nu_{S_\om}}(t_n)} = t_n - \fsp{x-\mathbb P_{x_n,\nu_{S_\om}}(x)}\nu_{S_\om}(x_n),
\ees
and since $x\in U_n$, we have that
\bes
\fsp{x-\mathbb P_{x_n,\nu_{S_\om}}(x)}\nu_{S_\om}(x_n)\in (-t_n, t_n)
\ees
and 
\bes
t_n\leq 2\operatorname{dist}\fsp{x,T_{x_n,\nu_{S_\om}}(t_n)}+ \fsp{x-\mathbb P_{x_n, \nu_{S_\om}}(x)}\nu_{S_\om}(x_n) \leq  3t_n\leq 10.5\tau r_n<\frac12 r_n.
\ees
Hence, following a similar computation in \eqref{big_proj_com}, we deduce \eqref{jump_right_place}.\\\\
Moreover, according to \eqref{whereistn} and the definition of $R_{\tau/2,\nu_{S_\om}}(x_n,r_n)$, we have that 
\be
\fsp{U_n^+\cup U_n^-}\cap R_{\tau/2,\nu_{S_\om}}(x_n,r_n)=\emptyset.
\ee
We define $\bar u_\tau$ as follows (see Figure \ref{figue_approx_u_simplex_jump}):
\be\label{approx_u_simplex_jump}
\bar u_\tau(x):=
\begin{cases}
u(x) & \text{ if }x\in U_n^+\cup U_n^-,\\
u\fsp{x+2{\operatorname{dist}(x,T_{x_n,\nu_{S_\om}}(t_n))\nu_{S_\om}(x_n)}}&\text{ if }x\in U_n,
\end{cases}
\ee
and 
\bes
\bar u_\tau(x):=u(x)\text{ if }x\in \Om\setminus \fsp{\bigcup_{n=1}^{M_\tau} Q_{\nu_{S_\om}}(x_n,r_n)}.
\ees
\begin{figure}[h]
\includegraphics[width=0.8\textwidth]{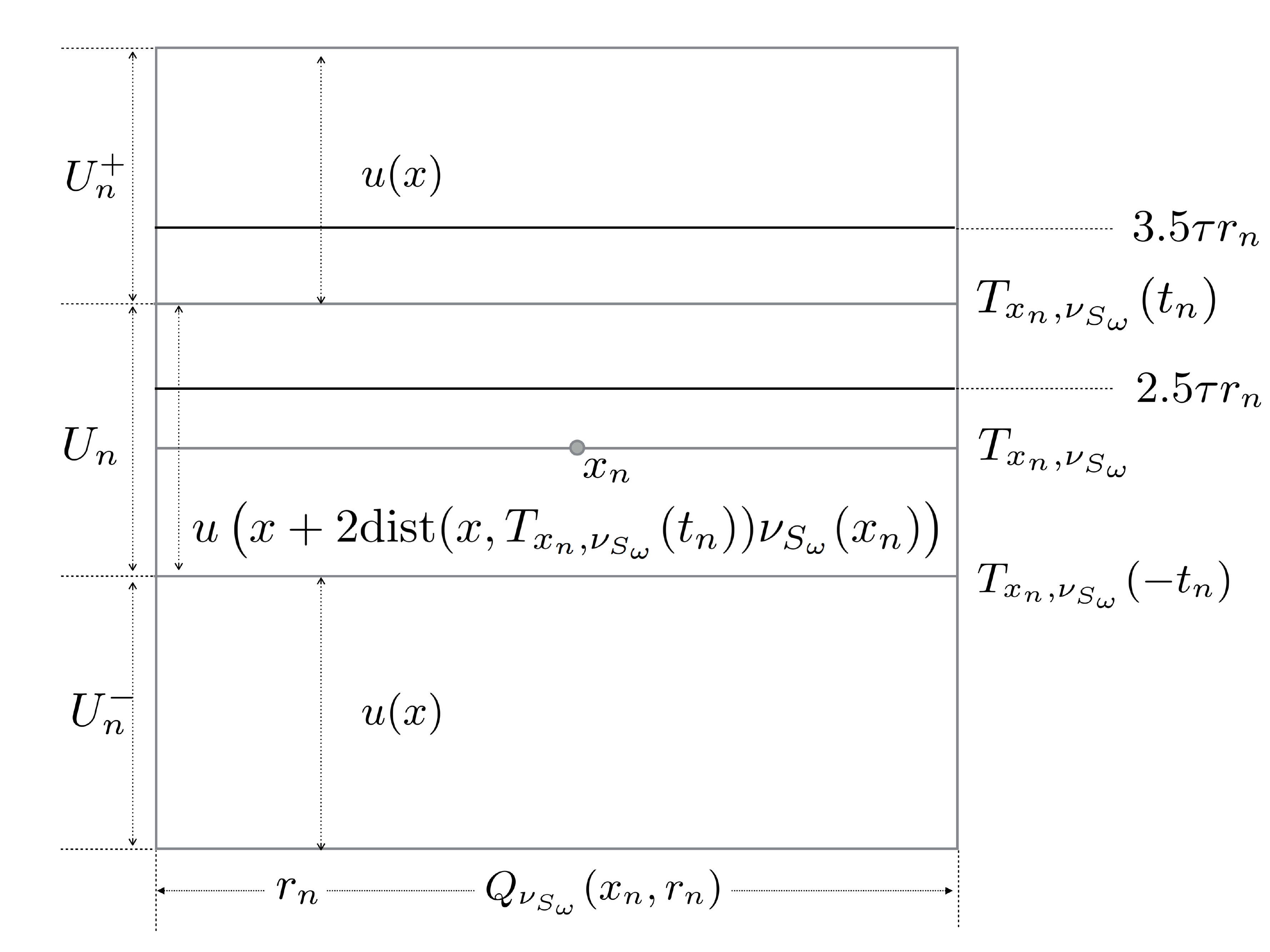}
\caption{Construction of $\bar u_\tau(x)$ in (\ref{approx_u_simplex_jump})}\label{figue_approx_u_simplex_jump}
\end{figure}

Note that the jump set of $\bar u_\tau$ is contained by
\begin{enumerate}[1.]
\item
\be\label{major_part_m_jump}
\bigcup_{n=1}^{M_\tau}\fmp{T_{x_n,\nu_{S_\om}}(-t_n)\cap Q_{\nu_{S_\om}}(x_n,r_n)};
\ee
\item
$$\bigcup_{n=1}^{M_\tau}\partial \fsp{Q_{\nu_{S_\om}}(x_n,r_n)}\cap \overline{U_n};$$
\item
$S_u\setminus F_\tau$, where $F_\tau$ is defined in \eqref{what_left_F_jump}.
\end{enumerate}
The construction of $\seqe{u_\e}\subset W_\om^{1,2}(\Om)$ and $\seqe{v_\e}\subset W^{1,2}(\Om)$ satisfying \eqref{jump_upper_bdd_need} is same as in the proof of Proposition \ref{limsup_n_c}, using \eqref{approx_u_simplex_jump} instead of \eqref{approx_u_simplex_jump_cont}, and at \eqref{cont_key_step} we apply \eqref{upper_sup_ready_limsup_jump} instead of \eqref{upper_sup_ready_limsup_cont}.\\\\
\underline{Step 2:} Suppose that $\hnmo((S_\om\setminus S_u)\cup(S_u\setminus S_\om))>0$. Note that we are only interested in the part $S_u\setminus S_\om$ but not $S_\om\setminus S_u$, because we only need to recover $S_u$.\\\\
We first apply Proposition \ref{cont_ready_coro_limsup} on $S_u$ to obtain a countable family of disjoint cubes $\mathcal F=\seqn{Q_{\nu_{S_u}}(x_n,r_n)}$ such that \eqref{cont_sdfsfsdfs}-\eqref{cont_coro_finitesum} hold. Furthermore, extract a finite collection $\mathcal T_\tau$ from $\mathcal F$ such that \eqref{ext_fint_part} holds. \\\\
We define $\bar u_\tau$ inside each $Q_{\nu_{S_u}}(x_n,r_n)\in \mathcal T_\tau$ as follows (see Figure \ref{figue_approx_u_mix}):
\begin{enumerate}[1.]
\item
if $x_n\in \overline S_\om$, we apply Proposition \ref{SBV_Simple_jump_approx} to obtain item 5 in Proposition \ref{jump_ready_coro_limsup} for this $Q_{\nu_{S_u}}(x_n,r_n)$, and we define $\bar u_\tau$ in this cube in the way of \eqref{approx_u_simplex_jump};
\item
if $x_n\in S_u\setminus\overline S_\om$, we apply Lemma \ref{cont_arb_dic_slice} to obtain item 5 in Proposition \ref{cont_ready_coro_limsup} for this $Q_{\nu_{S_u}}(x_n,r_n)$, and we define $\bar u_\tau$ in this cube in the way of \eqref{approx_u_simplex_jump_cont}.
\end{enumerate}
For points $x$ outside $\mathcal T_\tau$, we define $\bar u_\tau(x):=u(x)$.
\begin{figure}[h]
\includegraphics[width=0.8\textwidth]{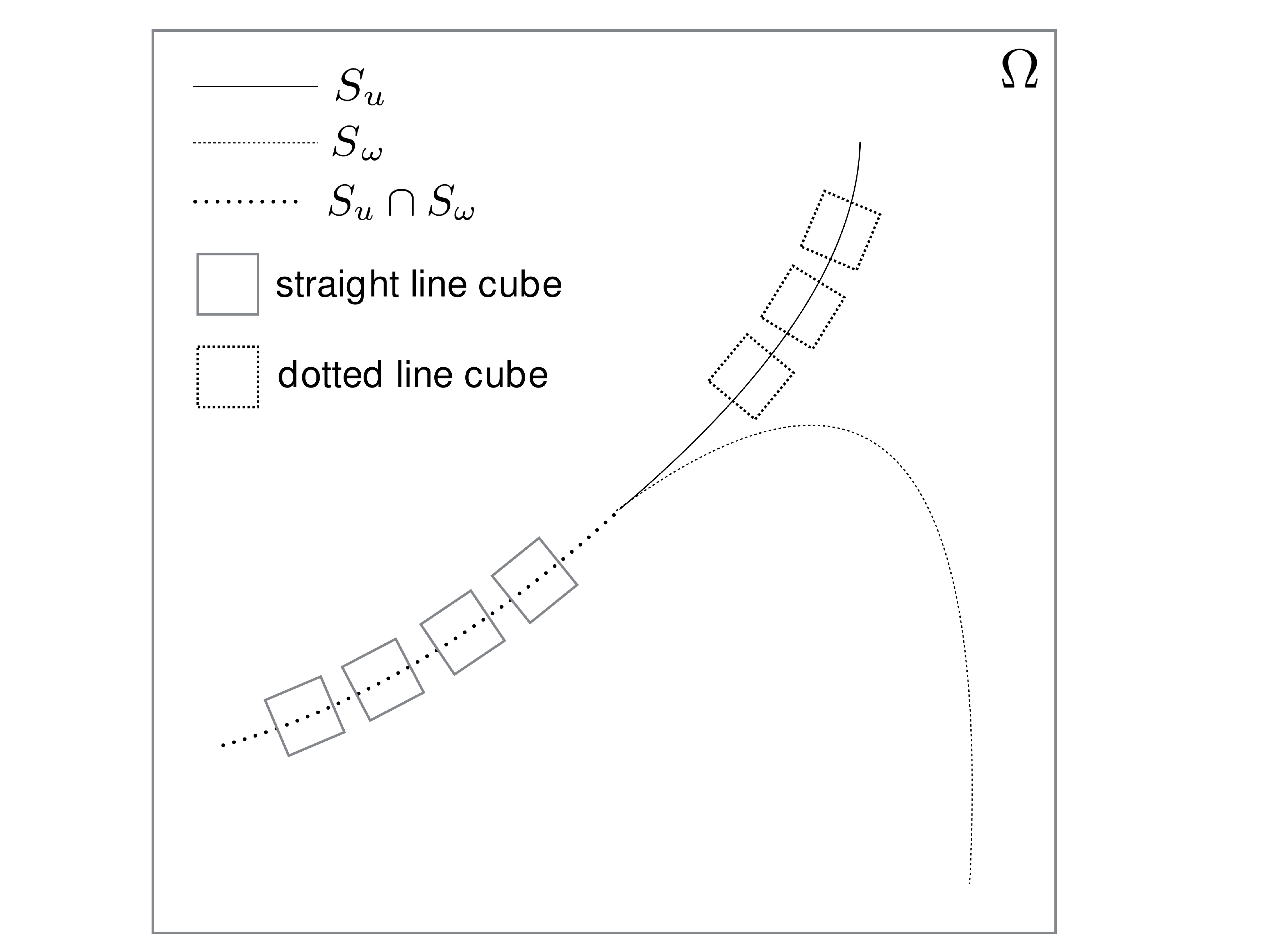}
\caption{Applying (\ref{approx_u_simplex_jump}) in dotted line cube and (\ref{approx_u_simplex_jump_cont}) in straight line cube.}\label{figue_approx_u_mix}
\end{figure}

Reasoning as in Proposition \ref{limsup_n_c} and Proposition \ref{limsup_n_j}, we conclude \eqref{jump_upper_bdd_need}.
\end{proof}

\begin{proof}[Proof of Theorem \ref{AT_n_case}]
The $\liminf$ inequality follows from Proposition \ref{liminf_part}. On the other hand, for any given $u\in GSBV_\om$ such that $E_\om(u)<\infty$, we have, by Lebesgue Monotone Convergence Theorem,
\bes
E_\om(u)=\lim_{K\to\infty} E_\om(K\wedge u\vee -K),
\ees
and a diagonal argument, together with Proposition \ref{limsup_n_j}, yields the $\limsup$ inequality for $u$.
\end{proof}

\setcounter{equation}{0}
\setcounter{theorem}{0}
\appendix
\section*{Appendix}
\renewcommand{\thetheorem}{A.\arabic{theorem}}
\renewcommand{\theequation}{A.\arabic{equation}}

\begin{define}[\cite{ambrosio2004topics}, Definition $4.4.9$]\label{hausdorff_distance}
Let $\mathcal X$ be a metric space. We denote by $\mathcal C_{\mathcal X}$ the family of all nonempty closed subsets of $X$. Then
\bes
d_{\mathcal H}(C,D):=\min\flp{1,h(C,D)},\,\,\,C,D\in \mathcal C_{\mathcal X},
\ees
where 
\bes
h(C,D):=\inf\flp{\delta\in[0,+\infty]:\,C\subset D_\delta\text{ and }D\subset C_\delta},
\ees
is a metric on $\mathcal C_{\mathcal X}$, and is called the \emph{Hausdorff} distance between the set $C$ and $D$ (see \eqref{flatten_set} for definition of $D_\delta$ and $C_\delta$).
\end{define}
Consider $\mathcal X$ to be the interval $(0,1)$ with the Euclidian distance. We remark that for two intervals $[a_1,b_1]$ and $[a_2,b_2]$ in $(0,1)$,
\be\label{jiushizhegejulia}
d_{\mathcal H}([a_1,b_1],\,[a_2,b_2]) =\min\flp{1,\, \max\flp{\abs{a_1-a_2},\abs{b_1-b_2}}}.
\ee
Indeed, the $\delta$-neighborhood of $[a_1,b_1]$ is $[a_1-\delta,b_1+\delta]$, and contains $[a_2,b_2]$ if and only if
$$\delta\ge \max\flp{a_1-a_2, b_2-b_1}.$$
Similarly, the $\delta$-neighborhood of $[a_2,b_2]$ contains $[a_1,b_1]$ if and only if
$$\delta\ge \max\flp{a_2-a_1, b_1-b_2},$$
and we conclude \eqref{jiushizhegejulia}. 
%

\begin{lemma}\label{hausdorff_conv_int}
Let $I_n:=[a_n,b_n]\subset (-1,1)$. Then, up to the extraction of a subsequence,
\bes
I_n\wtoh I_\infty\subset (-1,1),
\ees
where $I_\infty$ is connected and closed in $(-1,1)$, and 
\bes
\mathcal L^1(I_\infty) = \limn \mathcal L^1(I_n).
\ees
Moreover, for arbitrary $K\subset\subset I_\infty$, $K$ must be contained in $I_n$ for $n$ large enough.
\end{lemma}
\begin{proof}
Because $I_n\subset (-1,1)$, we have that $\seqn{a_n}$ and $\seqn{b_n}$ are bounded and so, up to the extraction of a subsequence, there exist
\be\label{liangbiandelai}
a_\infty:=\limn a_n\text{ and } b_\infty: = \limn b_n,
\ee
where $-1\leq a_\infty\leq b_\infty\leq 1$. We define $I_\infty:=[a_\infty,b_\infty]$ if $-1<a_\infty\leq b_\infty<1$, $I_\infty:=(-1,b_\infty]$ if $a_\infty=-1$, and $I_\infty:=[a_\infty, 1)$ if $b_\infty=1$. Hence $I_\infty$ is connected and closed in $(-1,1)$ (in the case in which $a_\infty = b_\infty =-1$, or $a_\infty = b_\infty =1$, we have $I_\infty =\varnothing$ and it is still closed in $(-1,1)$).\\\\
Hence 
\bes
\limn d_{\mathcal H}(I_n,\, I_\infty) =\limn \max\flp{\abs{a_n-a_\infty},\,\abs{b_n-b_\infty}}=0,
\ees
and we have for $I_\infty\neq \varnothing$,
\bes
\mathcal L^1(I_\infty)=b_\infty-a_\infty=\limn(b_n-a_n)=\limn \mathcal L^1(I_n),
\ees
as desired. \\\\
Next, if $K\subset \subset I_\infty$ then $K\subset(\alpha,\beta)$ for some $\alpha$, $\beta$ such that $a_\infty<\alpha<\beta< b_\infty$. By \eqref{liangbiandelai} choose $N$ large enough such that for all $n\geq N$, 
\bes
a_n<\alpha<\beta<b_n,
\ees
so that $K\subset I_n$ for all $n\geq N$.
\end{proof}

\begin{lemma}\label{osl_small_contral_proof}
Let $\flp{v_\e}_{\e> 0}\subset W^{1,2}(I)$ be such that $0\leq v_\e\leq 1$, $v_\e\to1$ in $L^1(I)$ and pointwise a.e., and
\be\label{up_lmt_goes1}
\limsup_{\e\to0}\int_{I}\left[\frac{\e}{2}\abs{ v_\e'}^2+\frac{1}{2\e}(v_\e-1)^2\right]\,dx <\infty.
\ee
Then for arbitrary $0<\eta<1$ there exists an open set $H_\eta\subset I$ satisfying: 
\begin{enumerate}[1.]
\item
the set $I\setminus H_\eta$ is a collection of finitely many points in $I$;
\item
for every set $K$ compactly contained in $H_\eta$, we have $K\subset B_\e^\eta$ for $\e>0$ small enough, where 
\be\label{1_d_level_set_proof}
B_\e^\eta:=\flp{x\in I:\,v^2_\e(x)\geq \eta}.
\ee
\end{enumerate}
\end{lemma}

\begin{proof}
Choose a constant $M>0$ such that
\bes
M\geq \limsup_{\e\to0}\int_{I}\left[\frac{\e}{2}\abs{ v_\e'}^2+\frac{1}{2\e}(v_\e-1)^2\right]\,dx\geq \limsup_{\e\to 0}\int_I \abs{v_\e'}\abs{1-v_\e}\,dx =\limsup_{\e\to 0} \frac12\int_I\abs{c_\e'}dx,
\ees
where $c_\e(x):=(1-v_\e(x))^2$. Note that by \eqref{up_lmt_goes1}, $c_\e\to 0$ in $L^1(I)$. Fix $\sigma$, $\delta$ with
\be\label{given_delta}
0<\sigma<\delta<1.
\ee
By the co-area formula we have, for $0<\e<\e_0$ with $\e_0$ sufficiently small,
\bes
2M+1\geq \int_I\abs{c_\e'(x)}dx = \int_{-\infty}^\infty \mathcal H^0(\flp{x:\, c_\e(x)=t})\,dt\geq \int_\sigma^\delta \mathcal H^0(\flp{x:\, c_\e(x)=t})\,dt.
\ees
Hence, for each $\e>0$ there exist $\delta_\e\in (\sigma,\delta)$ such that 
\be\label{at_most_finite_com}
\frac{2M+1}{\delta-\sigma}\geq  \mathcal H^0(\flp{x:\, c_\e(x)=\delta_\e}).
\ee
Define, for a fixed $r>0$,
\bes
A_\e^r:=\flp{x\in I:\,c_\e(x)\leq r}.
\ees
Since $v_\e\in W^{1,2}(I)$, $v_\e$ is continuous and so is $c_\e$, therefore $A_\e^{\delta_\e}$ is closed and has at most $(2M+1)/(\delta-\sigma)+1$ connected components because of \eqref{at_most_finite_com} and in view of the continuity of $c_\e$. Note that the number $(2M+1)/(\delta-\sigma)$ does not depend on $\e>0$. \\\\
For $\e\in(0,\e_0)$ and $k\in\mathbb N$ depending only on $\delta-\sigma$ and $M$, we have
\begin{enumerate}[1.]
\item
$A_\e^{\delta_\e} = \bigcup_{i=1}^k I_\e^i$, where each $I_\e^i$ is a closed interval or $\varnothing$;
\item
$\text{for all } i<j$, $\max\flp{x:\,x\in I_\e^i}<\min \flp{x:\,x\in I_\e^j}$.
\end{enumerate}
By Lemma \ref{hausdorff_conv_int}, up to the extraction of a subsequence, for each $i\in \flp{1,2,\ldots, k}$ let $I_0^i$ be the \emph{Hausdorff} limit of the $I_\e^i$ as $\e\to 0$, i.e., $I_\e^i\wtoh I_0^i$, with $I_0^i$ is connected and closed in $I$, and for all $i<j$, $\max I_0^i\leq \min I_0^j$.\\\\
Set 
\be\label{osl_small_contral_proof_limiandegongshi}
T_\delta:=\bigcup_{i=1}^k(I_0^i)^\circ\text{ and }\,T_{\delta,\e}:=\bigcup_{i=1}^k(I_\e^i)^\circ,
\ee
where by $(\cdot)^\circ$ we denote the interior of a set. Since
\bes
I\setminus A_\e^{\delta_\e}\subset\flp{x\in I:\,\,c_\e(x)\geq \sigma}
\ees
and $c_\e\to 0$ in $L^1(I)$, by Chebyshev's inequality we have 
\bes
\lim_{\e\to 0}\mathcal L^1(T_{\delta,\e})=\lime \mathcal L^1(A_\e)=2.
\ees
Moreover, since $T_{\delta,\e}\wtoh T_\delta$, by Lemma \ref{hausdorff_conv_int} we have
\bes
\mathcal L^1(T_{\delta}) = \sum_{i=1}^k \mathcal L^1(I_0^i)^\circ=\sum_{i=1}^k\lime \mathcal L^1(I_\e^i)^\circ=\lime \sum_{i=1}^k \mathcal L^1(I_\e^i)^\circ=\lim_{\e\to 0}\mathcal L^1(T_{\delta,\e})=2.
\ees
Thus $\abs{I\setminus T_\delta}=0$. Moreover, since $T_\delta$ has at most $k$ connected components, $I\setminus T_\delta$ is a finite collection of points in $I$.\\\\
Next, let $K\subset\subset T_\delta$ be a compact subset. We claim that $K$ must be contained in $A_\e^{\delta_\e}$ for $\e>0$ small enough. Recall $I_0^i$ and $I_\e^i$ from \eqref{osl_small_contral_proof_limiandegongshi}. Define $K_i:=K\cap (I_0^i)^\circ$ for $i=1,\ldots, k$. Then $K_i\subset \subset (I_o^i)^\circ$ for each $i$, and so by Lemma \ref{hausdorff_conv_int} there exists $\e_i>0$ such that for all $0<\e<\e_i$, $K_i\subset I_\e^i$. Define 
\bes
\e':=\min_{i\in\flp{1,\ldots,k}}\flp{\e_i}.
\ees
For $0<\e<\e'$ we have $K_i\subset I^i_\e$, and so
\bes
K=\bigcup_{i=1}^k K_i\subset \bigcup_{i=1}^k I_\e^i=A_\e^{\delta_\e}.
\ees
Finally, given $\eta\in(0,1)$, set $\delta:=\fsp{1-\sqrt{\eta}}^2$ with $H_\eta:=T_{(1-\sqrt{\eta})^2}$ and $B_\e^\eta:=A_\e^{(1-\sqrt{\eta})^2}$, and properties 1 and 2 are satisfied.
\end{proof}

\section*{Acknowledgements}
The authors wish to acknowledge the Center for Nonlinear Analysis where this work was carried out. The research of both authors was partially funded by the National Science Foundation under Grant No. DMS - 1411646.

\bibliographystyle{abbrv}
\bibliography{PV_WG_MS}{}

\end{document}